\documentclass[11pt]{amsart}
\usepackage{amsopn}
\usepackage{amssymb, amscd}
\usepackage{multirow}
\usepackage{graphicx, graphics, epsfig}
\usepackage{faktor} 
\usepackage{enumerate}
\usepackage{tikz}
\usepackage{pgfplots}
\usepackage{hyperref}
\usepackage{color}

\newcommand{\nc}{\newcommand}

\nc{\fg}{\mathfrak{f} } \nc{\vg}{\mathfrak{v} } \nc{\wg}{\mathfrak{w} }
\nc{\zg}{\mathfrak{z} } \nc{\ngo}{\mathfrak{n} } \nc{\kg}{\mathfrak{k} }
\nc{\mg}{\mathfrak{m} } \nc{\bg}{\mathfrak{b} } \nc{\ggo}{\mathfrak{g} } \nc{\eg}{\mathfrak{e} }
\nc{\ggob}{\overline{\mathfrak{g}} } \nc{\sog}{\mathfrak{so} }
\nc{\sug}{\mathfrak{su} } \nc{\spg}{\mathfrak{sp} } \nc{\slg}{\mathfrak{sl} }
\nc{\glg}{\mathfrak{gl} } \nc{\cg}{\mathfrak{c} } \nc{\rg}{\mathfrak{r} }
\nc{\hg}{\mathfrak{h} } \nc{\tg}{\mathfrak{t} } \nc{\ug}{\mathfrak{u} }
\nc{\dg}{\mathfrak{d} } \nc{\ag}{\mathfrak{a} } \nc{\pg}{\mathfrak{p} }
\nc{\sg}{\mathfrak{s} } \nc{\affg}{\mathfrak{aff} } \nc{\qg}{\mathfrak{q} } \nc{\lgo}{\mathfrak{l} }

\nc{\pca}{\mathcal{P}} \nc{\nca}{\mathcal{N}} \nc{\lca}{\mathcal{L}}
\nc{\oca}{\mathcal{O}} \nc{\mca}{\mathcal{M}} \nc{\tca}{\mathcal{T}}
\nc{\aca}{\mathcal{A}} \nc{\cca}{\mathcal{C}} \nc{\gca}{\mathcal{G}}
\nc{\sca}{\mathcal{S}} \nc{\hca}{\mathcal{H}} \nc{\bca}{\mathcal{B}}
\nc{\dca}{\mathcal{D}} \nc{\eca}{\mathcal{E}} \nc{\wca}{\mathcal{W}} \nc{\ica}{\mathcal{I}} \nc{\kca}{\mathcal{K}}

\nc{\vp}{\varphi} \nc{\ddt}{\tfrac{d}{dt}} \nc{\dsdt}{\tfrac{d^2}{dt^2}} \nc{\dds}{\frac{d}{ds}}
\nc{\dpar}{\frac{\partial}{\partial t}} \nc{\im}{\mathrm{i}}

\nc{\SO}{\mathrm{SO}} \nc{\Spe}{\mathrm{Sp}} \nc{\Sl}{\mathrm{SL}}
\nc{\SU}{\mathrm{SU}} \nc{\Or}{\mathrm{O}} \nc{\U}{\mathrm{U}} \nc{\Gl}{\mathrm{GL}}
\nc{\Se}{\mathrm{S}} \nc{\Cl}{\mathrm{Cl}} \nc{\Spin}{\mathrm{Spin}}
\nc{\Pin}{\mathrm{Pin}} \nc{\G}{\mathrm{GL}_n(\RR)} \nc{\g}{\mathfrak{gl}_n(\RR)}
\nc{\Eg}{\mathrm{E}} \nc{\Fg}{\mathrm{F}} \nc{\Gg}{\mathrm{G}}

\nc{\RR}{{\Bbb R}} \nc{\HH}{{\Bbb H}} \nc{\CC}{{\Bbb C}} \nc{\ZZ}{{\Bbb Z}}
\nc{\FF}{{\Bbb F}} \nc{\NN}{{\Bbb N}} \nc{\QQ}{{\Bbb Q}} \nc{\PP}{{\Bbb P}} \nc{\OO}{{\Bbb O}}

\nc{\vs}{\vspace{.2cm}} \nc{\vsp}{\vspace{1cm}} \nc{\ip}{\langle\cdot,\cdot\rangle}
\nc{\ipp}{(\cdot,\cdot)} \nc{\la}{\langle} \nc{\ra}{\rangle} \nc{\unm}{\tfrac{1}{2}}
\nc{\unc}{\tfrac{1}{4}} \nc{\und}{\frac{1}{16}} \nc{\no}{\vs\noindent}
\nc{\lam}{\Lambda^2(\RR^n)^*\otimes\RR^n} \nc{\tangz}{{\rm T}^{\rm Zar}}
\nc{\nor}{{\sf n}}  \nc{\mum}{/\!\!/} \nc{\kir}{/\!\!/\!\!/}
\nc{\Ri}{\tfrac{4\Ric_{\mu}}{||\mu||^2}} \nc{\ds}{\displaystyle}
\nc{\ben}{\begin{enumerate}} \nc{\een}{\end{enumerate}} \nc{\f}{\frac}
\nc{\lb}{[\cdot,\cdot]} \nc{\isn}{\tfrac{1}{||v||^2}}
\nc{\gkp}{(\ggo=\kg\oplus\pg,\ip)} \nc{\ukh}{(\ug=\kg\oplus\hg,\ip)}
\nc{\tgkp}{(\tilde{\ggo}=\kg\oplus\pg,\ip)}
\nc{\wt}{\widetilde}
\nc{\iop}{\mathtt{i}} \nc{\jop}{\mathtt{j}} 
\nc{\Hk}{H_{\kil}} \nc{\gk}{g_{\kil}}

\nc{\Hess}{\operatorname{Hess}} \nc{\ad}{\operatorname{ad}}
\nc{\Ad}{\operatorname{Ad}} \nc{\rank}{\operatorname{rk}}
\nc{\Irr}{\operatorname{Irr}} \nc{\End}{\operatorname{End}}
\nc{\Aut}{\operatorname{Aut}} \nc{\Inn}{\operatorname{Inn}}
\nc{\Der}{\operatorname{Der}} \nc{\Ker}{\operatorname{Ker}}
\nc{\Iso}{\operatorname{Iso}} \nc{\Diff}{\operatorname{Diff}}
\nc{\Lie}{\operatorname{L}} \nc{\tr}{\operatorname{tr}} \nc{\dif}{\operatorname{d}}
\nc{\sen}{\operatorname{sen}} \nc{\modu}{\operatorname{mod}}
\nc{\CRic}{\operatorname{PP}} \nc{\Cric}{\operatorname{P}} \nc{\Ricci}{\operatorname{Ric}}
\nc{\sym}{\operatorname{sym}} \nc{\herm}{\operatorname{herm}} \nc{\symac}{\operatorname{sym^{ac}}}
\nc{\symc}{\operatorname{sym^{c}}} \nc{\scalar}{\operatorname{Sc}}
\nc{\grad}{\operatorname{grad}} \nc{\ricci}{\operatorname{Rc}} \nc{\kil}{\operatorname{B}} \nc{\cas}{\operatorname{C}} \nc{\lic}{\operatorname{L}}
\nc{\Nor}{\operatorname{Norm}}  \nc{\ricc}{\operatorname{Rc^{c}}}
\nc{\Ricc}{\operatorname{Ric^{c}}} \nc{\ricac}{\operatorname{Rc^{ac}}}
\nc{\Ricac}{\operatorname{Ric^{ac}}} \nc{\Riem}{\operatorname{Rm}} \nc{\Sec}{\operatorname{Sec}}
\nc{\riccig}{\operatorname{ric^{\gamma}}} \nc{\mm}{\operatorname{m}} \nc{\Mm}{\operatorname{M}}
\nc{\Le}{\operatorname{L}} \nc{\tang}{\operatorname{T}}
\nc{\level}{\operatorname{level}} \nc{\rad}{\operatorname{r}}
\nc{\abel}{\operatorname{ab}} \nc{\CH}{\operatorname{CH}} \nc{\Cone}{{\mathcal C}} \nc{\CCone}{\operatorname{CC}} \nc{\CP}{{\mathcal P}}
\nc{\mcc}{\operatorname{mcc}} \nc{\Adj}{\operatorname{Adj}}
\nc{\Order}{\operatorname{O}}  \nc{\inj}{\operatorname{inj}} \nc{\proy}{\operatorname{pr}}
\nc{\vol}{\operatorname{vol}} \nc{\Diag}{\operatorname{Dg}} \nc{\Diagg}{\operatorname{Diag}}
\nc{\Spec}{\operatorname{Spec}} \nc{\Ima}{\operatorname{Im}} \nc{\Rea}{\operatorname{Re}}
\nc{\spann}{\operatorname{span}} \nc{\Aff}{\operatorname{Aff}} \nc{\E}{\operatorname{E}} \nc{\id}{\operatorname{id}} \nc{\dete}{\operatorname{det}} \nc{\Crit}{\operatorname{Crit}} \nc{\val}{\operatorname{val}}

\theoremstyle{plain}
\newtheorem{theorem}{Theorem}[section]
\newtheorem{proposition}[theorem]{Proposition}
\newtheorem{corollary}[theorem]{Corollary}
\newtheorem{lemma}[theorem]{Lemma}

\theoremstyle{definition}
\newtheorem{definition}[theorem]{Definition}

\theoremstyle{remark}
\newtheorem{remark}[theorem]{Remark}
\newtheorem{example}[theorem]{Example}

\title[Aligned homogeneous spaces]{Ricci curvature and Einstein metrics on aligned homogeneous spaces}

\author{Jorge Lauret}  
\author{Cynthia Will}

\address{FaMAF, Universidad Nacional de C\'ordoba and CIEM, CONICET (Argentina)}
\email{jorgelauret@unc.edu.ar} 
\email{cynthia.will@unc.edu.ar}

\thanks{We would like to acknowledge support from the ICTP through the Associates Programme and from the Simons Foundation through grant number 284558FY19.}

\date{\today}

\begin{document}

\maketitle

\begin{abstract}
Let $M=G/K$ be a compact homogeneous space and assume that $G$ and $K$ have many simple factors.  We show that the topological condition of having maximal third Betti number, in the sense that $b_3(M)=s-1$ if $G$ has $s$ simple factors, so called {\it aligned}, leads to a relatively manageable algebraic structure on the isotropy representation, paving the way to the computation of Ricci curvature formulas for a large class of $G$-invariant metrics.  As an application, we study the existence and classification of Einstein metrics on aligned homogeneous spaces.  
\end{abstract}

\tableofcontents

\section{Introduction}\label{intro}

As known, the study of the $G$-invariant Riemannian geometry of a compact homogeneous space $M=G/K$ such that $G$ and $K$ have many simple factors may involve hard technical difficulties.  In this paper, we show that the topological condition of having maximal third Betti number, in the sense that $b_3(M)=s-1$ if $G$ has $s$ simple factors, so called {\it aligned}, leads to a relatively manageable algebraic structure on the isotropy representation, paving the way to discover and understand the geometry of aligned homogeneous spaces.  As a first step, we provide here Ricci curvature formulas for a large class of $G$-invariant metrics and study the existence of Einstein metrics.  Some results in the case when $s=2$ have previously been obtained in the articles \cite{BRF, HHK, Es2}.     

In \cite{H3}, the authors gave a formula for $b_3(M)$ in terms of the Killing forms of the corresponding Lie algebras $\kg\subset\ggo$.  Assume that $G$ is semisimple and consider
$$
G=G_1\times\dots\times G_s, \qquad K=K_0\times K_1\times\dots\times K_t,
$$ 
the decompositions in simple factors (up to finite cover), where $K_0$ is the center of $K$.  According to \cite[Propositions 4.1,4.3]{H3}, if $K$ is semisimple, then  
$$
b_3(M)=s-\dim{\spann\left\{\left(\tfrac{1}{a_{1j}},\dots,\tfrac{1}{a_{sj}}\right) : j=1,\dots,t\right\}},
$$
where $\kil_{\pi_i(\kg_j)} = a_{ij}\kil_{\ggo_i}|_{\pi_i(\kg_j)}$ for all $i=1,\dots,s$, $j=1,\dots,t$, $\pi_i:\ggo\rightarrow\ggo_i$ is the usual projection and $\kil_\hg$ denotes the Killing form of a Lie algebra $\hg$ (technically, $\tfrac{1}{a_{ij}}$ must be replaced by $0$ if $\pi_i(\kg_j)=0$).  Thus for nontrivial $K$, $b_3(M)\leq s-1$ and equality holds if and only if the vectors $\left(\tfrac{1}{a_{1j}},\dots,\tfrac{1}{a_{sj}}\right)$ are collinear (e.g., if $K$ is simple, or more in general, if $\kil_{\pi_i(\kg)} = a_i\kil_{\ggo_i}|_{\pi_i(\kg)}$ for all $i=1,\dots,s$).  When $K$ is not semisimple, in order to have $b_3(M)=s-1$, some conditions on $\kil_\ggo|_{\kg_0}$ must be added (see Definition \ref{alig-def-2}).  We call $M=G/K$ an {\it aligned} homogeneous space when $b_3(M)=s-1$.  

In this paper, after some rather technical (but crucial) preliminaries in \S\ref{preli3}, we first study in \S\ref{normal-sec} the $s$-parametric space $\mca^{norm}$ of all {\it normal} metrics on an aligned space $M=G/K$ as above, i.e., the $G$-invariant metrics defined by bi-invariant metrics on $G$.  We obtain a formula for the Ricci curvature and, as an application, we prove the following.    

\begin{theorem}\label{thm1}
A normal metric on an aligned space $M=G/K$ is never Einstein, unless $M$ is the Ledger-Obata space $K\times\dots\times K/\Delta K$.    
\end{theorem}

We also give a formula for the scalar curvature functional $\scalar:\mca^{norm}\rightarrow\RR$ and study, for $s=2$, its critical points and the relevance of the {\it standard} metric $\gk\in\mca^{norm}$ (i.e., defined by $-\kil_\ggo$) from this variational point of view.  

Each normal metric $g_b$ determines a reductive decomposition $\ggo=\kg\oplus\pg$, where $\pg$ is the $g_b$-orthogonal complement of $\kg$.  In turn, the isotropy representation $\pg$ admits an $\Ad(K)$-invariant $g_b$-orthogonal decomposition given by 
$$
\pg=\pg_1\oplus\dots\oplus\pg_s\oplus\pg_{s+1}\oplus\dots\oplus\pg_{2s-1}, \quad 
\pg_j=\pg_j^0\oplus\dots\oplus\pg_j^t, \quad \forall j=s+1,\dots,2s-1,
$$
where for $i=1,\dots,s$, $\ggo_i=\pi_i(\kg)\oplus\pg_i$ is the $\kil_{\ggo_i}$-orthogonal reductive decomposition of the homogeneous space $G_i/\pi_i(K)$, and for $j=s+1,\dots,2s-1$, as $\Ad(K)$-representations, $\pg_j$ is equivalent to the adjoint representation $\kg$ and $\pg_j^l$ to $\kg_l$ for any $l=0,\dots,t$.  This naturally provides a large amount of $G$-invariant metrics on any aligned space $M=G/K$ given as a multiple of $g_b$ on each of these subspaces, denoted by 
$$
g=(x_1,\dots,x_s,x_{s+1,0},\dots,x_{s+1,t},\dots,x_{2s-1,0},\dots,x_{2s-1,t})_{g_b}, \quad x_1,\dots,x_{2s-1,t}>0, 
$$ 
and called {\it diagonal} metrics.  We denote by $\mca^{diag}$ the space of all diagonal metrics, on which the $s$ parameters needed to define $g_b$ are also varying.  Note that in spite the dimension of $\mca^{diag}$ is quite large, it may be still far from exhausting $\mca^G$, the space of all $G$-invariant metrics on $M$ (see Remark \ref{param}). 

In \S\ref{diag-sec}, we compute the Ricci operator of any diagonal metric $g$ on the first subspaces $\pg_1,\dots,\pg_s$, obtaining that 
$$
\Ricci(g)|_{\pg_i}\in\spann\left\{ I_{\pg_i}, \cas_{\chi_{i,0}}|_{\pg_i},\dots,\cas_{\chi_{i,t}}|_{\pg_i}\right\}, \qquad\forall i=1,\dots,s,
$$ 
where $I_{\pg_i}$ is the identity map on $\pg_i$ and $\cas_{\chi_{i,l}}$ is the Casimir operator of the isotropy representation of the homogeneous space $G_i/\pi_i(K_l)$.  As an application, the following structural conditions on $M=G/K$ are necessary for the existence of a diagonal Einstein metric: $I_{\pg_i}$ must belong to the subspace spanned by $\cas_{\chi_{i,0}}|_{\pg_i},\dots,\cas_{\chi_{i,t}}|_{\pg_i}$.  We note that this generalizes the Einstein condition for the standard metric on $G_i/\pi_i(K)$, that is, $\cas_{\chi_i}=\kappa_iI_{\pg_i}$ for some $\kappa_i\in\RR$, where $\cas_{\chi_i}=\cas_{\chi_{i,0}}+\dots+\cas_{\chi_{i,t}}$ is the Casimir operator of the isotropy representation of $G_i/\pi_i(K)$ (these spaces were classified by Wang and Ziller in \cite{WngZll2}).  

We henceforth focus on in the case when $s=2$, that is, on aligned homogeneous spaces $M=G_1\times G_2/K$, where $K=K_0\times\dots\times K_t$, with isotropy representation 
\begin{equation}\label{dec-intro}
\pg=\pg_1\oplus\pg_2\oplus\pg_3^0\oplus\dots\oplus\pg_3^t.
\end{equation}  
It is shown in \S\ref{s2-sec} that it is enough to consider the standard metric $\gk$ as a background metric to attain all diagonal metrics, so we study metrics of the form 
\begin{equation}\label{diag-intro}
g=(x_1,x_2,x_{3,0},\dots,x_{3,t})_{\gk}.
\end{equation}
A main result of this paper is the following formula for the Ricci curvature of $g$ obtained in \S\ref{s2-sec} and \S\ref{E-sec}.  

\begin{theorem}\label{ric-intro}  
If $\pg_i=\pg_i^1\oplus\dots\oplus\pg_i^{q_i}$, $i=1,2$ are decompositions in $\Ad(K)$-irreducible summands, the for any metric $g$ as in \eqref{diag-intro}, the decomposition \eqref{dec-intro} is $\ricci(g)$-orthogonal, 
$$
\ricci(g)(\pg_i^j,\pg_i^k)=0, \qquad\forall 1\leq j<k\leq q_i, \quad i=1,2,
$$
and the Ricci eigenvalues are given by 
\begin{align*}
r_1^j=&\tfrac{1}{2x_1} \sum\limits_{l=0}^t\left(1 - \tfrac{(c_1-1)x_{3,l}}{c_1x_1}\right) \kappa_{1,l}^j
+ \tfrac{1}{4x_1}, \quad j=1,\dots,q_1,  \\ 
r_2^j=&\tfrac{1}{2x_2} \sum\limits_{l=0}^t\left(1 - \tfrac{x_{3,l}}{c_1x_2}\right) \kappa_{2,l}^j
+ \tfrac{1}{4x_2}, \quad j=1,\dots,q_2,  \\
r_3^l=& \tfrac{(c_1-1)\lambda_l}{4x_{3,l}}\left(\tfrac{c_1^2}{(c_1-1)^2}-\tfrac{x_{3,l}^2}{x_1^2} 
-\tfrac{x_{3,l}^2}{(c_1-1)^2x_2^2}\right) 
+\tfrac{(c_1-1)x_{3,l}}{4c_1}\left(\tfrac{1}{x_1^2} 
+\tfrac{1}{(c_1-1)x_2^2}\right), \quad l=0,1,\dots,t,  
\end{align*} 
where $\cas_{\chi_{i,l}}|_{\pg_i^j}=\kappa_{i,l}^jI_{\pg_i^j}$ and $c_1,\lambda_1,\dots,\lambda_t$ are the positive constants associated to the aligned space $M=G_1\times G_2/K$ defined by $(a_{1l},a_{2l})=\lambda_l(c_1,\tfrac{c_1}{c_1-1})$ for $l=0,1,\dots,t$.    
\end{theorem}

In particular, if $\pg_1$ and $\pg_2$ are $\Ad(K)$-irreducible, then the space $\mca^{diag}$ is Ricci flow invariant.  On the other hand, we obtain that 
$$
(1,\dots,1)\in\spann\left\{(\kappa_{i,l}^1,\dots,\kappa_{i,l}^{q_i}):l=0,\dots,t\right\}\subset\RR^{q_i}, \qquad i=1,2,
$$ 
are necessary structural conditions for the existence of a diagonal Einstein metric, which may become very strong as $q_i$ increases.  

We use Theorem \ref{ric-intro} to study in \S\ref{E-sec} the existence and classification of diagonal Einstein metrics on aligned spaces.  The case when $x_{3,0}=\dots=x_{3,t}=:x_3$ has already been considered in previous articles: 
\begin{enumerate}[(a)]
\item \cite{Es2} An aligned space $M=G_1\times G_2/K$ admits a diagonal Einstein metric of the form $g=(x_1,x_2,x_3)_{\gk}$ if and only if for both $i=1,2$, $\cas_{\chi_i}=\kappa_iI_{\pg_i}$ for some $\kappa_i>0$ (i.e., the standard metric on $G_i/\pi_i(K)$ is Einstein), $\kil_{\pi_i(\kg)}=a_i\kil_{\ggo_i}|_{\pi_i(\kg)}$ for some $a_i\in\RR$ (in particular, $K$ is either semisimple or abelian) and certain quartic polynomial $p$ whose coefficients only depend on the dimensions and $a_1$, $a_2$ admits a real root.  

\item \cite{Es2} Any aligned space for which 
$$
\mca^G=\{ (x_1,x_2,x_3)_{\gk}:x_i>0\}
$$
automatically satisfies the two structural conditions in (a) and so the existence of a $G$-invariant Einstein metric only depends on whether $p$ has a real root or not.  The existence rate among this huge class, consisting of $12$ infinite families and $70$ sporadic examples, is approximately $\%75$.  

\item \cite{HHK} In the special case when $G_1=G_2=:H$ and $K$ is diagonally embedded in $H\times H$, the existence of an Einstein metric $g=(x_1,x_2,x_3)_{\gk}$ on $M=H\times H/\Delta K$ is therefore equivalent by part (a) to $\cas_\chi=\kappa I$, $\kil_\kg=a\kil_\hg|_\kg$ and the quartic equation defined by $p$, which reduces in this case to a quadratic one with discriminant $(2\kappa+1)^2 \geq 8a(1-a+\kappa)$.  

\item \cite{HHK} For almost all the spaces satisfying the first two structural conditions in part (c) ($17$ infinite families and $50$ sporadic examples), the above inequality strictly holds, so the existence of two Einstein metrics mostly holds.  
\end{enumerate}

For a general diagonal metric, i.e., such that not all $x_{3,l}$'s are equal (see \eqref{diag-intro}), the structural obstructions on the aligned space $M=G_1\times G_2/K$ for the Einstein condition are expected to be weaker.  The main results of our exploration in \S\ref{E-sec} can be summarized as follows.   

\begin{theorem}\label{E-intro}
\hspace{1cm} 

\begin{enumerate}[{\rm (i)}] 
\item The following  homogeneous spaces do not admit $G$-invariant Einstein metrics (see \S\ref{class1}):
\begin{enumerate}[{\small $\bullet$}] 
\item $\SO(2m)\times\SU(m+1)/S^1\times\SU(m)$, \qquad $m\geq 3$.  

\item $\Spe(m)\times\SU(m+1)/S^1\times\SU(m)$, \qquad $m\geq 2$.

\item $\Spe(m)\times\SO(2m)/S^1\times\SU(m)$, \qquad $m\geq 3$. 
\end{enumerate}

\item The space $M^{116}=\SU(9)\times\Fg_4/\SU(3)^2$ admits, up to scaling, exactly two diagonal Einstein metrics  with $x_{3,1}=x_{3,2}$ and one with $x_{3,1}\ne x_{3,2}$ which are pairwise non-homothetic (see Example \ref{su3}).  

\item If $G_1/K\ne G_2/K$ and for $i=1,2$, $\cas_{\chi_i}=\kappa_iI_{\pg_i}$ and $\kil_{\pi_i(\kg)}=a_i\kil_{\ggo_i}|_{\pi_i(\kg)}$, then any diagonal Einstein metric as in \eqref{diag-intro} necessarily has $x_{3,1}=\dots=x_{3,t}$, unless $M=\SU(9)\times\Fg_4/\SU(3)^2$.  All these spaces ($2$ infinite families and $24$ sporadic examples, see Table \ref{all1}) admit two diagonal Einstein metrics of the form $g=(x_1,x_2,1,\dots,1)$ (see \S\ref{class2}).  

\item For $M=H\times H/\Delta K$, where $H/K$ is an isotropy irreducible space such that $\kil_\kg=a\kil_\hg$, the existence or not of Einstein metrics of the form $g=(1,1,x_{3,1},\dots,x_{3,t})$ with not all $x_{3,l}$'s equal is established in Table \ref{irr} ($3$ infinite families and $3$ sporadic examples, see \S\ref{class3}).  

\item If $H/K$ is an irreducible symmetric space such that $\kil_\kg\ne a\kil_\hg$ for all $a\in\RR$, the existence or not of a diagonal Einstein metric with $x_1=x_2$ (so the $x_{3,l}$'s can not be all equal by (c) above) is described in Table \ref{sym} ($6$ infinite families and $7$ sporadic examples, see \S\ref{class4}).  

\item For any irreducible symmetric space $H/K$, with the exception perhaps of $\SO(p+q)/\SO(p)\times\SO(q)$ and $\Spe(p+q)/\Spe(p)\times\Spe(q)$, the aligned space $M=H\times H/\Delta K$ do not admit a diagonal Einstein metric of the form $g=(x_1,x_2,x_{3,0},\dots,x_{3,t})_{\gk}$ with $x_1\ne x_2$ (see Proposition \ref{x1nex2}).  
\end{enumerate}
\end{theorem}

It is proved in \cite{HHK} that the aligned spaces $M=H\times H/\Delta K$ admit a non-diagonal Einstein metric for any irreducible symmetric space $H/K$.  We note that part (iv), together with \cite[Theorem 7.3]{HHK}, complete the classification of $H\times H$-invariant Einstein on these spaces in the case when $\kil_\kg=a\kil_\hg$ for some $a>0$.

\vs \noindent {\it Acknowledgements.}  
We are very grateful with Emilio Lauret for many helpful conversations.

\section{Aligned homogeneous spaces}\label{preli3}
We overview in this section the class of homogeneous spaces with the richest third cohomology (other than Lie groups), i.e., the third Betti number satisfies that $b_3(G/K)=s-1$ if $G$ has $s$ simple factors, called {\it aligned} homogeneous spaces.  See \cite{H3, BRF} for more complete treatments.  

\subsection{Definition}\label{def-sec} 
Given a compact and connected differentiable manifold $M^n$ which is homogeneous, we fix an almost-effective transitive action of a compact connected Lie group $G$ on $M$.  The $G$-action determines a presentation $M=G/K$ of $M$ as a homogeneous space, where $K\subset G$ is the isotropy subgroup at some point $o\in M$. 

Let $M=G/K$ be a homogeneous space, where $G$ is a compact, connected and semisimple Lie group and $K$ is a closed subgroup.  We fix decompositions for the corresponding Lie algebras, 
\begin{equation}\label{decs}
\ggo=\ggo_1\oplus\dots\oplus\ggo_s, \qquad \kg=\kg_0\oplus\kg_1\oplus\dots\oplus\kg_t, 
\end{equation}
where the $\ggo_i$'s and $\kg_j$'s are simple ideals of $\ggo$ and $\kg$, respectively, and $\kg_0$ is the center of $\kg$.  If $\pi_i:\ggo\rightarrow\ggo_i$ is the usual projection, then we set $Z_i:=\pi_i(Z)$ for any $Z\in\ggo$.  

\begin{remark}\label{GiKj}
Up to finite cover, we have that 
$$
M=G_1\times\dots\times G_s/K_0\times K_1\times\dots\times K_t,
$$ 
where the $G_i$'s and $K_j$'s are Lie groups with Lie algebras $\ggo_i$'s and $\kg_j$'s, respectively.  
\end{remark}

The Killing form of a Lie algebra $\hg$ will always be denoted by $\kil_\hg$.  We consider the {\it Killing constants}, defined by 
$$
\kil_{\pi_i(\kg_j)} = a_{ij}\kil_{\ggo_i}|_{\pi_i(\kg_j)}, \qquad i=1,\dots,s, \quad j=0,1,\dots,t.  
$$ 
Note that $0\leq a_{ij}\leq 1$, $a_{ij}=0$ if and only if $j=0$ or $\pi_i(\kg_j)=0$, and $a_{ij}=1$ if and only if $\pi_i(\kg_j)=\ggo_i$ (see \cite{DtrZll} for more information on these constants).  

\begin{definition}\label{alig-def-2}
A homogeneous space $G/K$ as above with $K$ semisimple (i.e., $\kg_0=0$) is said to be {\it aligned} if $\pi_i(\kg_j)\ne 0$ (i.e., $a_{ij}>0$) for all $i,j$ and the vectors of $\RR^s$ given by  
$$
(a_{1j},\dots,a_{sj}), \qquad j=1,\dots,t,
$$ 
are collinear, say, there exist numbers $c_1,\dots,c_s>0$ with $\frac{1}{c_1}+\dots+\frac{1}{c_s}=1$ such that
$$
(a_{1j},\dots,a_{sj}) = \lambda_j(c_1,\dots,c_s) \quad\mbox{for some}\quad \lambda_j>0, \quad\forall j=1,\dots,t \quad \mbox{(i.e., $a_{ij}=\lambda_jc_i$)}.
$$
In the case when $\kg_0\ne 0$, $G/K$ is called {\it aligned} if in addition to the above conditions,     
\begin{equation}\label{al3} 
\kil_{\ggo_i}(Z_i,W_i) = \tfrac{1}{c_i}\kil_\ggo(Z,W), \qquad\forall Z,W\in\kg_0, \quad i=1,\dots,s.
\end{equation}   
Since $a_{i0}=0$ for all $i$, we set $\lambda_0:=0$.
\end{definition}

In other words, the ideals $\kg_j$'s are uniformly embedded in each $\ggo_i$ in some sense.  Note that $G/K$ is automatically aligned if $\kg$ is simple or one-dimensional, provided that $\pi_i(\kg)\ne 0$ for all $i=1,\dots,s$.

\begin{example}\label{S1}
The homogeneous space
$$
\SU(m)\times\SU(m)/\U(k)_{p,q}, \qquad 1<k<m, 
$$
where the center $K_0=S^1$ of $K=\U(k)_{p,q}$ is embedded with slope $(p,q)$, $p,q\in\NN$ and $K_1=\SU(k)$ is diagonally embedded, is aligned if and only if $p=q$, in which case $c_1=c_2=2$ (see \cite[Example 2.4]{Es2}).  On the other hand, for each space 
$$
\SO(2m)\times\SU(m+1)/S^1_{p,q}\times\SU(m), \qquad m\geq 3, 
$$
we have that $a_{11}=\tfrac{m}{2(m-1)}$ and $a_{21}=\tfrac{m}{m+1}$, so $c_1=\tfrac{3m-1}{2(m-1)}$, $c_2=\tfrac{3m-1}{m+1}$ and $\lambda_1=\tfrac{m}{3(m-1)}$.  This space is therefore aligned if and only if the slope $(p,q)$ satisfies \eqref{al3}, that is, 
$$
\kil_{\sog(2m)}(Z_1,Z_1)=\tfrac{2(m-1)}{3m-1}, \qquad 
\kil_{\sug(m+1)}(Z_2,Z_2)=\tfrac{m+1}{3m-1}, 
$$  
for $Z=(Z_1,Z_2)\in\kg_0$ such that $\kil_{\ggo}(Z,Z)=1$, where $\kg_0$ is the Lie algebra of $S^1_{p,q}$.  
\end{example}

The following properties of an aligned homogeneous space $G/K$ easily follow (see \cite{H3}): 
\begin{enumerate}[{\small $\bullet$}]
\item $\pi_i(\kg)\simeq\kg$ for all $i=1,\dots,s$.  

\item For any $Z,W\in\kg$, 
\begin{equation}\label{al2} 
\kil_{\ggo_i}(Z_i,W_i) = \tfrac{1}{c_i}\kil_\ggo(Z,W), \qquad \forall i=1,\dots,s.  
\end{equation} 
The existence of $c_1,\dots,c_s>0$ such that \eqref{al2} holds is an alternative definition of the notion of aligned.  

\item The Killing form of $\kg_j$ is given by
\begin{equation}\label{al1}
\kil_{\kg_j}=\lambda_j\kil_{\ggo}|_{\kg_j}, \qquad\forall j=1,\dots,t.  
\end{equation}
\end{enumerate}
Under the assumption that $\pi_i(\kg)\ne 0$ for all $i=1,\dots,s$, a homogeneous space $G/K$ is aligned if and only if $b_3(G/K)=s-1$, which is in turn equivalent to the fact that $Q|_{\kg\times\kg}$ is a scalar multiple of $\kil_\ggo|_{\kg\times\kg}$ for any bi-invariant symmetric bilinear form $Q$ on $\ggo$ (see \cite[Proposition 4.10]{H3}).

\subsection{Large classes}\label{exa-sec}
We now list some general constructions of aligned homogeneous spaces.  See \cite[Section 2.2]{Es2} for more examples in the case when $s=2$.  

\begin{example}\label{GLO}
If $\ggo_1=\dots=\ggo_s=\hg$ and $\pi_1=\dots=\pi_s$, i.e., $G=H\times\dots\times H$ ($s$-times), $H$ simple and $K\subset H$ a subgroup, then $G/\Delta K$ is aligned with 
$$
c_1=\dots=c_s=s, \qquad  \lambda_1=\tfrac{a_1}{s},\dots,\lambda_t=\tfrac{a_t}{s}, 
$$
where $\kil_{\kg_j}=a_j\kil_\hg|_{\kg_j}$ for each simple factor $\kg_j$ of $\kg$.  It is easy to see that $M=G/\Delta K$ is diffeomorphic to $(H/K)\times H^{s-1}$, where $H^{s-1}:=H\times\dots\times H$ ($(s-1)$-times).  In the particular case when $K=H$ and $s\geq 3$, these spaces are called {\it Ledger-Obata} (see \cite{NklNkn}).  The case $s=2$ was studied in \cite{HHK}.   
\end{example}

\begin{example}\label{kill-exa} 
Given compact homogeneous spaces $G_i/H_i$, $i=1,\dots,s$, such that $G_i$ is simple, $H_i\simeq K$ and $\kil_{\hg_i}=a_i\kil_{\ggo_i}|_{\hg_i}$ for some $a_i>0$ (e.g.\ if $K$ is simple, see \cite[pp.35]{DtrZll}) for all $i$, we consider $M=G/\Delta K$, where $G:=G_1\times\dots\times G_s$, $\Delta K:=\{ (\theta_1(k),\dots,\theta_s(k)):k\in K\}$ and $\theta_i:K\rightarrow H_i$ a Lie group isomorphism.  Note that $K$ is necessarily semisimple.  It is easy to see that $M=G/\Delta K$ is an aligned homogeneous space with 
$$
c_1=a_1\sum_{r=1}^s\tfrac{1}{a_r}, \quad \dots, \quad c_s=a_s\sum_{r=1}^s\tfrac{1}{a_r}, 
\qquad \lambda_1=\dots=\lambda_t=\left(\sum_{r=1}^s\tfrac{1}{a_r}\right)^{-1},
$$
and also that any aligned homogeneous space with $K$ semisimple and $\lambda_1=\dots=\lambda_t$ can be constructed in this way.  For $s=2$ and $a_1\leq a_2$ we obtain 
$$
1<c_1=\tfrac{a_1+a_2}{a_2}\leq 2\leq c_2=\tfrac{a_1+a_2}{a_1}, \qquad \lambda_j=\tfrac{a_1a_2}{a_1+a_2}, \quad\forall j.
$$    
\end{example}

\begin{example}\label{ex3}
Consider $M=\SU(n_1)\times\dots\times\SU(n_s)/\SU(k_1)\times\dots\times\SU(k_t)$, where $k_1+\dots+k_t< n_i$ for all $i$ and the standard block diagonal embedding is always taken.  It follows from \cite[pp.37]{DtrZll} that $a_{ij}=\tfrac{k_j}{n_i}$, which implies that $G/K$ is aligned with 
$$
c_i=\tfrac{n_1+\dots+n_s}{n_i}, \qquad \lambda_j=\tfrac{k_j}{n_1+\dots+n_s}.  
$$
These aligned spaces are therefore different from those provided by Examples \ref{GLO} and \ref{kill-exa}.  
\end{example}

We note that an aligned space has $c_1=\dots=c_s=s$ if and only if $a_{1j}=\dots=a_{sj}=:a_j$ for any $j=1,\dots,t$ (unless $K$ is abelian).  In that case, $\lambda_j=\tfrac{a_j}{s}$ for all $j$.

\subsection{Reductive decompositions}\label{rd-sec}
The technical background needed to work on aligned homogeneous spaces is given in this section.  Let $\mca^G$ denote the finite-dimensional manifold of all $G$-invariant Riemannian metrics on a compact homogeneous space $M=G/K$.  For any reductive decomposition $\ggo=\kg\oplus\pg$ (i.e., $\Ad(K)\pg\subset\pg$), giving rise to the usual identification $T_oM\equiv\pg$, we identify any $g\in\mca^G$ with the corresponding $\Ad(K)$-invariant inner product on $\pg$, also denoted by $g$.  

We assume from now on that $M=G/K$ is an aligned homogeneous space as in Definition \ref{alig-def-2}.  For any given bi-invariant metric on $G$, say
\begin{equation}\label{gbdef}
g_b=z_1(-\kil_{\ggo_1})+\dots+z_s(-\kil_{\ggo_s}), \qquad z_1,\dots,z_s>0,   
\end{equation}
we consider the $g_b$-orthogonal reductive decomposition $\ggo=\kg\oplus\pg$ and the corresponding {\it normal} metric, also denoted by $g_b\in\mca^G$.  The following notation will be strongly used from now on without any further reference: for $j=s+1,\dots,2s-1$, 
$$
A_{j}:=-\tfrac{c_{j-s+1}}{z_{j-s+1}}\left(\tfrac{z_1}{c_1}+\dots+\tfrac{z_{j-s}}{c_{j-s}}\right), \quad B_{j}:=\tfrac{z_1}{c_1}+\dots+\tfrac{z_{j-s}}{c_{j-s}}+A_{j}^2\tfrac{z_{j-s+1}}{c_{j-s+1}}, \quad
B_{2s}:=\tfrac{z_1}{c_1}+\dots+\tfrac{z_s}{c_s}.  
$$ 
Let us also consider the $g_b$-orthogonal $\Ad(K)$-invariant decomposition 
\begin{equation}\label{dec1}
\pg=\pg_1\oplus\dots\oplus\pg_s\oplus\pg_{s+1}\oplus\dots\oplus\pg_{2s-1},
\end{equation}
provided by \cite[Proposition 5.1]{H3}.  For $i\leq s$, $\pg_i$ is identified with its (only nonzero) projection on $\ggo_i$, which comes from the $\kil_{\ggo_i}$-orthogonal reductive decomposition 
$$
\ggo_i=\pi_i(\kg)\oplus\pg_i, \qquad i=1,\dots,s,
$$ 
of the homogeneous space $M_i:=G_i/\pi_i(K)$.  For $s<j$, $\pg_j$ is defined by 
$$
\pg_j:=\{\vp_j(Z):=\left(Z_1,\dots,Z_{j-s},A_jZ_{j-s+1},0,\dots,0\right):Z\in\kg\}, 
$$
so $\vp_j:\kg\rightarrow\pg_j$ is an $\Ad(K)$-equivariant isomorphism.  In this way, as $\Ad(K)$-representations, $\pg_i$ is equivalent to the isotropy representation of the homogeneous space $G_i/\pi_i(K)$ for each $i=1,\dots,s$ and $\pg_j$ is equivalent to the adjoint representation $\kg$ for any $j=s+1,\dots,2s-1$.  We note that 
$$
\pi_1(\kg)\oplus\dots\oplus\pi_s(\kg) = \pg_{s+1}\oplus\dots\oplus\pg_{2s-1}\oplus\kg  
$$
is a Lie subalgebra of $\ggo$, which is abelian if and only if $\kg$ is abelian.  

The only nonzero brackets between the $\pg_i$'s are: 
\begin{align}
&[\pg_1,\pg_1]\subset\pg_1+\pg_{s+1}+\dots+\pg_{2s-1}+\kg, \label{pipj0}\\  
&[\pg_i,\pg_i]\subset\pg_i+\pg_{s+i-1}+\dots+\pg_{2s-1}+\kg, \qquad 2\leq i \leq s, \label{pipj1}\\  
&[\pg_j,\pg_i]\subset\pg_i, \label{pipj2}, \qquad i\leq s<j,\\ 
&[\pg_j,\pg_j] \subset\pg_j+\dots+\pg_{2s-1}+\kg, \qquad s<j,\label{pipj3}\\ 
&[\pg_j,\pg_k] \subset\pg_j, \qquad s<j<k. \label{pipj4}
\end{align}
Each subspace $\pg_j$, $s<j$ in turn admits an $\Ad(K)$-invariant decomposition 
\begin{equation}\label{pjdec}
\pg_j=\pg_j^0\oplus\pg_j^1\oplus\dots\oplus\pg_j^t, 
\end{equation}
which is also orthogonal with respect to $g_b$ and for any $l=0,\dots,t$, the subspace $\pg_j^l$ is equivalent (also via $\vp_j$) to the adjoint representation $\kg_l$ as an $\Ad(K)$-representation (see \cite[Proposition 5.1]{H3}); in particular, $\pg_j^l$ is $\Ad(K)$-irreducible for any $j$ and $1\leq l$.  This provides an alternative $g_b$-orthogonal $\Ad(K)$-invariant decomposition   
\begin{equation}\label{ptildedec}
 \pg_{s+1}\oplus\dots\oplus\pg_{2s-1} = \qg_0\oplus\qg_1\oplus\dots\oplus\qg_t, \qquad \qg_l:=\pg_{s+1}^l\oplus\dots\oplus\pg_{2s-1}^l, \quad l=0,\dots,t.
\end{equation}
In particular, $\qg_l\simeq\kg_l\oplus\dots\oplus\kg_l$ ($(s-1)$-times) as an $\Ad(K)$-representation and thus none of the irreducible components of $\qg_l$ is equivalent to any of those of $\qg_m$ for $l\ne m$, which implies that $\qg_l$ and $\qg_m$ are necessarily orthogonal with respect to any $\Ad(K)$-invariant symmetric $2$-tensor (e.g., $G$-invariant metrics and their Ricci tensors).  

We set 
$$
\ip:=-\kil_{\ggo}|_{\kg},
$$ 
and fix a $\ip$-orthonormal basis $\{ Z^\alpha\}_{\alpha=1}^{\dim{\kg}}$ of $\kg$ adapted to the $\ip$-orthogonal decomposition $\kg=\kg_0\oplus\dots\oplus\kg_t$.  It follows from \eqref{al2} that 
\begin{equation}\label{gbpi}
g_b(Z_i,W_i) = \tfrac{z_i}{c_i}\la Z,W\ra, \qquad\forall Z,W\in\kg, \qquad\mbox{and so}\qquad g_b|_{\kg\times\kg}=B_{2s}\ip.
\end{equation}
For each $l=0,\dots,t$, a $g_b$-orthonormal basis $\{e_\alpha^{j,l}\}_{\alpha=1}^{\dim{\kg_l}}$ of $\pg_j^l$,  can be defined by 
\begin{equation}\label{basis}
e_\alpha^{j,l}:=\tfrac{1}{\sqrt{B_j}}\vp_j(Z^\alpha), \qquad s<j\leq 2s, \quad Z^\alpha\in\kg_l.
\end{equation}
The notation $\{e_\alpha^j\}:=\{e_\alpha^{j,l}:\alpha=1,\dots,\dim{\kg_l}, \; l=0,\dots,t\}$ will also be used for any $j>s$, which represents a $g_b$-orthonormal basis of $\pg_j$.  Note that $\pg_{2s}:=\kg$ and $e_\alpha^{2s}:=\tfrac{1}{\sqrt{B_{2s}}}Z^\alpha$ is a $g_b$-orthonormal basis of $\pg_{2s}$.  The union of all these bases together with $g_b$-orthonormal bases $\{e_\alpha^i\}_{\alpha=1}^{\dim{\pg_i}}$ of $\pg_i$, $i=1,\dots,s$, form the $g_b$-orthonormal basis of $\ggo$ which will be used in the computations of the Ricci curvature.  

The following notation will be used throughout the paper:
\begin{align*}
d:=\dim{K}, \qquad d_l:=\dim{\kg_l}, \quad l=0,\dots,t, \qquad \mbox{so}\quad  d=d_0+d_1+\dots+d_t, \\ 
n_i:=\dim{\pg_i}=\dim{G_i}-d, \quad i=1,\dots,s, \quad n:=\dim{M}=n_1+\dots+n_s+(s-1)d.
\end{align*}

\begin{example}\label{dim5-2}
The well-known homogeneous space $M^5=\SU(2)\times\SU(2)/S^1_{p,q}$ is the lowest dimensional aligned space, here $c_1=\tfrac{p^2+q^2}{p^2}$ and $c_2=\tfrac{p^2+q^2}{q^2}$ (see \cite[Example 2.3]{Es2}).  For the normal metric $g_b=(z_1,z_2)$ on $M$, we have that 
$$
A_3:=-\tfrac{p^2z_1}{q^2z_2}, \qquad B_3:=\tfrac{p^2z_1}{p^2+q^2}+\tfrac{p^4z_1^2}{q^2(p^2+q^2)z_2},  \qquad
B_4:=\tfrac{p^2z_1}{p^2+q^2}+\tfrac{q^2z_2}{p^2+q^2}.  
$$
Thus the $g_b$-orthogonal reductive decomposition is given by $\ggo = \kg\oplus \pg_1\oplus\pg_2\oplus\pg_3$, where $\kg:=\RR(pZ,qZ)$, 
$$
\pg_1:=\spann\{(X,0),(Y,0)\}, \quad \pg_2:=\spann\{(0,X),(0,Y)\}, \quad \pg_3:=\RR(pZ,A_3qZ), 
$$ 
$X:=\left[\begin{smallmatrix} 0&-1\\ 1&0\end{smallmatrix}\right]$, $Y:=\left[\begin{smallmatrix} 0&\im\\ \im&0\end{smallmatrix}\right]$ and $Z:=\left[\begin{smallmatrix} \im&0\\ 0&-\im\end{smallmatrix}\right]$ and the adapted $g_b$-orthonormal basis by 
$$
e^1_1=\tfrac{1}{\sqrt{8}}(X,0), \; e^1_2=\tfrac{1}{\sqrt{8}}(Y,0), \; e^2_1=\tfrac{1}{\sqrt{8}}(0,X), \; e^2_2=\tfrac{1}{\sqrt{8}}(0,Y), \; e^3_1=\tfrac{1}{\sqrt{8B_3}}(pZ,A_3qZ).
$$  
\end{example}

\section{Normal metrics}\label{normal-sec}

In this section, we compute the Ricci curvature of normal metrics on an aligned homogeneous space $M=G/K$ with positive constants $c_1,\dots,c_s$ and $\lambda_1,\dots,\lambda_t$ as in \S\ref{preli3} (recall that $\lambda_0=0$), in order to study the Einstein condition and the behavior of the scalar curvature functional $\scalar$ restricted to the submanifold $\mca^{norm}\subset\mca^G$ of all normal metrics on $M=G/K$.    

We denote by
$$
g_b=(z_1,\dots,z_s)
$$ 
the normal metric on $M=G/K$ defined by the inner product $g_b|_{\pg\times\pg}$, where $g_b$ is the bi-invariant metric given in \eqref{gbdef} and $\ggo=\kg\oplus\pg$ is the $g_b$-orthogonal reductive decomposition.

\subsection{Ricci and scalar curvature}
We first consider, for each $i=1,\dots,s$, the homogeneous space $M_i=G_i/\pi_i(K)$ (see Remark \ref{GiKj}) with $\kil_{\ggo_i}$-orthogonal reductive decomposition $\ggo_i=\pi_i(\kg)\oplus\pg_i$ endowed with its standard metric, which will be denoted by $\gk^i$.  According to \cite[Proposition (1.91)]{WngZll2} (see also \cite[(5)]{HHK} and \cite[(6)]{stab}), the Ricci operator is given by  
\begin{equation}\label{MgBi}
\Ricci(\gk^i) = \unm\cas_{\chi_i} + \unc I_{\pg_i} = \unc\sum_\alpha(\ad_{\pg_i}{e^i_\alpha})^2 + \unm I_{\pg_i}, \qquad\forall i=1,\dots,s,
\end{equation}
where 
$$
\cas_{\chi_i}:=\cas_{\pg_i,-\kil_{\ggo_i}|_{\pi_i(\kg)}}:\pg_i\longrightarrow\pg_i
$$ 
is the Casimir operator of the isotropy representation $\chi_i:\pi_i(K)\rightarrow\End(\pg_i)$ of $G_i/\pi_i(K)$ with respect to the bi-invariant inner product $-\kil_{\ggo_i}|_{\pi_i(\kg)}$.  Note that $\cas_{\chi_i}\geq 0$, where equality holds if and only if $\pg_i=0$ (i.e., $M_i$ is a point).

Let us first compute the Ricci curvature of the standard metric $\gk$ on $M=G/K$.  Since $\{\sqrt{c_i}Z^\alpha_i\}_{\alpha=1}^{\dim{\kg}}$ is a $-\kil_{\ggo_i}$-orthonormal basis of $\pi_i(\kg)$ (see \eqref{al2}) and $\ad_{\ggo_i}{Z^\alpha_i}|_{\pg_i} = \ad{Z^\alpha}|_{\pg_i}$, we have that 
\begin{equation}\label{Cchi}
\cas_{\chi_i} = c_i\cas_\chi|_{\pg_i}, \qquad\forall i=1,\dots,s, 
\end{equation}
where $\cas_\chi:\pg\rightarrow\pg$ is the Casimir operator of the isotropy representation of $G/K$ with respect to $-\kil_{\ggo}|_{\kg}$.  On the other hand, 
\begin{equation}\label{Cchi2}
\cas_\chi|_{\qg_l}=\lambda_lI_{\qg_l}, \qquad\forall l=0,1,\dots,t; 
\end{equation}
indeed, $\ad{Z^\alpha}|_{\pg_j^l} = \vp_j\ad_{\kg_l}{Z^\alpha}\vp_j^{-1}$ for all $Z^\alpha\in\kg_l$, $s<j$ and so by \eqref{al1},
$$
\cas_\chi|_{\pg_j^l} = \vp_j\cas_{\kg_l,-\kil_\ggo|_{\kg_l}}\vp_j^{-1} 
= \vp_j\cas_{\kg_l,-\tfrac{1}{\lambda_l}\kil_{\kg_l}}\vp_j^{-1} = \lambda_l\vp_jI_{\kg_l}\vp_j^{-1} =  \lambda_lI_{\pg_j^l}.
$$
Thus $\cas_\chi$ and $\Ricci(\gk)$ leave invariant the $\gk$-ortogonal $\Ad(K)$-invariant decomposition (see \eqref{dec1} and \eqref{ptildedec})
\begin{equation}\label{dec2}
\pg=\pg_1\oplus\dots\oplus\pg_{s}\oplus\qg_0\oplus\qg_1\oplus\dots\oplus\qg_t,
\end{equation}
where $\ggo=\kg\oplus\pg$ is a $\gk$-orthogonal reductive decomposition,
\begin{equation}\label{Cchial}
\cas_\chi = \tfrac{1}{c_1}\cas_{\chi_1}+\dots+\tfrac{1}{c_s}\cas_{\chi_s} 
+\lambda_0 I_{\qg_0}+\lambda_1I_{\qg_1}+\dots+\lambda_tI_{\qg_t}, 
\end{equation}
and so
\begin{align}
\Ricci(\gk) = \unm\cas_{\chi} + \unc I_{\pg} \notag  
=& (\tfrac{1}{2c_1}\cas_{\chi_1}+\unc I_{\pg_1})+\dots+(\tfrac{1}{2c_s}\cas_{\chi_s} 
+\unc I_{\pg_s}) \label{ricgBal}\\ 
&+\unc I_{\qg_0}+(\unm\lambda_1+\unc)I_{\qg_1}+\dots+(\unm\lambda_t+\unc)I_{\qg_t}.  \notag
\end{align}

Secondly, we compute the Ricci curvature of any normal metric $g_b$ on $M=G/K$, which, as expected, will have a more involved formula.  We respectively denote by $\Ricci(g)$ and $\ricci(g)$ the Ricci operator and Ricci tensor of a metric $g$.  The $g_b$-orthogonal decompositions given in \eqref{dec1} and \eqref{ptildedec} are considered in the following statement.  

\begin{proposition}\label{ricgbal}
The Ricci curvature of a normal metric $g_b=(z_1,\dots,z_s)$ on an aligned homogeneous space $M=G/K$ is given by: 
\begin{enumerate}[{\rm (i)}]
\item $\Ricci(g_b)|_{\pg_i} = \tfrac{1}{2B_{2s}c_i}\cas_{\chi_i} + \tfrac{1}{4z_i}I_{\pg_i}$, \quad $i\leq s$.
\item[ ]
\item $\ricci(g_b)(\pg_i,\pg_j)=0$, \quad $i\ne j\leq s$ or $\quad i\leq s<j$.  
\item[ ]
\item $\Ricci(g_b)|_{\pg_j^l}=  
\left(\tfrac{\lambda_l}{2B_{2s}} + \tfrac{1}{4B_j}\left(\tfrac{1}{c_1}+\dots+\tfrac{1}{c_{j-s}}+A_j^2\tfrac{1}{c_{j-s+1}}\right)\right)I_{\pg_j^l}$, \quad $s<j$, $\quad l=0,1,\dots,t$. 
\item[ ]
\item $\ricci(g_b)(\pg_j^l,\pg_j^m)=0$, \quad $s<j$, $\quad 0\leq l\ne m\leq t$.  
\item[ ]
\item $\ricci(g_b)(e^{j,l}_\alpha,e^{k,m}_\beta) = \delta_{l,m}\delta_{\alpha\beta} \tfrac{1}{4\sqrt{B_jB_k}z_{j-s+1}}\left(\tfrac{z_{j-s+1}-z_1}{c_1}+\dots+\tfrac{z_{j-s+1}-z_{j-s}}{c_{j-s}}\right)$, \\ $s<j<k$.  
\end{enumerate}
\end{proposition}

\begin{proof} 
According to \cite[Proposition (1.91)]{WngZll2} (see also \cite[(5)]{HHK}), the Ricci operator is given by 
\begin{equation}\label{ricgb}
\Ricci(g_b) = \unm\cas_{\pg,g_b|_\kg} + \unc\cas_{\ggo,g_b}|_\pg, 
\end{equation}
where 
$
\cas_{\pg,g_b|_\kg}:\pg\longrightarrow\pg
$ 
is the Casimir operator of the isotropy representation $K\rightarrow\End(\pg)$ of $M=G/K$ with respect to $g_b|_{\kg}$, 
$
\cas_{\ggo,g_b}:\ggo\longrightarrow\ggo
$
is the Casimir operator of the adjoint representation of $G$ relative to $g_b$ and $\cas_{\ggo,g_b}|_\pg$ is the map obtained by restricting and projecting $\cas_{\ggo,g_b}$ on $\pg$.  

Since $\cas_{\pg,g_b|_\kg} = \tfrac{1}{B_{2s}}\cas_{\chi}$ (recall that $g_b|_\kg=B_{2s}\ip$ by \eqref{gbpi}) and $\cas_{\ggo,g_b} = \tfrac{1}{z_1}I_{\ggo_1}+\dots+\tfrac{1}{z_s}I_{\ggo_s}$, parts (i) and (ii) follow from \eqref{ricgb} and \eqref{Cchial}.  In order to prove the other two parts, using \eqref{gbpi}, we compute for any pair $s<j< k$ (recall that $Z^\alpha\in\kg_l$ and $Z^\beta\in\kg_m$),
\begin{align*}
&g_b(\cas_{\ggo,g_b}e^{j,l}_\alpha,e^{k,m}_\beta) \\
=& \tfrac{1}{\sqrt{B_jB_k}} g_b\left(\left(\tfrac{1}{z_1}Z^\alpha_1,\dots,\tfrac{1}{z_{j-s}}Z^\alpha_{j-s},\tfrac{1}{z_{j-s+1}}A_jZ^\alpha_{j-s+1},0,\dots,0\right),
\vp_k(Z^\beta)\right) \\ 
=& \tfrac{1}{\sqrt{B_jB_k}} \left(\sum_{m=1}^{j-s}\tfrac{1}{z_m}g_b(Z_m^\alpha,Z_m^\beta) + \tfrac{A_j}{z_{j-s+1}}g_b(Z_{j-s+1}^\alpha,Z_{j-s+1}^\beta)\right) \\ 
=&  \tfrac{1}{\sqrt{B_jB_k}} \left(\sum_{m=1}^{j-s}\tfrac{1}{c_m}\la Z^\alpha,Z^\beta\ra + \tfrac{A_j}{c_{j-s+1}}\la Z^\alpha,Z^\beta\ra\right) \\ 
=& \delta_{l,m}\delta_{\alpha\beta}\tfrac{1}{\sqrt{B_jB_k}}\left(\tfrac{1}{c_1}+\dots+\tfrac{1}{c_{j-s}}+A_j\tfrac{1}{c_{j-s+1}}\right) \\ 
=&\delta_{l,m}\delta_{\alpha\beta}\tfrac{1}{\sqrt{B_jB_k}}\tfrac{1}{z_{j-s+1}}\left(\tfrac{z_{j-s+1}-z_1}{c_1}+\dots+\tfrac{z_{j-s+1}-z_{j-s}}{c_{j-s}}\right),
\end{align*}
and for any $s<j$, 
\begin{align*}
&g_b(\cas_{\ggo,g_b}e^{j,l}_\alpha,e^{j,m}_\beta) \\
=& \tfrac{1}{B_j}g_b\left(\left(\tfrac{1}{z_1}Z^\alpha_1,\dots,\tfrac{1}{z_{j-s}}Z^\alpha_{j-s},\tfrac{1}{z_{j-s+1}}A_jZ^\alpha_{j-s+1},0,\dots,0\right),
\vp_j(Z^\beta)\right) \\ 
=& \tfrac{1}{B_j}\left(\sum_{m=1}^{j-s}\tfrac{1}{z_m}g_b(Z_m^\alpha,Z_m^\beta) + \tfrac{A_j^2}{z_{j-s+1}}g_b(Z_{j-s+1}^\alpha,Z_{j-s+1}^\beta)\right) \\ 
=&  \tfrac{1}{B_j}\left(\sum_{m=1}^{j-s}\tfrac{1}{c_m}\la Z^\alpha,Z^\beta\ra + \tfrac{A_j^2}{c_{j-s+1}}\la Z^\alpha,Z^\beta\ra\right) 
= \delta_{l,m}\delta_{\alpha\beta}\tfrac{1}{B_j} \left(\tfrac{1}{c_1}+\dots+\tfrac{1}{c_{j-s}}+A_j^2\tfrac{1}{c_{j-s+1}}\right). 
\end{align*}
Thus parts  (iii), (iv) and (v) also follow from \eqref{ricgb} and \eqref{Cchial}, concluding the proof.
\end{proof}

\begin{corollary}\label{scgb}
The scalar curvature of a normal metric $g_b=(z_1,\dots,z_s)$ on an aligned homogeneous space $M=G/K$ is given by: 
$$
\scalar(g_b) = \tfrac{d-\sum\lambda_ld_l}{2} \tfrac{1}{\tfrac{z_1}{c_1}+\dots+\tfrac{z_s}{c_s}} 
+\unc\sum_{i=1}^s\tfrac{n_i}{z_i} 
+\tfrac{d}{4}\sum_{j=s+1}^{2s-1} 
\tfrac{\left(\tfrac{1}{c_1}+\dots+\tfrac{1}{c_{j-s}}\right)\tfrac{z_{j-s+1}^2}{c_{j-s+1}} +  
\left(\tfrac{z_1}{c_1}+\dots+\tfrac{z_{j-s}}{c_{j-s}}\right)^2}
{z_{j-s+1}\left(\tfrac{z_1}{c_1}+\dots+\tfrac{z_{j-s}}{c_{j-s}}\right)\left(\tfrac{z_1}{c_1}+\dots+\tfrac{z_{j-s+1}}{c_{j-s+1}}\right)},
$$
and in particular, if $c_1=\dots=c_s=s$, then 
$$
\scalar(g_b) = \tfrac{s\left(d-\sum\lambda_ld_l\right)}{2} \tfrac{1}{z_1+\dots+z_s} 
+\unc\sum_{i=1}^s\tfrac{n_i}{z_i} 
+\tfrac{d}{4}\sum_{j=s+1}^{2s-1} 
\tfrac{(j-s)z_{j-s+1}^2 +  
\left(z_1+\dots+z_{j-s}\right)^2}
{z_{j-s+1}\left(z_1+\dots+z_{j-s}\right)\left(z_1+\dots+z_{j-s+1}\right)}.
$$
\end{corollary}

\begin{remark}
We could not find in the literature a formula for $\vol(M,g_b)$.  It would be useful to compute the function $v(z_1,\dots,z_s)$ defined by $\vol(M,g_b)=v(z_1,\dots,z_s)\vol(M,\gk)$ for any $g_b$.  This was obtained for $s=2$ in \S\ref{gbs2-sec} below.  
\end{remark}

\begin{proof}
Since by \cite[(9)]{stab}, 
\begin{equation}\label{trC} 
\tr{\cas_{\chi_i}} = \sum_{l=0}^ t (1-a_{il})d_l = \sum_{l=0}^ t (1-\lambda_lc_i)d_l, \qquad \forall i=1,\dots,s,
\end{equation}
if we set $B'_j:= \tfrac{1}{c_1}+\dots+\tfrac{1}{c_{j-s}}+A_j^2\tfrac{1}{c_{j-s+1}}$, then it follows from Proposition \ref{ricgbal} that, 
\begin{align*}
\scalar(g_b) =& \tfrac{1}{2B_{2s}}\sum_{i,l} \tfrac{(1-\lambda_lc_i)d_l}{c_i} + \unc\sum_i \tfrac{n_i}{z_i} 
+ \tfrac{s-1}{2B_{2s}}\sum_l\lambda_ld_l + \unc\sum_{j,l} \tfrac{B'_jd_l}{B_j}\\
=& \tfrac{d}{2B_{2s}} - \tfrac{1}{2B_{2s}}\sum_{l} \lambda_ld_l + \unc\sum_i \tfrac{n_i}{z_i} 
+ \tfrac{d}{4}\sum_{j>s} \tfrac{B'_j}{B_j}.
\end{align*}
The formulas for $\scalar(g_b)$ given in the statement therefore follow from 
\begin{align*}
B_j=&\tfrac{z_1}{c_1}+\dots+\tfrac{z_{j-s}}{c_{j-s}}+\left(\tfrac{z_1}{c_1}+\dots+\tfrac{z_{j-s}}{c_{j-s}}\right)^2\tfrac{c_{j-s+1}}{z_{j-s+1}} \\
=&\tfrac{c_{j-s+1}}{z_{j-s+1}}\left(\tfrac{z_1}{c_1}+\dots+\tfrac{z_{j-s}}{c_{j-s}}\right)\left(\tfrac{z_1}{c_1}+\dots+\tfrac{z_{j-s+1}}{c_{j-s+1}}\right), \\
B'_j=&\tfrac{1}{c_1}+\dots+\tfrac{1}{c_{j-s}}+\left(\tfrac{z_1}{c_1}+\dots+\tfrac{z_{j-s}}{c_{j-s}}\right)^2\tfrac{c_{j-s+1}}{z_{j-s+1}^2} \\
=&\tfrac{c_{j-s+1}}{z_{j-s+1}^2}\left(\left(\tfrac{1}{c_1}+\dots+\tfrac{1}{c_{j-s}}\right)\tfrac{z_{j-s+1}^2}{c_{j-s+1}} 
+\left(\tfrac{z_1}{c_1}+\dots+\tfrac{z_{j-s+1}}{c_{j-s+1}}\right)^2\right), 
\end{align*}
concluding the proof.
\end{proof}

\subsection{Normal Einstein metrics}\label{Eno-sec} 
As an application of Proposition \ref{ricgbal}, we show in this section that normal metrics can never be Einstein on an aligned homogeneous space.  This was proved for the spaces of the form $M=H\times H/\Delta K$ in \cite{HHK}.  

\begin{theorem}\label{gbalE}
A normal metric $g_b$ on an aligned homogeneous space $M=G/K$ (other than the Ledger-Obata space) is never Einstein.  
\end{theorem}

\begin{remark}\label{LO}
In the extremal case when $\pi_i(\kg)=\ggo_i$ for all $i=1,\dots,s$, $s\geq 3$, we obtain the Ledger-Obata space $M=K\times\dots\times K/\Delta K$, where $K$ is a compact simple Lie group (see \cite{NklNkn}).  It is well known  that the standard metric is always Einstein here.  Note that if $s=2$ then $M=K\times K/\Delta K$ is an irreducible symmetric space.  
\end{remark}

\begin{proof}
If $g_b=(z_1,\dots,z_s)$ is Einstein, then it follows from Proposition \ref{ricgbal}, (v) that $z_1=\dots=z_{s-1}$, so we can assume that $z_i=1$ for all $1\leq i\leq s-1$ from now on.  By part (i) of the same proposition, we have that $\cas_{\chi_i}=\kappa_iI_{\pg_i}$ for some $\kappa_i\in\RR$ and so 
$$
\Ricci(g_b)=r_iI_{\pg_i}, \qquad r_i:=\tfrac{c_s\kappa_i}{2(c_s-1+z_s)c_i}+\unc, \quad\forall i=1,\dots,s-1, \qquad 
r_s:=\tfrac{\kappa_s}{2(c_s-1+z_s)}+\tfrac{1}{4z_s},
$$
which implies that $\tfrac{\kappa_1}{c_1}=\dots=\tfrac{\kappa_{s-1}}{c_{s-1}}$.  On the other hand, from Proposition \ref{ricgbal}, (iii), we obtain that either $K$ is semisimple and $\lambda_1=\dots=\lambda_t=:\lambda$ or $\kg=\kg_0$ (i.e., $\lambda=0$).  If $s\geq 3$, then
$$
\Ricci(g_b)=r_{s+1}I_{\pg_{s+1}}, \qquad r_{s+1}:=\tfrac{\lambda c_s}{2(c_s-1+z_s)}+\unc, 
$$
by using that $A_{s+1}=-\tfrac{c_2}{c_1}$ and $B_{s+1}=\tfrac{c_1+c_2}{c_1^2}$.  It now follows from $r_i=r_{s+1}$ that $\tfrac{\kappa_i}{c_i}=\lambda=\tfrac{a_{ij}}{c_i}$ and thus $K$ is semisimple and $\kappa_i=a_i:=a_{i1}=\dots=a_{it}$, for all $i=1,\dots,s-1$.  Thus each nontrivial piece $M_i=G_i/\pi_i(K)$ consists of a homogeneous space $H/K$ with $K$ semisimple such that $\cas_{\chi}=\kappa I_{\pg}$, $\kil_\kg=a\kil_\hg|_{\kg}$ and $a=\kappa=\tfrac{\dim{\kg}}{\dim{\hg}}$.  In \cite[Tables 3-11]{HHK}, where all these spaces have been listed, one can easily check that always $a\ne \kappa$ (or $a\ne\tfrac{\dim{\kg}}{\dim{\ggo}}$), which is a contradiction, unless all pieces $M_i=G_i/\pi_i(K)$ are trivial, i.e., $M=G/K$ is the Ledger-Obata space.    

It only remains to consider $s=2$.  In this case, it is easy to see that $g_b$ is Einstein if and only if 
$$
z_1=(2(\kappa_1-a_1)+1)z_2 \quad\mbox{and}\quad z_2=(2(\kappa_2-a_2)+1)z_1,
$$
which implies that $(2(\kappa_1-a_1)+1)(2(\kappa_2-a_2+1)=1$.  According to \cite[Remark 5.4]{HHK}, $a_i<\kappa_i$ and so this is impossible for all the spaces $H/K$ with $K$ semisimple such that $\cas_{\chi}=\kappa I_{\pg}$ and $\kil_\kg=a\kil_\hg|_{\kg}$, excepting only $\SO((m-1)(2m+1))/\Spe(m)$, $m\geq 3$ and $\Spe(mk)/\Spe(k)^m$, $k\geq 1$, $m\geq 3$.  For the first family, one always has that $2(\kappa-a)+1<0$, so it can be ruled out, and for the second family we obtain that $2(\kappa_1-a_1)+1=\tfrac{km}{km+1}$.  Thus we need a second space of the form $G_2/\Spe(k)^m$ such that $\kappa_2-a_2=\tfrac{1}{2km}$.  It is easy to check in the tables given in \cite{HHK} that this never holds (see items 2, 5 and 7 in Table 3, 5 in Table 5, 5 and 11 in Table 6 and 1 and 4 in Table 11), which concludes the proof.  
\end{proof}

\subsection{The case $s=2$}\label{gbs2-sec} 
The numbers defined at the begining of \S\ref{rd-sec} simplify as follows for $s=2$:
$$
\tfrac{1}{c_1}+\tfrac{1}{c_2} = 1 \quad (\mbox{i.e.},\; c_2=\tfrac{c_1}{c_1-1}), \quad A_3:=-\tfrac{c_2z_1}{c_1z_2}, \quad B_3:=\tfrac{z_1}{c_1}+A_3^2\tfrac{z_2}{c_2},  \quad
B_4:=\tfrac{z_1}{c_1}+\tfrac{z_2}{c_2}.  
$$
We obtain from Proposition \ref{ricgbal} the following formulas.  

\begin{corollary}\label{ricgbals2}
If $s=2$, then the Ricci curvature of a normal metric $g_b=(z_1,z_2)$ on an aligned homogeneous space $M=G/K$ is given by, 
\begin{enumerate}[{\rm (i)}]
\item $\Ricci(g_b)|_{\pg_i} = \tfrac{1}{2B_{4}c_i}\cas_{\chi_i} + \tfrac{1}{4z_i}I_{\pg_i}$, \quad $i=1,2$.
\item[ ]
\item $\ricci(g_b)(\pg_i,\pg_j)=0$, \quad $1\leq i\ne j\leq 3$.  
\item[ ]
\item $\Ricci(g_b)|_{\pg_3^l}=  
\left(\tfrac{\lambda_l}{2B_{4}} + \tfrac{1}{4B_3}\left(\tfrac{1}{c_1}+A_3^2\tfrac{1}{c_{2}}\right)\right)I_{\pg_3^l}$, \quad $\quad l=0,1,\dots,t$.  
\item[ ]
\item $\ricci(g_b)(\pg_3^l,\pg_3^m)=0$, \quad $0\leq l\ne m\leq t$, 
\end{enumerate}
and the scalar curvature by, 
$$
\scalar(g_b) = \tfrac{d(2c_1c_2z_1z_2+c_2z_1^2+c_1z_2^2)}{4z_1z_2\left(c_2z_1+c_1z_2\right)}+\tfrac{n_1}{4z_1}+\tfrac{n_2}{4z_2} 
-\tfrac{c_1c_2\Lambda}{2(c_2z_1+c_1z_2)}.
$$
where $\Lambda:=\sum\limits_{l=1}^t\lambda_ld_l$ and $\Lambda=0$ if $K$ is abelian.  
\end{corollary}

\begin{remark}
It can be easily checked that this formula coincides with the given in \cite[Lemma 4.4]{HHK} in the case when $c_1=c_2=2$ and $n_1=n_2$.  
\end{remark}

\begin{proof}
Parts (i)-(iv) follow directly from Proposition \ref{ricgbal} and for the scalar curvature, we have that
\begin{align*}
\scalar(g_b) =& \tfrac{d-\sum\lambda_ld_l}{2} \tfrac{1}{\tfrac{z_1}{c_1}+\dots+\tfrac{z_s}{c_s}} 
+\unc\sum_{i=1}^s\tfrac{n_i}{z_i} 
+\tfrac{d}{4}\sum_{j=s+1}^{2s-1} 
\tfrac{\left(\tfrac{1}{c_1}+\dots+\tfrac{1}{c_{j-s}}\right)\tfrac{z_{j-s+1}^2}{c_{j-s+1}} +  
\left(\tfrac{z_1}{c_1}+\dots+\tfrac{z_{j-s}}{c_{j-s}}\right)^2}
{z_{j-s+1}\left(\tfrac{z_1}{c_1}+\dots+\tfrac{z_{j-s}}{c_{j-s}}\right)\left(\tfrac{z_1}{c_1}+\dots+\tfrac{z_{j-s+1}}{c_{j-s+1}}\right)}\\
=& \tfrac{d-\sum\lambda_ld_l}{2} \tfrac{1}{\tfrac{z_1}{c_1}+\tfrac{z_2}{c_2}} 
+\tfrac{n_1}{4z_1} +\tfrac{n_2}{4z_2}
+\tfrac{d}{4} 
\tfrac{\left(\tfrac{1}{c_1}\right)\tfrac{z_{2}^2}{c_{2}} +  
\left(\tfrac{z_1}{c_1}\right)^2}
{z_{2}\left(\tfrac{z_1}{c_1}\right)\left(\tfrac{z_1}{c_1}+\tfrac{z_{2}}{c_{2}}\right)}\\ 
=& \tfrac{d-\sum\lambda_ld_l}{2} \tfrac{c_1c_2}{c_2z_1+c_1z_2} 
+\tfrac{n_1}{4z_1} +\tfrac{n_2}{4z_2}
+\tfrac{d}{4} 
\tfrac{c_1z_{2}^2+c_2z_1^2}  
{z_{2}z_1\left(c_2z_1+c_1z_{2}\right)}\\
=& -\tfrac{\sum\lambda_ld_l}{2} \tfrac{c_1c_2}{c_2z_1+c_1z_2} 
+\tfrac{n_1}{4z_1} +\tfrac{n_2}{4z_2}
+\tfrac{d}{4}\left(\tfrac{2c_1c_2}{c_2z_1+c_1z_2} 
+\tfrac{c_1z_{2}^2+c_2z_1^2}  
{z_{2}z_1\left(c_2z_1+c_1z_{2}\right)}\right)\\
=& -\tfrac{\sum\lambda_ld_l}{2} \tfrac{c_1c_2}{c_2z_1+c_1z_2} 
+\tfrac{n_1}{4z_1} +\tfrac{n_2}{4z_2}
+\tfrac{d}{4}\left(\tfrac{2c_1c_2z_2z_1+c_1z_{2}^2+c_2z_1^2}  
{z_{2}z_1\left(c_2z_1+c_1z_{2}\right)}\right),
\end{align*}
concluding the proof.  
\end{proof}

By Remark \ref{gbz2} below, the {\it normalized scalar curvature} $\scalar_N(g_b) := \scalar(g_b)\det_{\gk}(g_b)^{\tfrac{1}{n}} $ is given by 
\begin{align*}
\scalar_N(g_b) = \scalar(g_b)z_1^{\tfrac{n_1}{n}}z_2^{\tfrac{n_2}{n}}y^{\tfrac{d}{n}}, \qquad \mbox{where}\quad y=\tfrac{(c_1+c_2)^2z_1z_2(c_2^2z_1+c_1^2z_2)}{(c_1^2+c_2^2)(c_2z_1+c_1z_2)^2}, \quad n=n_1+n_2+d.
\end{align*}
If we consider the normalization given by $z_1=z$ and $z_2=\tfrac{1}{z}$, then we obtain the one variable function
\begin{align*}
\scalar_N(z) = \left(\tfrac{d(2c_1c_2z^2+c_2z^4+c_1)}{4\left(c_2z^2+c_1\right)z}+\tfrac{n_1}{4z}+\tfrac{n_2z}{4} 
-\tfrac{c_1c_2\Lambda z}{2(c_2z^2+c_1)}\right)
z^{\tfrac{n_1-n_2}{n}}
\left(\tfrac{(c_1+c_2)^2(c_2^2z^2+c_1^2)z}{(c_1^2+c_2^2)(c_2z^2+c_1)^2}\right)^{\tfrac{d}{n}}.
\end{align*}
Thus $\scalar_N(z)$ converges to $\infty$ as $z\to 0$ or $z\to \infty$; indeed, the exponents of $z$ are respectively $-\tfrac{2(n_2+d)}{n}<0$ and $\tfrac{3n_1+n_2+d}{n}>0$.  The existence of a global minimum for $\scalar_N:\mca_1^{norm}\rightarrow\RR$ therefore follows, which was proved in \cite{HHK} to be the standard metric $\gk$ (i.e., $z=1$) in the case when $c_1=c_2=2$ and $n_1=n_2$.  It is worth noting that this can also hold for spaces other than those of the form $M=H\times H/\Delta K$ for some $H/K$ (studied in \cite{HHK}), as the following example shows.   

\begin{example}
For the family $M =\Spe(m)\times\SO(2m + 1)/S^1$, $m \ge 3$, where we are assuming that the embedding of $S^1$ satisfies $\kil_{\spg(m)}(Z_1,Z_1)=\kil_{\sog(2m+1)}(Z_2,Z_2)$ if $\kg=\RR(Z_1,Z_2)$, we have that $c_1=c_2=2$ and $n_1=n_2=m(2m+1)-1$.  The global minimum for $\scalar_N:\mca_1^{norm}\rightarrow\RR$ is therefore attained at $\gk$.  
\end{example}

\begin{example}
Consider $M^{53}=\SU(6)\times\SO(8)/\Spe(2)$, so $n_1=25$, $n_2=18$, $d=10$, $c_1=c_2=2$ and $\Lambda=\tfrac{5}{2}$ (see \cite[Example 2.8]{Es2}).  It is easy to see that 
$$
\scalar_N(z)= \tfrac{2^\frac{10}{53}(28 z^4 + 73 z^2 + 35)}{4(z^2 + 1)^{\frac{63}{53}} z^\frac{36}{53}} \qquad\mbox{and}\qquad 
\scalar_N'(z)=\tfrac{2^\frac{10}{53} 35 (10 z^6 + 6 z^4 - 4 z^2 - 9) }{53 (z^2 + 1)^\frac{116}{53} z^\frac{89}{53}}.
$$
Thus $\scalar_N'$ has only one positive zero, approximately given by $0.9561$, which is therefore the global minimum and unique critical point of $\scalar_N:\mca_1^{norm}\rightarrow\RR$.  Note that the minimal metric is different from $\gk$.  
\end{example}

\begin{example}
For the family $M = \SU(m) \times \SO(m + 1) /  \SO(m)$, $m \ge 6$, we have that 
$$
n=m^2+m-1,  \quad n_1=\tfrac{(m-1)(m+2)}{2}, \quad n_2=m, \quad d=\tfrac{m(m-1)}{2}, 
$$
$$
c_1=\tfrac{3m-1}{2m}, \quad \lambda=\tfrac{m-2}{3m-1}, \quad \Lambda= \tfrac{m(m - 2)(m - 1)}{(6m - 2)},
$$
and it is straightforward to see that 
$$
\scalar_N(z) = 
\tfrac{ (m^3 + m^2)z^4 + m (3m^2 + m - 4) z^2 + m^3 - m^2 - m + 1 }
{4(2m z^2 + m - 1) z^{\tfrac{2m}{m^2 + m - 1}}}
\left( \tfrac{(3m - 1)^2 (4m^2 z^2 + m^2 - 2m + 1)}{ (2m z^2 + m - 1)^2 (5m^2 - 2m + 1)}\right)^{\tfrac{m(m - 1)}{2m^2 + 2m - 2}}. 
$$
Its derivative vanishes at a point $z_m\in\RR$ if and only if $z_m^2$ is a zero of the quartic polynomial 
\begin{align*} 
p(z)=&4m^4(m + 2) z^4 -2m^2 (m^3 + 6m^2 + 6m - 1) z^3 \\ 
&+ m(2m^4 - 11m^3 - 2m^2 + 13m - 2)z^2 \\ 
&+ (m^5 - 8m^4 + 12m^3 - 2m^2 - 5m + 2) z - ( m^4 - 4m^3 + 6m^2 - 4m + 1).
\end{align*}
Since $p(0)<0$, $p'(0)>0$ and $p''(z)>0$ for all $z>0$, there is only one positive zero $z_m^2$ for $p$, which converges to $0$ as $m\to\infty$; indeed, $p(\frac{1}{m-6})>0$ for any $m$.  The global minimum $g_b(z_m)$ of $\scalar_N:\mca_1^{norm}\rightarrow\RR$ is therefore the unique critical point and goes away from $\gk$ as $m\to\infty$. 
\end{example}

The same behavior as in the above two examples was detected for the remaining six aligned homogeneous spaces listed in \cite[Table 3]{Es2}, for which the approximate values for the global minima are respectively given by 
$$
0.6884, \quad 0.8608, \quad 0.9491, \quad 1.1131, \quad 0.9951, \quad 2.5858.
$$
The following natural question remains open: is $\gk$ the global minimum of $\scalar_N:\mca_1^{norm}\rightarrow\RR$ if and only if $c_1=c_2=2$ and $n_1=n_2$?

\section{Diagonal metrics}\label{diag-sec} 

Let $M=G/K$ be an aligned homogeneous space with positive constants $c_1,\dots,c_s$ and $\lambda_1,\dots,\lambda_t$ as in \S\ref{preli3}.  We compute in this section the Ricci curvature of a large class of $G$-invariant metrics on $M$, keeping an eye on the Einstein condition.    

For each normal metric 
$
g_b=(z_1,\dots,z_s) 
$
on $M$ as in \S\ref{normal-sec}, we consider the $\Ad(K)$-invariant $g_b$-orthogonal decompositions (see \eqref{dec1} and \eqref{pjdec})
$$
\pg=\pg_1\oplus\dots\oplus\pg_s\oplus\pg_{s+1}\oplus\dots\oplus\pg_{2s-1}, \qquad 
\pg_j=\pg_j^0\oplus\dots\oplus\pg_j^t, \quad \forall s<j,
$$
and a $G$-invariant metric of the form
\begin{equation*}
g|_{\pg_1\oplus\dots\oplus\pg_s}=x_1g_b|_{\pg_1}+\dots+x_sg_b|_{\pg_s}, \qquad
g|_{\pg_j}=x_{j,0}g_b|_{\pg_j^0}+\dots+x_{j,t}g_b|_{\pg_j^t}, \quad \forall s<j,
\end{equation*} 
where $x_1,\dots,x_{j,t}>0$, which will be denoted by 
\begin{equation}\label{metg}
g=(x_1,\dots,x_s,x_{s+1,0},\dots,x_{s+1,t},\dots,x_{2s-1,0},\dots,x_{2s-1,t})_{g_b}.  
\end{equation}   
This determines, for any $g_b$, a submanifold $\mca^{diag}_{g_b}\subset\mca^G$ of dimension $s+(s-1)(t+1)$ if $\zg(\kg)\ne 0$ and $s+(s-1)t$ if $\kg$ is semisimple, where $\mca^G$ is the manifold of all $G$-invariant metrics on $M=G/K$.      

\begin{remark}\label{param}
The union of all spaces $\mca_{g_b}^{diag}$ over all normal metrics $g_b$ may be far from exhausting $\mca^G$.  Indeed, each $\pg_i$, $i\leq s$ can be $\Ad(K)$-reducible, non-mutliplicity-free and share irreducible components with another $\pg_j$.  Moreover, since the representations $\pg_{s+1},\dots,\pg_{2s-1}$ are pairwise equivalent, there is always an $\frac{(s-1)(s-2)}{2}$-parameter set of non-diagonal metrics in $\mca^G$ for each fixed $g_b$.  The special case when $\mca^{diag}_{\gk}=\mca^G$ will be studied in \S\ref{class1}.  
\end{remark}

According to \cite[7.38]{Bss}, given any reductive decomposition $\ggo=\kg\oplus\pg$ for a compact homogeneous space $M=G/K$, the Ricci tensor $\ricci(g):\pg\times\pg\rightarrow\RR$ of a $G$-invariant metric on $M$ is given by
\begin{align}
\ricci(g)(X,Y) =& -\unm\sum_{i,j} g([X,X_i]_\pg,X_j)g([Y,X_i]_\pg,X_j) \label{Rc}\\ 
&+ \unc\sum_{i,j} g([X_i,X_j]_\pg,X)g([X_i,X_j]_\pg,Y) -\unm\kil_\ggo(X,Y), \qquad\forall X,Y\in\pg, \notag
\end{align} 
where $\{ X_i\}$ is any $g$-orthonormal basis of $\pg$, $\lb_\pg$ denotes the projection of the Lie bracket of $\ggo$ on $\pg$ relative to $\ggo=\kg\oplus\pg$ and $\kil_{\ggo}$ is the Killing form of the Lie algebra $\ggo$.  For the computation of the Ricci curvature of $g$, we consider the $g$-orthonormal bases $\{ \tfrac{1}{\sqrt{x_k}}e^k_\alpha\}$ and $\{ \tfrac{1}{\sqrt{x_{j,l}}}e^{j,l}_\alpha\}$ of $\pg_k$, $k\leq s$ and $\pg_j^l$, $s<j$, respectively (see the end of \S\ref{preli3}).  

We denote by $\cas_{\chi_{i,l}}$ the Casimir operator of the isotropy representation of the homogeneous space $M_{i,l}:=G_i/\pi_i(K_l)$ with respect to $-\kil_{\ggo}$ (recall from \eqref{Cchi} the definitions of $\cas_{\chi_i}$ and $\cas_{\chi}$).  Thus 
\begin{equation}\label{Cchiil}
\cas_{\chi_i} = \sum_{l=0}^t \cas_{\chi_{i,l}}|_{\pg_i}:\pg_i\rightarrow\pg_i,  \qquad\forall i=1,\dots,s,
\end{equation}
and note that $\cas_{\chi_{i,l}}$ is nonzero for any $l$ since $\ggo_i$ is simple.  Recall the definition of the numbers $A_i$'s and $B_j$'s from \S\ref{rd-sec}.   

\begin{proposition}\label{ricggal}
The Ricci operator of the metric $g$ given in \eqref{metg} satisfies that 
$$
\Ricci(g)|_{\pg_k} 
=  \tfrac{1}{2x_k} \sum_{l=0}^t  \left(\tfrac{1}{z_k} - \tfrac{1}{x_kc_k} 
\left(\tfrac{x_{k+s-1,l}}{B_{k+s-1}} A_{k+s-1}^2+ \tfrac{x_{k+s,l}}{B_{k+s}}+\dots+\tfrac{x_{2s-1,l}}{B_{2s-1}}\right)\right) \cas_{\chi_{k,l}}|_{\pg_k} 
+ \tfrac{1}{4x_kz_k}I_{\pg_k}, 
$$ 
for any $k\leq s$ and $\ricci(g)(\pg_i,\pg_j)=0$ for all $i\ne j\leq s$.  
\end{proposition}

\begin{remark}\label{ricort}
If for any $k\leq s$, $\pg_k$ does not contain irreducible components equivalent to the adjoint representations $\kg_1,\dots,\kg_t$ (or the trivial representation if $\kg_0\ne 0$), then $\ricci(g)(\pg_k,\pg_j)=0$ for all $k\leq s<j$.  It is worth noting that the proposition gives that $\ricci(g)(\pg_i,\pg_k)=0$ for all $i\ne k\leq s$ regardless $\pg_i$ and $\pg_k$ contains or not equivalent irreducible components.  Moreover, if we consider any decomposition $\pg_k=\pg_k^1\oplus\dots\oplus\pg_k^{q_k}$ in $K$-irreducible summands, then it also follows from the proposition that $\ricci(g)(\pg_k^j,\pg_k^m)=0$ for all $j\ne m$ (note that $\pg_k^j$ is $\cas_{\chi_{k,l}}$-invariant for any $j$ and $l$), and since $\cas_{\chi_{k,l}}|_{\pg_k^j}=\kappa_{k,l}^j I|_{\pg_k^j}$ for some $\kappa_{k,l}^j\in\RR$ (see Lemma \ref{casil} below), a formula for the corresponding Ricci eigenvalues follows.    
\end{remark}

\begin{proof} 
The last statement follows from \eqref{Rc} and \eqref{pipj0}-\eqref{pipj4}.  If $X\in\pg_k$, $k\leq s$, then by \eqref{Rc}, 
\begin{align*}
g(\Ricci(g)X,X) 
=& -\unm\sum_{\substack{\alpha,\beta}} \tfrac{1}{x_k^2}g([X,e^k_\alpha]_\pg,e^k_\beta)^2 
+\unc\sum_{\substack{\alpha,\beta}} \tfrac{1}{x_k^2}g([e^k_\alpha,e^k_\beta]_\pg,X)^2 \\ 
& -\unm\sum_{\substack{\alpha,\beta\\ s<j;\, l}} \tfrac{1}{x_kx_{j,l}}g([X,e^k_\alpha]_\pg,e^{j,l}_\beta)^2 
+\unm\sum_{\substack{\alpha,\beta\\ s<j;\, l}} \tfrac{1}{x_kx_{j,l}}g([e^k_\alpha,e^{j,l}_\beta]_\pg,X)^2 \\ 
& -\unm\sum_{\substack{\alpha,\beta\\ s<j;\, l}} \tfrac{1}{x_kx_{j,l}}g([X,e^{j,l}_\beta]_\pg,e^k_\alpha)^2 
- \unm\kil_\ggo(X,X) \\ 
=& -\unm\sum_{\substack{\alpha,\beta}} g_b([X,e^k_\alpha],e^k_\beta)^2 
+\unc\sum_{\substack{\alpha,\beta}}g_b([e^k_\alpha,e^k_\beta],X)^2 
- \unm\kil_{\ggo_k}(X,X) \\ 
& -\unm\sum_{\substack{\alpha,\beta\\ s<j;\, l}} \tfrac{x_{j,l}}{x_k}g_b([X,e^k_\alpha],e^{j,l}_\beta)^2 
+\unm\sum_{\substack{\alpha,\beta\\ s<j;\, l}} \tfrac{x_k}{x_{j,l}}g_b([e^k_\alpha,e^{j,l}_\beta],X)^2 \\ 
& -\unm\sum_{\substack{\alpha,\beta\\ s<j;\, l}} \tfrac{x_k}{x_{j,l}}g_b([X,e^{j,l}_\beta],e^k_\alpha)^2. 
\end{align*}
This implies that
\begin{align*}
x_kg_b(\Ricci(g)X,X) =& g_b^k(\Ricci(g_b^k)X,X) 
-\tfrac{1}{2x_k}\sum_{\substack{s<j;\, l}} x_{j,l} \sum_{\substack{\alpha,\beta}} g_b([X,e^k_\alpha],e^{j,l}_\beta)^2 \\ 
=& g_b(\Ricci(g_b^k)X,X) 
-\tfrac{1}{2x_k}\sum_{\substack{k\leq j-s;\, l}} \tfrac{x_{j,l}}{B_j} \sum_{\substack{\alpha,\beta}} g_b([X,e^k_\alpha],Z^\beta)^2 \\ 
& -\tfrac{1}{2x_k}\sum_l \tfrac{x_{k+s-1,l}A_{k+s-1}^2}{B_{k+s-1}} \sum_{\substack{\alpha,\beta}} g_b([X,e^k_\alpha],Z^{l,\beta})^2,  
\end{align*}
and since by \eqref{Cchi},
\begin{align}
\sum_{\substack{\alpha,\beta}} g_b([X,e^k_\alpha],Z^{l,\beta})^2 
=& g_b\Big(-\sum_{\substack{\beta}} (\ad{Z^{l,\beta}}|_{\pg_k})^2X,X\Big) = g_b(\cas_{\chi^l}|_{\pg_k}X,X) \label{chik} \\ 
=& g_b(\tfrac{1}{c_k}\cas_{\chi_{k,l}}X,X), \notag
\end{align}
we obtain
\begin{align*}
x_kg_b(\Ricci(g)X,X) =
& g_b(\Ricci(g_b^k)X,X) 
-\tfrac{1}{2x_k}\sum_{\substack{k\leq j-s;\, l}} \tfrac{x_{j,l}}{B_j} g_b(\tfrac{1}{c_k}\cas_{\chi_{k,l}}X,X) \\ 
& -\tfrac{1}{2x_k}\sum_l \tfrac{x_{k+s-1,l}A_{k+s-1}^2}{B_{k+s-1}} g_b(\tfrac{1}{c_k}\cas_{\chi_{k,l}}X,X), 
\end{align*}
and so
\begin{align*}
x_k\Ricci(g)|_{\pg_k} = \Ricci(g_b^k) - \tfrac{1}{2x_kc_k} \sum_l
\left(\tfrac{x_{k+s-1,l}}{B_{k+s-1}} A_{k+s-1}^2+ \tfrac{x_{k+s,l}}{B_{k+s}}+\dots+\tfrac{x_{2s-1,l}}{B_{2s-1}}\right) \cas_{\chi_{k,l}}|_{\pg_k}. 
\end{align*}
But $\Ricci(g_b^k) = \tfrac{1}{z_k}\Ricci(\gk^k) = \tfrac{1}{2z_k}\cas_{\chi_k}+\tfrac{1}{4z_k}I_{\pg_k}$, thus the formula for $\Ricci(g)|_{\pg_k} $ given by the proposition follows.  
\end{proof}

The formula given in Proposition \ref{ricggal} provides the following necessary condition for the existence of an Einstein diagonal metric on an aligned homogeneous space.   

\begin{corollary}\label{ricggal-cor}
If a diagonal metric $g$ is Einstein, then each homogeneous space $G_i/\pi_i(K)$ satisfies that 
$$
I_{\pg_i}\in\spann\{\cas_{\chi_{i,0}}|_{\pg_i},\dots,\cas_{\chi_{i,t}}|_{\pg_i}\}, \qquad\forall i=1,\dots,s.
$$
\end{corollary}

For any fixed $i\in\{ 1,\dots,s\}$, consider any decomposition 
$$
\pg_i=\pg_i^1\oplus\dots\oplus\pg_i^{q_i}, 
$$ 
in $K$-irreducible summands.  Recall that all the Casimir operators are defined with respect to $-\kil_\ggo$.

\begin{lemma}\label{casil}
For any $l=0,1,\dots,t$ and $j=1,\dots,q_i$, the subspace $\pg_i^j$ is $\cas_{\chi_{i,l}}$-invariant and
$$
\cas_{\chi_{i,l}}|_{\pg_i^j}=\kappa_{i,l}^jI_{\pg_i^j}, \qquad\mbox{for some numbers}\quad \kappa_{i,l}^j\geq 0,  
$$
which satisfy that  
$$
(1-a_{il})d_l=\kappa_{i,l}^1\dim{\pg_i^1}+\dots+\kappa_{i,l}^{q_i}\dim{\pg_i^{q_i}}, \qquad \forall l=0,1,\dots,t.
$$  
where $d_l=\dim{K_l}$.  In particular, when $\pg_i$ is $K$-irreducible, $\cas_{\chi_{i,l}}|_{\pg_i}=\kappa_{i,l}I_{\pg_i}$ for the positive number $\kappa_{i,l}=\tfrac{(1-a_{il})d_l}{n_i}$.
\end{lemma}

\begin{proof}  
As well known, for each $j$, $\pg_i^j\otimes\CC=\mg_{i,0}^j\otimes\dots\otimes\mg_{i,t}^j$, where $\mg_{i,l}^j$ is an irreducible complex representation of $K_l$ (see e.g.\ \cite[Proposition 4.14]{BrcDck} or \cite[Theorem 3.9]{Spn}).  In particular, $\cas_{\mg_{i,l}^j}=\kappa_{i,l}^jI_{\mg_{i,l}^j}$ for some $\kappa_{i,l}^j\in\RR$.    Since $\pg_i^j\otimes\CC=\mg_{i,l}^j\oplus\dots\oplus\mg_{i,l}^j$ as a $K_l$-representation, we obtain that the real representation $\pg_i^j$ can be decomposed in $K_l$-irreducible summands as $\pg_i^j=\vg_{i,l}^j\oplus\dots\oplus\vg_{i,l}^j$, where either $\mg_{i,l}^j=\vg_{i,l}^j\otimes\CC$ or $\mg_{i,l}^j\oplus\mg_{i,l}^j=\vg_{i,l}^j\otimes\CC$, depending on whether $\mg_{i,l}^j$ is of complex type or of real or quaternionic type, respectively (see \cite[(6.2)]{BrcDck}).  This implies that $\cas_{\chi_{i,l}}|_{\pg_i^j}=\kappa_{i,l}^jI_{\pg_i^j}$ for all $j=1,\dots,q_i$, as was to be shown.  

On the other hand, the isotropy representation of $G_i/\pi_i(K_l)$ is 
$$
\pg_i\oplus\kg_0\oplus\dots\oplus\hat{\kg_l}\oplus\dots\oplus\kg_t,
$$
but $\cas_{\chi_{i,l}}$ vanishes on $\kg_1\oplus\dots\oplus\hat{\kg_l}\oplus\dots\oplus\kg_t$ as $[\kg_l,\kg_k]=0$ for all $k\ne l$, hence  the remaining formulas follow from \cite[(9)]{stab}, concluding the proof.  
\end{proof}

It follows from Lemma \ref{casil} that the necessary condition given in Corollary \ref{ricggal-cor} is equivalent to have that, for any $i=1,\dots,s$, 
\begin{equation}\label{Kappas}
(1,\dots,1)\in\spann\{\kca_{i,0},\dots,\kca_{i,t}\}, \qquad \mbox{where}\quad \kca_{i,l}:=(\kappa_{i,l}^1,\dots,\kappa_{i,l}^{q_i})\in\RR^{q_i}.
\end{equation}
A classification of all the homogeneous spaces satisfying this seems to be out of reach.  The larger the number of irreducible summands, the stronger the condition.  

We note that \eqref{Kappas} clearly holds when the standard metric on each $G_i/\pi_i(K)$ is Einstein, i.e., $\cas_{\chi_i}=\kappa_iI_{\pg_i}$ for some $\kappa_i\in\RR$, which by Lemma \ref{casil} it is a necessary condition for the diagonal metric $g$ being Einstein when $x_{j,0}=\dots=x_{j,t}$ for all $s<j$ (cf.\ \cite{HHK, Es2}).  Indeed, $\cas_{\chi_i}=\kappa_iI_{\pg_i}$ if and only if  
$$
\kappa_{i,0}^j+\dots+\kappa_{i,t}^j=\kappa_i, \qquad\forall j=1,\dots,q_i,  
$$
that is, $\kca_{i,0}+\dots+\kca_{i,t}=(\kappa_i,\dots,\kappa_i)$.  Such spaces were classified in \cite{WngZll2}.

\section{The case of $M=G_1\times G_2/K$}\label{s2-sec}

We focus in this section on aligned homogeneous spaces $M=G/K$ with $s=2$.  Thus  
\begin{equation*}
\ggo=\ggo_1\oplus\ggo_2, \qquad \kg=\kg_0\oplus\kg_1\oplus\dots\oplus\kg_t, 
\end{equation*}
and for any bi-invariant metric
\begin{equation*}
g_b=z_1(-\kil_{\ggo_1})+z_2(-\kil_{\ggo_2}), \qquad z_1,z_2>0,   
\end{equation*}
we consider the $g_b$-orthogonal reductive decomposition $\ggo=\kg\oplus\pg$ and the $g_b$-orthogonal $\Ad(K)$-invariant decompositions 
\begin{equation*}
\pg=\pg_1\oplus\pg_2\oplus\pg_3, \qquad \pg_3=\pg_3^0\oplus\pg_3^1\oplus\dots\oplus\pg_3^t,
\end{equation*}
where
\begin{equation*}
\pg_3 := \left\{ \left(Z_1,-\tfrac{c_2z_1}{c_1z_2}Z_2\right):Z\in\kg\right\}, \qquad \pg_3^l := \left\{ \left(Z_1,-\tfrac{c_2z_1}{c_1z_2}Z_2\right):Z\in\kg_l\right\}, \quad l=0,1,\dots,t.
\end{equation*}
Properties \eqref{pipj0}-\eqref{pipj4} simplify as follows,
\begin{align}
&[\pg_1,\pg_1]\subset\pg_1+\pg_3+\kg, \label{pipj0s2}\\  
&[\pg_2,\pg_2]\subset\pg_2+\pg_3+\kg, \label{pipj1s2}\\  
&[\pg_3,\pg_1]\subset\pg_1, \qquad [\pg_3,\pg_2]\subset\pg_2, \label{pipj2s2}, \\ 
&[\pg_3,\pg_3] \subset\pg_3+\kg, \label{pipj3s2}
\end{align}
and 
$$
c_2=\tfrac{c_1}{c_1-1}, \qquad A_3:=-\tfrac{c_2z_1}{c_1z_2}, \qquad B_3:=\tfrac{z_1}{c_1}+A_3^2\tfrac{z_2}{c_2},  \qquad
B_4:=\tfrac{z_1}{c_1}+\tfrac{z_2}{c_2}.  
$$

\subsection{Ricci curvature}
We give a formula for the full Ricci operator of a diagonal metric in this case.  

\begin{proposition}\label{ricggals2}
If $s=2$, then the Ricci curvature of a diagonal metric 
$$
g=(x_1,x_2,x_{3,0},\dots,x_{3,t})_{g_b},
$$ 
on an aligned homogeneous space $M=G_1\times G_2/K$ with positive constants $c_1,c_2,\lambda_1,\dots,\lambda_t$ is given by 
\begin{enumerate}[{\rm (i)}]
\item $\Ricci(g)|_{\pg_1} = 
\tfrac{1}{2x_1} \sum\limits_{l=0}^t \left(\tfrac{1}{z_1} - \tfrac{x_{3,l}}{x_1c_1B_{3}}\right) \cas_{\chi_{1,l}}|_{\pg_1}
+ \tfrac{1}{4x_1z_1}I_{\pg_1}$.  
\item[ ]
\item $\Ricci(g)|_{\pg_2} = 
\tfrac{1}{2x_2} \sum\limits_{l=0}^t\left(\tfrac{1}{z_2} - \tfrac{x_{3,l}}{x_2c_2B_{3}} A_{3}^2\right) \cas_{\chi_{2,l}}|_{\pg_2} 
+ \tfrac{1}{4x_2z_2}I_{\pg_2} $.
\item[ ]
\item The decomposition $\pg=\pg_1\oplus\pg_2\oplus\pg_3^0\oplus\dots\oplus\pg_3^t$ is $\ricci(g)$-orthogonal.
\item[ ]
\item $\Ricci(g)|_{\pg_3^l}= r_3^lI_{\pg_3^l}$, $\quad l=0,1,\dots,t$, where 
\begin{align*}
r_3^l :=& \tfrac{\lambda_l}{4x_{3,l}B_3}\left(\tfrac{2x_1^2-x_{3,l}^2}{x_1^2} 
+\tfrac{(2x_2^2-x_{3,l}^2)A_3^2}{x_2^2}  
-\tfrac{1+A_3}{B_3}\left(\tfrac{z_1}{c_1}+\tfrac{z_2}{c_2}A_3^3\right)\right) \\ 
&+\tfrac{1}{4x_{3,l}B_3}\left(2\left(\tfrac{1}{c_1}+\tfrac{1}{c_2}A_3^2\right) 
-\tfrac{2x_1^2-x_{3,l}^2}{x_1^2c_1} 
- \tfrac{(2x_2^2-x_{3,l}^2)A_3^2}{x_2^2c_2}\right).
\end{align*} 
\end{enumerate}
\end{proposition}

\begin{proof}  
Parts (i)-(iii) follow from Proposition \ref{ricggal}.  To prove (iv), we follow the lines of the proof of \cite[Proposition 3.2,(iv)]{BRF}.  Consider $X\in\pg_3^l$.  By \eqref{Rc} and \eqref{pipj0s2}-\eqref{pipj3s2}, 
\begin{align*}
& g(\Ricci(g)X,X)  \\  
=& -\unm\sum_{\substack{\alpha,\beta\\ k=1,2,3}} \tfrac{1}{x_k^2}g([X,e^k_\alpha]_\pg,e^k_\beta)^2 
+\unc\sum_{\substack{\alpha,\beta\\ k=1,2,3}} \tfrac{1}{x_k^2}g([e^k_\alpha,e^k_\beta]_\pg,X)^2 -\unm\kil_\ggo(X,X)\\ 
=& -\unm\sum_{\substack{\alpha,\beta\\ k=1,2,3}} g_b([X,e^k_\alpha],e^k_\beta)^2 
+\unc\sum_{\substack{\alpha,\beta\\ k=1,2,3}} \tfrac{x_{3,l}^2}{x_k^2}g_b([e^k_\alpha,e^k_\beta],X)^2 -\unm\kil_\ggo(X,X), \\ 
=& \sum_{\substack{k=1,2}} \tfrac{2x_k^2-x_{3,l}^2}{4x_k^2} \tr{(\ad{X}|_{\pg_k})^2} 
-\unm\kil_\ggo(X,X) + \unc\tr{(\ad_\pg{X}|_{\pg_3})^2} \\ 
=& \tfrac{2x_1^2-x_{3,l}^2}{4x_1^2} \tr{(\ad{X}|_{\pg_1})^2} 
+ \tfrac{2x_2^2-x_{3,l}^2}{4x_2^2} \tr{(\ad{X}|_{\pg_2})^2} 
+  \unc\tr{(\ad_\pg{X}|_{\pg_3})^2} -\unm\kil_\ggo(X,X), 
\end{align*}
and if $X=(Z_1,A_3Z_2)$, $Z\in\kg_l$, then $\kil_\ggo(X,X) = -\left(\tfrac{1}{c_1}+\tfrac{1}{c_2}A_3^2\right) \la Z,Z\ra$,
$$
\tr{(\ad{X}|_{\pg_1})^2} = -\tfrac{1-\lambda_lc_1}{c_1}\la Z,Z\ra, \quad
\tr{(\ad{X}|_{\pg_2})^2} = -\tfrac{1-\lambda_lc_2}{c_2}A_3^2\la Z,Z\ra, \quad
g_b(X,X) = B_3\la Z,Z\ra,
$$
and
\begin{equation}\label{trad3}
\tr{(\ad_\pg{X}|_{\pg_3})^2} =-\tfrac{1+A_3}{B_3}\left(\tfrac{z_1}{c_1}+\tfrac{z_2}{c_2}A_3^3\right)\lambda_l\la Z,Z\ra. 
\end{equation}

All this implies that
\begin{align*}
x_{3,l}g_b(\Ricci(g)X,X) 
=& -\tfrac{2x_1^2-x_{3,l}^2}{4x_1^2} \tfrac{1-\lambda_lc_1}{c_1}\la Z,Z\ra 
- \tfrac{2x_2^2-x_{3,l}^2}{4x_2^2}A_3^2 \tfrac{1-\lambda_lc_2}{c_2}\la Z,Z\ra \\
&-  \tfrac{1+A_3}{4B_3}\left(\tfrac{z_1}{c_1}+\tfrac{z_2}{c_2}A_3^3\right)\lambda_l \la Z,Z\ra
+ \tfrac{1}{2B_3}\left(\tfrac{1}{c_1}+\tfrac{1}{c_2}A_3^2\right) g_b(X,X)\\ 
=& -\tfrac{(2x_1^2-x_{3,l}^2)(1-\lambda_lc_1)}{4x_1^2c_1B_3} g_b(X,X) 
- \tfrac{(2x_2^2-x_{3,l}^2)(1-\lambda_lc_2)A_3^2}{4x_2^2c_2B_3} g_b(X,X) \\
&-  \tfrac{1+A_3}{4B_3^2}\left(\tfrac{z_1}{c_1}+\tfrac{z_2}{c_2}A_3^3\right)\lambda_l g_b(X,X)
+ \tfrac{1}{2B_3}\left(\tfrac{1}{c_1}+\tfrac{1}{c_2}A_3^2\right) g_b(X,X),
\end{align*}
from which the formula for $r_3^l$ follows.  
\end{proof}

Let us check that the formula in Corollary \ref{ricgbals2} (iii) coincides with the one given in Proposition \ref{ricggals2}, (iv) in the case when $g=g_b$ (i.e., $x_1=x_2=x_{3,0}=\dots=x_{3,t}=1$):   
$$  
\tfrac{\lambda_l}{2B_{4}} + \tfrac{1}{4B_3}\left(\tfrac{1}{c_1}+A_3^2\tfrac{1}{c_{2}}\right)
$$
must be equal to 
\begin{align*}
& \tfrac{\lambda_l}{4B_3}\left(1 
+A_3^2  
-\tfrac{1+A_3}{B_3}\left(\tfrac{z_1}{c_1}+\tfrac{z_2}{c_2}A_3^3\right)\right) 
+\tfrac{1}{4B_3}\left(2\left(\tfrac{1}{c_1}+\tfrac{1}{c_2}A_3^2\right) 
-\tfrac{1}{c_1} 
- \tfrac{A_3^2}{c_2}\right) \\
=&
\tfrac{\lambda_l}{4B_3}\left(1 
+A_3^2  
-\tfrac{1+A_3}{B_3}\left(\tfrac{z_1}{c_1}+\tfrac{z_2}{c_2}A_3^3\right)\right) 
+\tfrac{1}{4B_3}\left(\tfrac{1}{c_1}+\tfrac{1}{c_2}A_3^2\right). 
\end{align*}
This is equivalent to 
$$
\tfrac{2B_3}{B_4}=1 +A_3^2  -\tfrac{1+A_3}{B_3}\left(\tfrac{z_1}{c_1}+\tfrac{z_2}{c_2}A_3^3\right),
$$
which is straightforward to check.

\subsection{Standard diagonalization}
We now show that in the case when $s=2$, in order to cover all diagonal metrics, it is sufficient to consider the standard metric $\gk$ (i.e., $z_1=z_2=1$) as a background metric (cf.\ \cite[Lemma 4.1]{HHK}).  

\begin{lemma}\label{gz1z2}
For any $z_1,z_2>0$, 
$$
\left(x_1,x_2,x_{3,0},\dots,x_{3,t}\right)_{g_b} = \left(z_1x_1,z_2x_2,yx_{3,0},\dots,yx_{3,t}\right)_{\gk},
$$
where
$$
y=y(z_1,z_2,c_1):=\tfrac{(c_1+c_2)^2z_1z_2(c_2^2z_1+c_1^2z_2)}{(c_1^2+c_2^2)(c_2z_1+c_1z_2)^2} 
=\tfrac{c_1^2}{(c_1-1)^2+1}\tfrac{z_1z_2(z_1+(c_1-1)^2z_2)}{(z_1+(c_1-1)z_2)^2}.
$$
\end{lemma}

\begin{remark}\label{gbz2}
In particular, the normal metric $g_b=(z_1,z_2)$ is given by  
$$
g_b= \left(z_1,z_2,y,\dots,y\right)_{\gk},   
$$
and if $c_1=c_2=2$, then $y=\tfrac{2z_1z_2}{z_1+z_2}$, as proved in \cite[Lemma 4.1]{HHK} for $M=H\times H/\Delta K$.  
\end{remark}

\begin{proof}
Let $\ggo=\kg\oplus\pg$ and $\ggo=\kg\oplus\mg$ be the orthogonal reductive decompositions with respect to $g_b$ and $\gk$, respectively, and let $\hat{g}$ denote the $\Ad(K)$-invariant inner product on $\mg$ identified with the metric $g:=\left(x_1,x_2,x_{3,0},\dots,x_{3,t}\right)_{g_b}\in\mca^G$.  If $T:\ggo\rightarrow\pg$ is the $g_b$-orthogonal projection, then 
\begin{equation}\label{gz1}
\hat{g}(X,Y) = g(TX,TY), \qquad\forall X,Y\in\pg.  
\end{equation}
Using \eqref{gbpi} and the orthogonal decompositions 
$$
\mg=\pg_1\oplus\pg_2\oplus\mg_3^0\oplus\dots\oplus\mg_3^t, \qquad  \pg=\pg_1\oplus\pg_2\oplus\pg_3^0\oplus\dots\oplus\pg_3^t,
$$
where $\mg_3^l:=\{(Z_1,-\tfrac{c_2}{c_1}Z_2):Z\in\kg_l\}$ and $\pg_3^l:=\{(Z_1,-\tfrac{c_2z_1}{c_1z_2}Z_2):Z\in\kg_l\}$, we obtain that $T|_{\pg_1\oplus\pg_2}=I$ and 
$$
T(Z_1,-\tfrac{c_2}{c_1}Z_2)=\tfrac{(c_1+c_2)z_2}{c_1z_2+c_2z_1}(Z_1,-\tfrac{c_2z_1}{c_1z_2}Z_2), \qquad\forall Z\in\kg,
$$
since 
$$
(Z_1,-\tfrac{c_2}{c_1}Z_2) = \tfrac{c_2(z_1-z_2)}{c_1z_2+c_2z_1}(Z_1,Z_2) + 
\tfrac{(c_1+c_2)z_2}{c_1z_2+c_2z_1}(Z_1,-\tfrac{c_2z_1}{c_1z_2}Z_2).  
$$
It now follows from \eqref{gz1} that $\pg_1$, $\pg_2$ and $\mg_3$ are $\hat{g}$-orthogonal, i.e., $\hat{g}=\left(y_1,y_2,y_{3,0},\dots,y_{3,t}\right)_{\gk}$, 
\begin{align*}
y_1\gk(X,X)=\hat{g}(X,X) =& x_1g_b(X,X) = x_1z_1\gk(X,X), \qquad\forall X\in\pg_1,\\ 
y_2\gk(X,X)=\hat{g}(X,X) =& x_2g_b(X,X) = x_2z_2\gk(X,X), \qquad\forall X\in\pg_2, 
\end{align*}
and for any $X=(Z_1,-\tfrac{c_2}{c_1}Z_2)\in\mg_3^l$, $Z\in\kg_l$,
\begin{align*}
\hat{g}(X,X) =& g(T_zX,T_zX) = \tfrac{(c_1+c_2)^2z_2^2}{(c_1z_2+c_2z_1)^2}x_{3,l}g_b\left((Z_1,-\tfrac{c_2z_1}{c_1z_2}Z_2),(Z_1,-\tfrac{c_2z_1}{c_1z_2}Z_2)\right) \\ 
=&\tfrac{(c_1+c_2)^2z_2^2}{(c_1z_2+c_2z_1)^2}x_{3,l}\left(z_1+z_2\tfrac{c_2^2z_1^2}{c_1^2z_2^2}\right)\gk(Z,Z)\\
=&\tfrac{(c_1+c_2)^2z_2^2}{(c_1z_2+c_2z_1)^2}x_{3,l}\tfrac{c_1^2z_2^2z_1+z_2c_2^2z_1^2}{c_1^2z_2^2}\gk(Z,Z)\\
=&\tfrac{(c_1+c_2)^2z_2^2}{(c_1z_2+c_2z_1)^2}x_{3,l}\tfrac{z_1z_2(c_1^2z_2+c_2^2z_1)}{c_1^2z_2^2}\gk(Z,Z)\\
=&\tfrac{(c_1+c_2)^2z_1z_2(c_1^2z_2+c_2^2z_1)}{(c_1z_2+c_2z_1)^2c_1^2}x_{3,l}\gk(Z,Z).
\end{align*}
On the other hand, 
$$
\hat{g}(X,X) = y_3^l\gk(X,X) = y_3^l\left(1+\tfrac{c_2^2}{c_1^2}\right)\gk(Z,Z) 
= \tfrac{c_1^2+c_2^2}{c_1^2}y_3^l\gk(Z,Z),
$$
which gives $y_3^l=\tfrac{(c_1+c_2)^2}{c_1^2+c_2^2}\tfrac{(c_1^2z_2+c_2^2z_1)z_1z_2}{(c_1z_2+c_2z_1)^2}x_{3,l}$ for any $l$, concluding the proof.  
\end{proof}

\begin{remark}
For $s\geq 3$, the $g_b$-orthogonal projection $T:\ggo\rightarrow\pg$ satisfies that $T\mg_j^l\subset\pg_j^l\oplus\dots\pg_{2s-1}^l$ for any $s<j$ and $l$, so it is possible to have $\gk(X,Y)=0$ but $\hat{g}(X,Y)=g(TX,TY)\ne 0$.  Thus a diagonal metric $g=\left(x_1,\dots,x_s,x_{s+1,0},\dots,x_{2s-1,t}\right)_{g_b}$ is not in general of the form $\hat{g}=\left(y_1,\dots,y_s,y_{s+1,0},\dots,y_{2s-1,t}\right)_{\gk}$ when $s\geq 3$.   
\end{remark}

\section{Einstein metrics}\label{E-sec}  

In this last section, we study the existence of Einstein metrics on aligned homogeneous spaces with $s=2$.  The case when $G_1=G_2$ and $K$ is diagonally embedded, i.e., $M=H\times H/\Delta K$ for some homogeneous space $H/K$, has already been studied in \cite{HHK}.  On the other hand, the case of diagonal metrics on $M=G_1\times G_2/K$ satisfying that $x_{3,0}=\dots=x_{3,t}$ was studied in \cite{Es2}.  We therefore focus here on the case when $K$ is neither simple nor abelian.   

In view of Lemma \ref{gz1z2}, we can consider from now on  the standard metric $\gk$ as a background metric to study diagonal metrics.  We denote by $\mca^{diag}$ the space of all metrics of the form 
\begin{equation}\label{diag}
g=(x_1,x_2,x_{3,0},\dots,x_{3,t})_{\gk}, \qquad x_1,x_2,x_{3,0},\dots,x_{3,t}>0.
\end{equation}
The following formulas follow from Proposition \ref{ricggals2}.  Recall from \eqref{Cchiil} the definition of the Casimir operators $\cas_{\chi_{i,l}}$.  

\begin{proposition}\label{ric}
The Ricci curvature of a diagonal metric $g$ as in \eqref{diag} on an aligned homogeneous space $M=G_1\times G_2/K$ with positive constants $c_1,\lambda_1,\dots,\lambda_t$ is given by 
\begin{enumerate}[{\rm (i)}]
\item $\Ricci(g)|_{\pg_1} = 
\tfrac{1}{2x_1} \sum\limits_{l=0}^t\left(1 - \tfrac{(c_1-1)x_{3,l}}{c_1x_1}\right) \cas_{\chi_{1,l}}|_{\pg_1}
+ \tfrac{1}{4x_1}I_{\pg_1}$.  
\item[ ]
\item $\Ricci(g)|_{\pg_2} = 
\tfrac{1}{2x_2}\sum\limits_{l=0}^t \left(1 - \tfrac{x_{3,l}}{c_1x_2} \right) \cas_{\chi_{2,l}}|_{\pg_2} 
+ \tfrac{1}{4x_2}I_{\pg_2} $.
\item[ ]
\item The decomposition $\pg=\pg_1\oplus\pg_2\oplus\pg_3^0\oplus\dots\oplus\pg_3^t$ is $\ricci(g)$-orthogonal.
\item[ ]
\item $\Ricci(g)|_{\pg_3^l}= r_3^lI_{\pg_3^l}$, $\quad l=0,1,\dots,t$, where 
\begin{align*}
r_3^l 
:=& \tfrac{(c_1-1)\lambda_l}{4x_{3,l}}\left(\tfrac{c_1^2}{(c_1-1)^2}-\tfrac{x_{3,l}^2}{x_1^2} 
-\tfrac{x_{3,l}^2}{(c_1-1)^2x_2^2}\right) 
+\tfrac{(c_1-1)x_{3,l}}{4c_1}\left(\tfrac{1}{x_1^2} 
+\tfrac{1}{(c_1-1)x_2^2}\right).
\end{align*} 
\end{enumerate}
\end{proposition}

\begin{remark}\label{x3l}
If $g$ is Einstein, say $\ricci(g)=g$, then by part (iv), $x_{3,0}$ is determined by $c_1$, $x_1$ and $x_2$, and for $l=1,\dots,t$, the number $x_{3,l}$ satisfies a quadratic equation in terms of $c_1$, $\lambda_l$, $x_1$ and $x_2$.  Note that $x_{3,0}$ appears only in the case when $K$ is not semisimple.   
\end{remark}

\begin{proof} 
Since we are considering $g_b=\gk$, i.e., $z_1=z_2=1$, we have that 
$A_3=-\tfrac{1}{c_1-1}=-B_3$ (recall that $c_2=\tfrac{c_1}{c_1-1}$), thus the proposition is a direct application of the formulas given in Proposition \ref{ricggals2}, except for the formula for $r_3^l$, which is obtained as follows: 
\begin{align*}
r_3^l =& \tfrac{\lambda_l}{4B_3x_{3,l}}\left(\tfrac{2x_1^2-x_{3,l}^2}{x_1^2} 
+\tfrac{(2x_2^2-x_{3,l}^2)A_3^2}{x_2^2}  
-\tfrac{1+A_3}{B_3}\left(\tfrac{1}{c_1}+\tfrac{1}{c_2}A_3^3\right)\right) \\ 
&+\tfrac{1}{4B_3x_{3,l}}\left(2\left(\tfrac{1}{c_1}+\tfrac{1}{c_2}A_3^2\right) 
-\tfrac{2x_1^2-x_{3,l}^2}{c_1x_1^2} 
- \tfrac{(2x_2^2-x_{3,l}^2)A_3^2}{c_2x_2^2}\right)\\ 
=& \tfrac{(c_1-1)\lambda_l}{4x_{3,l}}\left(\tfrac{2x_1^2-x_{3,l}^2}{x_1^2} 
+\tfrac{2x_2^2-x_{3,l}^2}{(c_1-1)^2x_2^2}  
-(c_1-2)\left(\tfrac{1}{c_1}-\tfrac{1}{c_1(c_1-1)^2}\right)\right) \\ 
&+\tfrac{c_1-1}{4x_{3,l}}\left(2\left(\tfrac{1}{c_1}+\tfrac{1}{c_1(c_1-1)}\right) 
-\tfrac{2x_1^2-x_{3,l}^2}{c_1x_1^2} 
- \tfrac{2x_2^2-x_{3,l}^2}{c_1(c_1-1)x_2^2}\right)\\ 
%
%
=& \tfrac{(c_1-1)\lambda_l}{4x_{3,l}}\left(\tfrac{2x_1^2-x_{3,l}^2}{x_1^2} 
+\tfrac{2x_2^2-x_{3,l}^2}{(c_1-1)^2x_2^2}  
-\tfrac{(c_1-2)^2}{(c_1-1)^2}\right) 
+\tfrac{c_1-1}{4x_{3,l}}\left(\tfrac{2}{c_1-1} 
-\tfrac{2x_1^2-x_{3,l}^2}{c_1x_1^2} 
- \tfrac{2x_2^2-x_{3,l}^2}{c_1(c_1-1)x_2^2}\right),\\
=& \tfrac{(c_1-1)\lambda_l}{4x_{3,l}}\left(\tfrac{c_1^2}{(c_1-1)^2}-\tfrac{x_{3,l}^2}{x_1^2} 
-\tfrac{x_{3,l}^2}{(c_1-1)^2x_2^2}\right) 
+\tfrac{c_1-1}{4x_{3,l}}\left(\tfrac{x_{3,l}^2}{c_1x_1^2} 
+\tfrac{x_{3,l}^2}{c_1(c_1-1)x_2^2}\right),
\end{align*}   
concluding the proof.  
\end{proof}

It follows from Proposition \ref{ric}, (i) and (ii) and Lemma \ref{casil} that if 
$$
\pg_i=\pg_i^1\oplus\dots\oplus\pg_i^{q_i}, \qquad  i=1,2, 
$$
are any decompositions in $\Ad(K)$-irreducible summands, then 
$$
\ricci(g)(\pg_i^j,\pg_i^k)=0, \qquad\forall j\ne k,
$$
and the Ricci eigenvalues of the metric $g=(x_1,x_2,x_{3,0},\dots,x_{3,t})_{\gk}$ on $\pg_1$ and $\pg_2$ are respectively given by 
\begin{align}
r_1^j=&\tfrac{1}{2x_1} \sum\limits_{l=0}^t\left(1 - \tfrac{(c_1-1)x_{3,l}}{c_1x_1}\right) \kappa_{1,l}^j
+ \tfrac{1}{4x_1}, \qquad j=1,\dots,q_1,  \label{ricij}\\ 
r_2^j=&\tfrac{1}{2x_2} \sum\limits_{l=0}^t\left(1 - \tfrac{x_{3,l}}{c_1x_2}\right) \kappa_{2,l}^j
+ \tfrac{1}{4x_2}, \qquad j=1,\dots,q_2.  \notag
\end{align}
In particular, if $\pg_1$ and $\pg_2$ are $\Ad(K)$-irreducible, then the space $\mca^{diag}$ is Ricci flow invariant.  

\begin{example}\label{sod}
Assume that an aligned space $M=G_1\times G_2/K$ satisfies that $G_1/\pi_1(K)=\SO(d)/K$, where $d=\dim{K}$ and the embedding is via the adjoint representation.  The isotropy representation of $\SO(d)/K$ is given by 
$$
\pg_1=\vg_1\oplus\dots\vg_t\oplus\bigoplus_{l<m}\kg_l\otimes\kg_m, 
$$
where $\sog(d_i)=\kg_i\oplus\vg_i$ is $\Ad(K_i)$-invariant and $\kg_i$ denotes the adjoint representation of $K_i$ for all $i=1,\dots,t$.  In the case when $K_1=\dots=K_t$, the Casimir constants $\kappa_{1,l}^{(l,m)}=\kappa_{1,m}^{(l,m)}$ corresponding to $\pg_1^{(l,m)}=\kg_l\otimes\kg_m$ are all the same for any $l<m$, so if a diagonal metric $g=(x_1,x_2,x_{3,1},\dots,x_{3,t})$ is Einstein, then from the formulas for the Ricci eigenvalues $r_1^{(l,m)}$, $1\leq l<m\leq t$ given in \eqref{ricij} we obtain that all the numbers $x_{3,l}+x_{3,m}$ must coincide.  This implies that $x_{3,1}=\dots=x_{3,t}$ if $t\geq 3$.  Alternatively, for each $\vg_l$, the only nonzero Casimir constant is $\kappa_{1,l}^l=\tau$, which implies that $x_{3,1}=\dots=x_{3,t}$ if $t\geq 2$.  
\end{example}

\subsection{The case $\mca^{diag}=\mca^G$}\label{class1}
We consider in this section aligned homogeneous spaces $M=G_1\times G_2/K$ such that 
\begin{enumerate}[a)]
\item $K$ is neither simple nor abelian (the case when $K$ is either simple or abelian was studied in \cite{Es2}). 

\item $G_1/\pi_1(K)$ is different from $G_2/\pi_2(K)$ and they are both isotropy irreducible (in particular, $\cas_{\chi_i}=\kappa_iI_{\pg_i}$ for some $\kappa_i\in\RR$).    
\end{enumerate}
Thus any $G$-invariant metric on $M$ is necessarily diagonal.  

It follows from \cite[Tables 3,4]{HHK} that the following are the only infinite families in this class.  Note that both $G_i/\pi_i(K)$, $i=1,2$ are Hermitian symmetric spaces in these families.  

\begin{proposition} (Theorem \ref{E-intro}, (i)).
The following aligned homogeneous spaces do not admit $G$-invariant Einstein metrics:
\begin{enumerate}[{\rm (i)}] 
\item $\SO(2m)\times\SU(m+1)/S^1\times\SU(m)$, \quad $m\geq 3$.  

\item $\Spe(m)\times\SU(m+1)/S^1\times\SU(m)$, \quad $m\geq 2$.

\item $\Spe(m)\times\SO(2m)/S^1\times\SU(m)$, \quad $m\geq 3$. 
\end{enumerate}
\end{proposition}

\begin{remark} 
The slope $(p,q)$, $p,q\in\NN$ with which $S^1$ is embedded in $G_1\times G_2$ must satisfy condition \eqref{al3} in order to have that the homogeneous space is aligned (see Example \ref{S1}).  The lowest dimensional space among these families is $M^{14}=\Spe(2)\times\SU(3)/S^1\times\SU(2)$. 
\end{remark}

\begin{proof}
For the space in part (i), using that $a_{11}=\tfrac{m}{2(m-1)}$ and $a_{21}=\tfrac{m}{m+1}$, we obtain  
$$
d_0=1, \quad d_1=m^2-1,  \quad n_1=m(m-1), \quad n_2=2m, \quad 
c_1=\tfrac{3m-1}{2(m-1)}, \quad \lambda_1=\tfrac{m}{3m-1},
$$
and by Lemma \ref{casil}, 
$$
\begin{array}{lll}
\kappa_{1,0}=\tfrac{d_0}{n_1}=\tfrac{1}{m(m-1)}, &&\kappa_{1,1}=\tfrac{(1-a_{11})d_1}{n_1}=\tfrac{(m-2)(m+1)}{2(m-1)m},\\ 
\kappa_{2,0}=\tfrac{d_0}{n_2}=\tfrac{1}{2m}, &&\kappa_{2,1}=\tfrac{(1-a_{21})d_1}{n_2}=\tfrac{m-1}{2m}.
\end{array}
$$
It follows from \eqref{ricij} and Proposition \ref{ric}, (iv) that the Ricci eigenvalues of the metric $g=(x_1,x_2,x_{3,0},x_{3,1})_{\gk}$ are given by
\begin{align*}
r_1=&\tfrac{1}{2x_1}\left(1 - \tfrac{(c_1-1)x_{3,0}}{c_1x_1}\right) \kappa_{1,0} 
+\tfrac{1}{2x_1}\left(1 - \tfrac{(c_1-1)x_{3,1}}{c_1x_1}\right) \kappa_{1,1}
+ \tfrac{1}{4x_1}   \\ 
=&\tfrac{c_1x_1-(c_1-1)x_{3,0}}{2c_1x_1^2} \tfrac{1}{m(m-1)} 
+\tfrac{c_1x_1-(c_1-1)x_{3,1}}{c_1x_1} \tfrac{(m-2)(m+1)}{2(m-1)m}
+ \tfrac{1}{4x_1},   \\
r_2=&\tfrac{1}{2x_2}\left(1 - \tfrac{x_{3,0}}{c_1x_2}\right) \kappa_{2,0} 
+\tfrac{1}{2x_2}\left(1 - \tfrac{x_{3,1}}{c_1x_2}\right) \kappa_{2,1}
+ \tfrac{1}{4x_2}\\
=&\tfrac{c_1x_2-x_{3,0}}{2c_1x_2^2} \tfrac{1}{2m}
+\tfrac{c_1x_2-x_{3,1}}{2c_1x_2^2} \tfrac{m-1}{2m}
+ \tfrac{1}{4x_2}, \\ 
r_3^0=&\tfrac{(c_1-1)x_{3,0}}{4c_1}\left(\tfrac{1}{x_1^2} +\tfrac{1}{(c_1-1)x_2^2}\right),\\
r_3^1=&\tfrac{(c_1-1)m}{4x_{3,1}(3m-1)}\left(\tfrac{c_1^2}{(c_1-1)^2}-\tfrac{x_{3,1}^2}{x_1^2} 
-\tfrac{x_{3,1}^2}{(c_1-1)^2x_2^2}\right) 
+\tfrac{(c_1-1)x_{3,1}}{4c_1}\left(\tfrac{1}{x_1^2} +\tfrac{1}{(c_1-1)x_2^2}\right).
\end{align*}
It is straightforward to check that if $r_2=r_3^1$, then that $x_{3,0}=-\tfrac{p}{q}$, where 
$$
p= (m(3m-1)x_2 - 2 (m^2-1)x_{3,1})^2 x_1^2 +4 (m-1)^2 (2m^2-m-2) x_1^2x_{3,1}^2  +m(m - 2)(m + 1)^2 x_2^2 x_{3,1}^2
$$
and 
$$
q= 4 (m+1)(m^2-1) x_1^2 x_{3,1}.
$$
Since $p,q>0$, this is a contradiction.  

In the case of (ii), we have that $a_{11}=\tfrac{m}{2(m+1)}$ and $a_{21}=\tfrac{m}{m+1}$, so
$$
d_0=1, \quad d_1=m^2-1,  \quad n_1=m(m+1), \quad n_2=2m, \quad 
c_1=\tfrac{3}{2}, \quad \lambda_1=\tfrac{m}{3(m+1)},
$$
and it follows from Lemma \ref{casil} that
$$
\begin{array}{lll}
\kappa_{1,0}=\tfrac{d_0}{n_1}=\tfrac{1}{m(m+1)}, &&\kappa_{1,1}=\tfrac{(1-a_{11})d_1}{n_1}=\tfrac{(m+2)(m-1)}{2(m+1)m},\\ 
\kappa_{2,0}=\tfrac{d_0}{n_2}=\tfrac{1}{2m}, &&\kappa_{2,1}=\tfrac{(1-a_{21})d_1}{n_2}=\tfrac{m-1}{2m}.
\end{array}
$$
We therefore obtain that
$$
\begin{array}{c}
r_1=\frac{6 m (m + 1) x_1 - 2 x_{30} - (m^2+m-2) x_{31} }{12 m (m + 1) x_1^2}, \qquad 
r_2=\frac{3m x_2 - x_{30}  - (m-1) x_{31} }{6m x_2^2 },\\  
r_3^0=\frac{(x_2^2+2x_1^2) x_{30}}{12  x_2^2x_1^2}, \qquad 
r_3^1=\frac{((m+2) x_2^2 + x_1^2 (8m+4) ) x_{31}^2 + 9 m x_1^2 x_2^2}{24 (m+1) x_2^2 x_1^2 x_{31}},
\end{array}
$$
and it is easy to see that condition $r_2=r_3^1$ implies that  
$$
x_{3,0}=-\tfrac{ (3 m x_2 -2 (m+1) x_{31})^2  x_1^2 + 4(2m^2-m-2) x_1^2 x_{31}^2 + m(m+2) x_2^2 x_{31}^2  }{ 4 (m+1)x_1^2 x_{31}  }<0,
$$
a contradiction.  

Finally, for the space in (iii), $a_{11}=\tfrac{m}{2(m+1)}$ and $a_{21}=\tfrac{m}{2(m-1)}$, thus 
$$
d_0=1, \quad d_1=m^2-1,  \quad n_1=m(m+1), \quad n_2=m(m-1), \quad 
c_1=\tfrac{2m}{m+1}, \quad \lambda_1=\tfrac{1}{4},
$$
and by Lemma \ref{casil},
$$
\begin{array}{lll}
\kappa_{1,0}=\tfrac{d_0}{n_1}=\tfrac{1}{m(m+1)}, &&\kappa_{1,1}=\tfrac{(1-a_{1,1})d_1}{n_1}=\tfrac{(m+2)(m-1)}{2(m+1)m},\\ 
\kappa_{2,0}=\tfrac{d_0}{n_2}=\tfrac{1}{m(m-1)}, &&\kappa_{2,1}=\tfrac{(1-a_{2,1})d_1}{n_2}=\tfrac{(m-2)(m+1)}{2(m-1)m}.
\end{array}
$$
In much the same way as in part (ii), we obtain that 
$$
\begin{array}{c}
r_1=\frac{4 m^2 (m + 1) x_1 -2 (m-1) x_{30}+x_{31}(-m^3+3m-2)}{8 m^2 (m + 1)x_1^2}, \qquad
r_2=\frac{4m^2 (m-1) x_2 -2 (m+1) x_{30} +x_{31}(-m^3+3m+2) }{8 x_2^2 m^2 (m-1)}\\  
r_3^0=\frac{((m-1)x_2^2+(m+1)x_1^2) x_{30}}{8 m x_2^2x_1^2}, \qquad
r_3^1=\frac{((m+2)(m-1)^2 x_2^2 + x_1^2 (3m-2)(m+1)^2 ) x_{31}^2+4 m^3 x_1^2 x_2^2}{16 m x_2^2 x_1^2 x_{31} (m^2-1)},
\end{array}
$$
which implies the following contradiction,
$$
x_{3,0}= -
\tfrac{(2x_2 m^2 - 2 (m^2 - 1) x_{31})^2 x_1^2 + x_{31}^2 (m^2 + 4 m - 8) (m + 1)^2 x_1^2 + m x_2^2 x_{31}^2 (m^3 - 3m + 2))}{4 x_1^2 x_{31} (m + 1)^2}<0,
$$
concluding the proof. 
\end{proof}

In the case when $K=K_1\times K_2$, it follows from \eqref{ricij} and Proposition \ref{ric}, (iv) that the Ricci eigenvalues of the metric $g=(x_1,x_2,x_{3,1},x_{3,2})_{\gk}$ are given by
\begin{align}
r_1=& \tfrac{1}{2x_1}\left(1 - \tfrac{(c_1-1)x_{3,1}}{c_1x_1}\right) \kappa_{1,1} 
+\tfrac{1}{2x_1}\left(1 - \tfrac{(c_1-1)x_{3,2}}{c_1x_1}\right) \kappa_{1,2}
+ \tfrac{1}{4x_1}, \notag  \\ 
r_2=&\tfrac{1}{2x_2}\left(1 - \tfrac{x_{3,1}}{c_1x_2}\right) \kappa_{2,1} 
+\tfrac{1}{2x_2}\left(1 - \tfrac{x_{3,2}}{c_1x_2}\right) \kappa_{2,2} 
+ \tfrac{1}{4x_2}, \label{rimf} \\
r_3^1=&\tfrac{(c_1-1)\lambda_1}{4x_{3,1}}\left(\tfrac{c_1^2}{(c_1-1)^2}-\tfrac{x_{3,1}^2}{x_1^2} 
-\tfrac{x_{3,1}^2}{(c_1-1)^2x_2^2}\right) 
+\tfrac{(c_1-1)x_{3,1}}{4c_1}\left(\tfrac{1}{x_1^2} +\tfrac{1}{(c_1-1)x_2^2}\right), \notag \\
r_3^2=&\tfrac{(c_1-1)\lambda_2}{4x_{3,2}}\left(\tfrac{c_1^2}{(c_1-1)^2}-\tfrac{x_{3,2}^2}{x_1^2} 
-\tfrac{x_{3,2}^2}{(c_1-1)^2x_2^2}\right) 
+\tfrac{(c_1-1)x_{3,2}}{4c_1}\left(\tfrac{1}{x_1^2} +\tfrac{1}{(c_1-1)x_2^2}\right). \notag
\end{align}

According to Table \ref{nonsimple}, the following is the unique aligned space in the class studied in this section such that $\kil_{\pi_i(\kg)}=a_i\kil_{\ggo_i}$ for $i=1,2$. 

\begin{example}\label{su3} (Theorem \ref{E-intro}, (ii)).
For the aligned space $M^{116}=\SU(9)\times\Fg_4/\SU(3)^2$ we have that   
$$
d_1=d_2=8,  \quad n_1=64, \quad n_2=36, \quad 
c_1=\tfrac{13}{9}, \quad \lambda_1=\lambda_2=\tfrac{1}{13}, 
$$
$$
a_{11}=a_{12}=a_1=\tfrac{1}{9}, \quad a_{21}=a_{22}=a_2=\tfrac{1}{4}, \quad \kappa_{1,1}=\kappa_{1,2}=\tfrac{\kappa_1}{2}=\tfrac{1}{9}, \quad \kappa_{2,1}=\kappa_{2,2}=\tfrac{\kappa_2}{2}=\tfrac{1}{6}.
$$
If $g=(x_1,x_2,1,1)_{\gk}$, then by \cite[Proposition 4.9]{Es2}, $g$ is Einstein if and only if $x_2$ is a root of the quartic polynomial  
$$
p= \tfrac{69460352}{387420489}  x^4 - \tfrac{1018001920}{1162261467} x^3 + \tfrac{116920960}{129140163} x^2 - \tfrac{1464320}{4782969} x + \tfrac{30976}{531441}, 
$$
and $x_1$ is determined by $x_2$.  A straightforward computation gives that the discriminant of $p$ is negative, so there exist two roots, giving rise to two Einstein metrics which can be numerically approximated as
$$
g_1=(1.01006,0.88242,1,1)_{\gk}, \qquad g_2=(3.29184,3.61083,1,1)_{\gk}.  
$$ 
On the other hand, for any metric $g=(x_1,x_2,y_1,y_2)_{\gk}$ on $M^{116}$, it follows from \eqref{rimf} that the Ricci eigenvalues are given by
$$
r_1= \tfrac{ 169 x_1 - 8 y_1 - 8 y_2}{468 x_1^2}, \qquad
r_2 = \tfrac{65 x_2 - 9 y_1 - 9 y_2}{156 x_2^2}, 
$$
$$
r_3^1 = \tfrac{(243 x_1^2 + 128 x_2^2) y_1^2 + 169 x_1^2 x_2^2}{1872 y_1 x_1^2 x_2^2},\qquad
r_3^2 = \tfrac{(243 x_1^2 + 128 x_2^2) y_2^2 + 169 x_1^2 x_2^2}{1872 y_2 x_1^2 x_2^2}.  
$$
We can assume that $y_1\ne y_2$ and $x_2=1$.  Thus condition $r_3^1=r_3^2$ holds if and only if 
$$
y_1 =\tfrac{ 169 x_1^2 }{y_2 (243 x_1^2 + 128)},
$$
and so $r_2=r_3^1$ becomes equivalent to  
$$
y_2 = \tfrac{13 x_1   \left(7290 x_1^3 + 3840 x_1  \pm \sqrt{23206257 x_1^6 + 18382464 x_1^4  - 737280 x_1^2  - 2097152}\right)}{85293 x_1^4 + 76032 x_1^2  + 16384 }.
$$
This implies that $r_1=r_2$ if and only if $x_1$ satisfies 
$$
-3645 x_1^3 + 4563 x_1^2 - 2400 x_1  + 1664 =0,
$$
whose unique real solution is 
$$
x_1 = \tfrac{\alpha}{405} - \tfrac{7439}{ 405 \alpha} + \tfrac{169}{405}, \qquad\mbox{where}\quad \alpha:= \left(10864009 + 360 \sqrt{913876206}\right)^{\tfrac{1}{3}}.
$$
We therefore obtain two extra Einstein metrics $g_3$ and $g_4$ on $M^{116}$ with numerical approximation 
$$
g_3=(1.04067, 1, 0.35896, 1.30348)_{\gk}, \qquad g_4=(1.04067, 1,  1.30348, 0.35896)_{\gk}.
$$
The normalized scalar curvature $\scalar_N(g):=(\det_{\gk}{g})^{\tfrac{1}{n}}\scalar(g)$ of these Einstein metrics are approximately given by 
$$
\scalar_N(g_1)=36.35359, \qquad \scalar_N(g_2)=35.52386, \qquad \scalar_N(g_3)=36.09433, 
$$ 
which provides numerical evidence of the fact that they are pairwise non-homothetic.  
\end{example}

Using the classification of isotropy irreducible homogeneous spaces (see \cite[7.106,7.107]{Bss}), we find $11$ more candidates of the form $G_1\times G_2/K_1\times K_2$ to be aligned, which depends on whether the vectors $(a_{11},a_{21})$, $(a_{12},a_{22})$ are collinear or not (recall Definition \ref{alig-def-2}).  The existence of an Einstein metric on each of them can be studied using \eqref{rimf}; indeed, it is equivalent to the existence of a solution $(x_1,x_2,y_1,y_2)$ to the system $r_1=r_2=r_3^1=r_3^2$, which by \cite[Theorem 1.1]{Es2} must satisfy $y_1\ne y_2$.

\begin{table} {\small
$$
\begin{array}{c|c|c}
K & d & G_i's
\\[2mm] \hline \hline \rule{0pt}{14pt}
\Spe(1)^2 & 6 &\Spe(2)^{**}, \quad\SO({\bf 6}), \quad \SO(9)
\\[2mm] \hline  \rule{0pt}{14pt}
\Spe(1)^3 & 9 & \Spe(3), \quad\SO({\bf 9}), \quad \underline{\Eg_6} 
\\[2mm] \hline  \rule{0pt}{14pt}
\Spe(1)^4 & 12 & \Spe(4),  \quad\SO({\bf 12}), \quad \underline{\SO(16)}, \quad \underline{\Eg_8} 
\\[2mm] \hline  \rule{0pt}{14pt}
\SU(3)^2 & 16 &  \Fg_4^*, \quad \SU(9)^*,  \quad \Eg_8  
\\[2mm] \hline  \rule{0pt}{14pt}
\Spe(2)^2 & 20 & \Spe(4)^{**},  \quad \SO(16), \quad \Eg_8, \quad \SO(25) 
\\[2mm] \hline  \rule{0pt}{14pt}
\Spe(1)^7 & 21 & \Spe(7), \quad\SO({\bf 21}), \quad \underline{\Eg_7}, 
\\[2mm] \hline  \rule{0pt}{14pt}
\Spe(1)^8 & 24 & \Spe(8), \quad\SO({\bf 24}), \quad \underline{\Eg_8},  
\\[2mm] \hline  \rule{0pt}{14pt}
\SU(4)^2 & 30 & \SU(16)^*,  \quad \SO(36) 
\\[2mm] \hline  \rule{0pt}{14pt}
\Spe(1)^m & 3m & \Spe(m), \quad \SO({\bf 3m}), \quad m\geq 3
\\[2mm] \hline  \rule{0pt}{14pt}
\Spe(2)^m & 10m & \Spe(2m), \quad \SO(5m), \quad m\geq 3
\\[2mm] \hline\hline
\end{array}
$$}
\caption{Homogeneous spaces $G_i/K$ with $K$ semisimple, non-simple such that $\cas_{\chi_i}=\kappa_iI_{\pg_i}$ for some $\kappa_i\in\RR$ and $\kil_\kg=a_i\kil_{\ggo_i}|_\kg$ for some $a_i\in\RR$ (see \cite[Tables 3,4,5,6,7,10,11]{HHK}).  Only the spaces sharing the isotropy subgroup $K$ with at least one other space have been listed.  For $K$ of the form $\Spe(k)^m$, the $G_i$'s which do not appear in any of the last two families have been underlined, $G_i^*$ means that $G_i/K$ is isotropy irreducible and $G_i^{**}$ means that $G_i/K$ is in addition symmetric.  We denote by $\SO({\bf d})$ the group on which $K^d$ is embedded via the adjoint representation.} \label{nonsimple}
\end{table}

\begin{table} 
{\tiny 
$$
\begin{array}{c|c|c|c|c|c|c|c|c}
M^n=G/K & n & d & n_1 & n_2 & a_1 & a_2 & c_1 & \lambda    
\\[2mm] \hline \hline \rule{0pt}{14pt}
\SO(6)\times\Spe(2)/\Spe(1)^2 & 19 & 6 & 9 & 4 & \tfrac{1}{4} & \tfrac{2}{3} & \tfrac{11}{8} & \tfrac{2}{11}  
\\[2mm]  \hline \rule{0pt}{14pt}
\SO(9)\times\Spe(2)/\Spe(1)^2 & 40 & 6 & 30 & 4 & \tfrac{1}{21} & \tfrac{2}{3} & \tfrac{15}{14} & \tfrac{2}{45}   
\\[2mm]  \hline \rule{0pt}{14pt}
\SO(9)\times\SO(6)/\Spe(1)^2 & 45 & 6 & 30 & 9 & \tfrac{1}{21} & \tfrac{1}{4} & \tfrac{25}{21} & \tfrac{1}{25}  
\\[2mm]  \hline \rule{0pt}{14pt}
\Eg_6\times\Spe(3)/\Spe(1)^3 & 90 & 9 & 69 & 12 & \tfrac{1}{24} & \tfrac{1}{2} & \tfrac{13}{12} & \tfrac{1}{26} 
\\[2mm]  \hline \rule{0pt}{14pt}
\Eg_6\times\SO(9)/\Spe(1)^3 & 105 & 9 & 69 & 27 & \tfrac{1}{24} & \tfrac{1}{7} & \tfrac{31}{24} &  \tfrac{1}{31} 
\\[2mm]  \hline \rule{0pt}{14pt}
\SO(16)\times\Spe(4)/\Spe(1)^4 & 144 & 12 & 108 & 24 & \tfrac{1}{28} & \tfrac{2}{5} & \tfrac{61}{56} & \tfrac{2}{61} 
\\[2mm]  \hline \rule{0pt}{14pt}
\Eg_8\times\Spe(4)/\Spe(1)^4 & 272 & 12 & 236 & 24 & \tfrac{1}{60} & \tfrac{2}{5} & \tfrac{25}{24} & \tfrac{2}{125}   
\\[2mm]  \hline \rule{0pt}{14pt}
\SO(16)\times\SO(12)/\Spe(1)^4 & 174 & 12 & 108 & 54 & \tfrac{1}{28} & \tfrac{1}{10} & \tfrac{19}{14} & \tfrac{1}{38} 
\\[2mm]  \hline \rule{0pt}{14pt}
\Eg_8\times\SO(12)/\Spe(1)^4 & 302 & 12 & 236 & 54 & \tfrac{1}{60} & \tfrac{1}{10} & \tfrac{7}{6} & \tfrac{1}{70}  
\\[2mm]  \hline \rule{0pt}{14pt}
\Eg_8\times\SO(16)/\Spe(1)^4 & 356  & 12 & 236 & 108 & \tfrac{1}{60} & \tfrac{1}{28} & \tfrac{22}{15} & \tfrac{1}{88}
\\[2mm]  \hline \rule{0pt}{14pt}
\SU(9)\times\Fg_4/\SU(3)^2 & 116 & 16 & 64 & 36 &\tfrac{1}{9} & \tfrac{1}{4} & \tfrac{13}{9} & \tfrac{1}{13}  
\\[2mm]  \hline \rule{0pt}{14pt}
\Eg_8\times\Fg_4/\SU(3)^2 & 284 & 16 & 232 & 36 & \tfrac{1}{30} & \tfrac{1}{4} & \tfrac{17}{15} & \tfrac{1}{34} 
\\[2mm]  \hline \rule{0pt}{14pt}
\Eg_8\times\SU(9)/\SU(3)^2 & 312 & 16 & 232 & 64 & \tfrac{1}{30} & \tfrac{1}{9} & \tfrac{13}{10} & \tfrac{1}{39} 
\\[2mm]  \hline \rule{0pt}{14pt}
\SO(16)\times\Spe(4)/\Spe(2)^2 & 136 & 20 & 100 & 16 & \tfrac{3}{28} & \tfrac{3}{5} & \tfrac{33}{28} & \tfrac{1}{11} 
\\[2mm]  \hline \rule{0pt}{14pt}
\Eg_8\times\Spe(4)/\Spe(2)^2 & 264 & 20 & 228 & 16 & \tfrac{1}{20} & \tfrac{3}{5} & \tfrac{13}{12} & \tfrac{3}{65}  
\\[2mm]  \hline \rule{0pt}{14pt}
\SO(25)\times\Spe(4)/\Spe(2)^2 & 316 & 20 & 280 & 16 & \tfrac{3}{115} & \tfrac{3}{5} & \tfrac{24}{23} & \tfrac{1}{40}  
\\[2mm]  \hline \rule{0pt}{14pt}
\Eg_8\times\SO(16)/\Spe(2)^2 & 348 & 20 & 228 & 100 & \tfrac{1}{20} & \tfrac{3}{28} & \tfrac{22}{15} & \tfrac{3}{88} 
\\[2mm]  \hline \rule{0pt}{14pt}
\SO(25)\times\SO(16)/\Spe(2)^2 & 400 & 20 & 280 & 100 & \tfrac{3}{115} & \tfrac{3}{28} & \tfrac{143}{115} & \tfrac{3}{143} 
\\[2mm]  \hline \rule{0pt}{14pt}
\SO(25)\times\Eg_8/\Spe(2)^2 & 528 & 20 & 280 & 228 & \tfrac{3}{115} & \tfrac{1}{20} & \tfrac{35}{23} & \tfrac{3}{175}  
\\[2mm]  \hline \rule{0pt}{14pt}
\Eg_7\times\Spe(7)/\Spe(1)^7 & 217 & 21 & 112 & 84 & \tfrac{1}{9} & \tfrac{1}{4} & \tfrac{13}{9} & \tfrac{1}{13} 
\\[2mm]  \hline \rule{0pt}{14pt}
\SO(21)\times\Eg_7/\Spe(1)^7 & 322 & 21 & 189 & 112 & \tfrac{1}{19} & \tfrac{1}{9} & \tfrac{28}{19} & \tfrac{1}{28}  
\\[2mm]  \hline \rule{0pt}{14pt}
\Eg_8\times\Spe(8)/\Spe(1)^8 & 384 & 24 & 224 & 136 & \tfrac{1}{15} & \tfrac{2}{9} & \tfrac{8}{15} & \tfrac{1}{24}  
\\[2mm]  \hline \rule{0pt}{14pt}
\SO(24)\times\Eg_8/\Spe(1)^8 & 500 & 24 & 252 & 224 & \tfrac{1}{22} & \tfrac{1}{15} & \tfrac{37}{22} & \tfrac{1}{37} 
\\[2mm]  \hline \rule{0pt}{14pt}
\SO(36)\times\SU(16)/\SU(4)^2 & 885 & 30 & 600 & 255 & \tfrac{1}{51} & \tfrac{1}{15} & \tfrac{67}{51} & \tfrac{1}{67}  
\\[2mm]  \hline \rule{0pt}{14pt}
\SO(3m)\times\Spe(m)/\Spe(1)^m &  & 3m & \tfrac{9m(m-1)}{2} & 2m(m-1) & \tfrac{1}{3m-2} & \tfrac{2}{m+1} & \tfrac{7m-3}{2(3m-2)} & \tfrac{2}{7m-3}  
\\[2mm]  \hline \rule{0pt}{14pt}
\SO(5m)\times\Spe(2m)/\Spe(2)^m &  & 10m & \tfrac{25m(m-1)}{2} & 8m(m-1) & \tfrac{3}{5m-2} & \tfrac{3}{2m+1} & \tfrac{7m-1}{5m-2} & \tfrac{3}{7m-1} 
\\[2mm] \hline\hline
\end{array}
$$}
\caption{Class $\cca_{ss}$ in \S\ref{class2}: $2$ infinite families ($m\geq 3$) and $24$ sporadic examples.  Here $\kappa_i=\tfrac{d(1-a_i)}{n_i}$, $c_1=\tfrac{a_1+a_2}{a_2}$ and $\lambda=\tfrac{a_1a_2}{a_1+a_2}$.  All of them admit two Einstein metrics of the form $g=(x_1,x_2,1,\dots,1)_{\gk}$.} \label{all1}
\end{table}

\subsection{The class $\cca_{ss}$}\label{class2}
In this section, we study the existence of Einstein metrics on the class $\cca_{ss}$ of aligned homogeneous spaces $M=G_1\times G_2/K$ defined as follows:
\begin{enumerate}[a)]
\item $K$ is semisimple, non-simple  (the case when $K$ is simple was studied in \cite{Es2} and was called class $\cca$). 

\item $\cas_{\chi_i}=\kappa_iI_{\pg_i}$ for some $\kappa_i\in\RR$, $i=1,2$. 

\item $\kil_{\pi_i(\kg)}=a_i\kil_{\ggo_i}|_\kg$ for some $a_i\in\RR$, $i=1,2$ (i.e., $a_{i1}=\dots=a_{it}=a_i$).   

\item $G_1/\pi_1(K)$ is different from $G_2/\pi_2(K)$ (the case $M=H\times H/\Delta K$ will be studied below in \S\ref{class3} and \S\ref{class4}).  
\end{enumerate}
According to Example \ref{kill-exa}, for any space in $\cca_{ss}$ we have that 
$$
c_1=\tfrac{a_1+a_2}{a_2}, \qquad \lambda_1=\dots=\lambda_t=:\lambda=\tfrac{a_1a_2}{a_1+a_2}, \qquad 
\kappa_i=\tfrac{d(1-a_i)}{n_i}.
$$ 
Using \cite[Tables 3,4,5,6,7,10,11]{HHK}, one can show that the class $\cca_{ss}$ can be explicitly obtained using Table \ref{nonsimple}.  It consists of $2$ infinite families and $24$ sporadic examples, which have been listed in Table \ref{all1}.  

There are in general non-diagonal $G$-invariant metrics on a space in $\cca_{ss}$ since $G_i/\pi_i(K)$ is not required to be isotropy irreducible.  Actually, it follows from Table \ref{nonsimple} that the only space in $\cca_{ss}$ satisfying that $\mca^{diag}=\mca^G$ is $\SU(9)\times\Fg_4/\SU(3)^2$, which has already been analyzed in Example \ref{su3}.   

We now show that the behavior exhibited in Example \ref{sod} is not exceptional.  

\begin{proposition}\label{allx3} (Theorem \ref{E-intro}, (iii)).
If an aligned homogeneous space in $\cca_{ss}$ other than $\SU(9)\times\Fg_4/\SU(3)^2$ admits an Einstein metric of the form $g=(x_1,x_2,x_{3,1},\dots,x_{3,t})_{\gk}$, then $x_{3,1}=\dots=x_{3,t}$.  
\end{proposition} 

\begin{proof}
Let $M=G_1\times G_2/K$ be a space in $\cca_{ss}$ and assume that one of the factors, say $G_1/\pi_1(K)$, satisfies at least one of the following conditions:
\begin{enumerate}[1)]
\item For each $l\in\{ 1,\dots,t\}$ there exists $j_l\in\{ 1,\dots,q_1\}$ such that only the factor $K_l$ of $K$ is acting non-trivially on $\pg_1^{j_l}$ and the corresponding constants $\kappa_{1,l}^{j_l}$'s are all equal (e.g., $\Eg_8/\Spe(1)^4$, see \cite[Table IB, 13, pp.579]{WngZll2}).  

\item The number of factors $t\geq 3$ and for each pair $1\leq l<m\leq t$ there exists $j_{l,m}\in\{ 1,\dots,q_1\}$ such that only $K_l\times K_m$ is acting non-trivially on $\pg_1^{j_{l,m}}$ and the corresponding constants $\kappa_{1,l}^{j_{l,m}}$'s and $\kappa_{1,m}^{j_{l,m}}$'s are all equal (e.g., $\Spe(m)/\Spe(1)^m$, $m\geq 3$, see \cite[Table IA, 2b, pp.577]{WngZll2}).  
\end{enumerate}
It follows from \eqref{ricij} that if condition 1) holds, then the equations $r_1^{j_1}=\dots=r_1^{j_t}$ imply that $x_{3,1}=\dots=x_{3,t}$, and under condition 2), the equality between the Ricci eigenvalues $r_1^{j_{l,m}}$, $l<m$ gives that the numbers $x_{3,l}+x_{3,m}$ are independent of $l,m$, that is, $x_{3,1}=\dots=x_{3,t}$.  

It is straightforward to check using \cite[Tables IA,IB, pp.577-580]{WngZll2} that for each $K$, all the spaces in $\cca_{ss}$, i.e., constructed from Table \ref{nonsimple}, satisfy one of conditions 1) or 2) (see also Example \ref{sod}), with the only exceptions of 
$$
\Spe(2)^2\subset \SO(16), \;\SO(25), \quad \SU(3)^2\subset\Fg_4, \;\SU(9), \quad 
\Spe(1)^2\subset\SO(9), \quad \SU(4)^2\subset\SO(36).
$$
It follows from \cite[Example 2, b), pp.573]{WngZll2} that the isotropy representation $\pg_1$ of the space $\SO(16)/\Spe(2)^2$ satisfies that 
$$
\pg_1\otimes\CC=S^2\CC^4\otimes\Delta^2_0\CC^4  \oplus \Delta^2_0\CC^4\otimes S^2\CC^4,
$$
so $\kappa_{1,1}^1=\kappa_{1,2}^2=\kappa$ and $\kappa_{1,2}^1=\kappa_{1,1}^2=\tau$, where $\cas_{S^2\CC^4}=\kappa I$ and $\cas_{\Delta^2_0\CC^4}=\tau I$ are the Casimir operators as $\Spe(2)$-representations.  According to \eqref{ricij}, if $r_1^1=r_1^2$ then $\kappa x_{3,1}+\tau x_{3,2}=\tau x_{3,1}+\kappa x_{3,2}$, and since $\kappa\ne\tau$ we obtain that $x_{3,1}=x_{3,2}$.  The other exceptions follow in much the same way as above, excepting the space $\SU(9)\times\Fg_4/\SU(3)^2$, which was considered in Example \ref{su3}.  
\end{proof}

It follows from Proposition \ref{allx3} that any Einstein diagonal metric on an aligned space in $\cca_{ss}$ is of the form $g=(x_1,x_2,1,\dots,1)$ up to scaling and so they can be studied as in \cite[Section 4.2]{Es2}.  Indeed, according to \cite[Proposition 4.9]{Es2}, a space in $\cca_{ss}$ admits such an Einstein metric if and only if there exists a real root for certain quartic polynomial $p$ whose coefficients depend on $n_1$, $n_2$, $d$, $a_1$, $a_2$.  

\begin{example}\label{spor}
If $M^{19}=\SO(6)\times\Spe(2)/\Spe(1)^2$, then 
$$ 
d = 6, \quad n_1 = 9, \quad n_2 = 4, \quad 
 c_1 = \tfrac{11}{8},\quad \lambda=\tfrac{2}{11},
\quad
a_1 = \tfrac{1}{4},\quad a_2 = \tfrac{2}{3},  \quad \kappa_1=\kappa_2=\tfrac{1}{2},               
$$
so the discriminant of $p$ is given by 
$$
\Delta=-\tfrac{13132049913661313669263453125}{324518553658426726783156020576256}<0. 
$$
It follows from \cite[Proposition 4.9]{Es2} that this space admits two Einstein metrics of the form $(x_1,x_2,1,1)_{\gk}$. 
\end{example}

\begin{example}\label{sp2m} 
For the space $M=\SO(5m)\times\Spe(2m)/\Spe(2)^m$, $m\geq 3$, we have that 
$$ 
d = 10m, \quad n_1 = \tfrac{25 m (m - 1)}{2}, \quad n_2 = 8 m (m - 1), \quad 
 c_1 = \tfrac{7m-1}{5m-2},\quad \lambda=\tfrac{3}{7m-1},
$$
and
$$
a_1 = \tfrac{3}{5m - 2},\quad a_2 = \tfrac{3}{2m + 1},  \quad \kappa_1=\tfrac{4}{5m-2}, \quad  \kappa_2=\tfrac{5}{4m+2}.               
$$
It is straightforward to see that the discriminant of $p$ is given by 
$$
\Delta=-\tfrac{10000}{(5 m - 2)^{48} } (2m+1)^{12}  (7m-1)^{12} (m-1)^4  (5m+6)^4 (41m-1)^2 q(m),
$$
where
\begin{align*}
q(m) :=& 1225000\, m^{12}  + 3542500 \, m^{11}   - 20541375 \, m^{10}   - 36716025 \,m^{9}  + 174441150 \, m^{8} \\ 
& - 14460660 \, m^{7}  - 441175491 \, m^{6}   + 435037662 \, m^{5}   + 5647239 \, m^{4}  - 122817600 \, m^{3} \\ 
& + 12225600\, m^{2}   + 3110400 \,m  - 518400.
\end{align*}
We rewrite $q(m)$ as 
$$
q(m)= q_1(m) m^8 + q_2(m) m^5 + q_3(m) m^2 + 3110400 \,m  - 518400,  
$$
where
$$
q_1(m):=1225000\, m^{4}  + 3542500 \, m^{3}   - 20541375 \, m^{2}   - 36716025 \,m + 100441150, 
$$
$$
q_2(m) := 74000000 \, m^{3} - 14460660 \, m^{2}  - 441175491 \, m  + 405037662, 
$$
$$
q_3(m) := 30000000  \, m^{3}  + 5647239 \, m^{2}  - 122817600 \, m + 12225600.
$$
It is easy to check that for $i=1,2,3$, $q_i(3), q_i'(3)>0$ and $q_i''(x)>0$ for all $x>3$, which implies that $q_i(m)>0$ for any $m\geq 3$.  Thus $\Delta<0$ and so $M$ admits two Einstein metrics of the form $(x_1,x_2,1,\dots,1)_{\gk}$ for any $m\geq 3$.  
\end{example}

The existence of Einstein metrics of the form $g=(x_1,x_2,1,\dots,1)_{\gk}$ on the remaining infinite family and $23$ sporadic examples in class $\cca_{ss}$ was analyzed using \cite[Proposition 4.9]{Es2} as in the above examples.  We obtained that for all the spaces in Table \ref{all1}, the discriminant of the quartic polynomial $p$ is negative, so they all admit two of such Einstein metrics.  This completes the proof of Theorem \ref{E-intro}, (iii).

\begin{table}
{\small 
$$
\begin{array}{c|c|c|c|c|c|c|c}
H/K & & n & d_1 & d_2 & d_3 & a &      
\\[2mm] \hline \hline \rule{0pt}{14pt}
\SO(2m)/\SO(m)^2 & m\geq 4 & m^2 & \tfrac{m(m-1)}{2} & \tfrac{m(m-1)}{2} & - & \tfrac{m-2}{2(m-1)} & \nexists\; (m\geq 8) 
\\[2mm]  \hline \rule{0pt}{14pt}
\Spe(2m)/\Spe(m)^2 & m\geq 1 & 4m^2 & m(2m+1) & m(2m+1) & - & \tfrac{m+1}{2m+1} & \nexists 
\\[2mm]  \hline \rule{0pt}{14pt}
\SU(m^2)/\SU(m)^2 & m\geq 3 & (m^2-1)^2 & m^2-1 & m^2-1 & - & \tfrac{1}{m^2} & \exists 
\\[2mm]  \hline \rule{0pt}{14pt}
\Eg_6/\SU(3)^3 & - & 54 & 8 & 8 & 8 & \tfrac{1}{4} & \nexists 
\\[2mm]  \hline \rule{0pt}{14pt}
\Eg_7/\Gg_2\times\Spe(3) & - & 98 & 14 & 21 & - & \tfrac{2}{9} & \exists 
\\[2mm]  \hline \rule{0pt}{14pt}
 \Fg_4/\SU(3)^2 & - & 36 & 8 & 8 & - & \tfrac{1}{4} & \exists 
\\[2mm] \hline\hline
\end{array}
$$}
\caption{Isotropy irreducible spaces $H/K$ such that $\kil_\kg= a\kil_\hg|_\kg$ for some $a\in\RR$ and $K$ is not simple (see \cite[Tables 3-7]{HHK}).  Here $\kappa_l=\tfrac{(1-a)d_l}{n}$ and only the first two lines are in addition symmetric spaces.  It is indicated in the last column whether or not $M=H\times H/\Delta K$ admits an Einstein metric $g=(x_1,x_2,x_{3,1},\dots,x_{3,t})$ such that not all $x_{3,l}$'s are equal (condition $x_1=x_2$ turns out to be necessary).} \label{irr}
\end{table}

\subsection{Spaces $H\times H/\Delta K$ with $H/K$ isotropy irreducible}\label{class3}   
Given a $n$-dimensional isotropy irreducible homogeneous space $H/K$, we consider in this section the aligned homogeneous space $M^{2n+d}=H\times H/\Delta K$.  The existence of Einstein metrics of the form $g=(x_1,x_2,1,\dots,1)_{\gk}$ on these spaces was studied in \cite{HHK}.  We are interested here in general diagonal Einstein metrics as in \eqref{diag} for nonsimple $K$.  

We have in this case that
$$
c_1=2, \quad a_{1l}=a_{2l}=:a_l, \quad \lambda_l=\tfrac{a_l}{2}, \quad \kappa_{1,l}=\kappa_{2,l} =\tfrac{(1-a_l)d_l}{n}=:\kappa_l, \qquad\forall l=0,1,\dots,t,
$$
so the Casimir operator of $H/K$ is given by $\cas_\chi=\kappa I$, where $\kappa:=\kappa_0+\dots+\kappa_t$.  According to Proposition \ref{ric} and \eqref{ricij}, the Ricci eigenvalues of a diagonal metric $g=(x_1,x_2,x_{3,0},\dots,x_{3,t})_{\gk}$ are given by  
\begin{equation}\label{riHHK}
\left\{\begin{array}{l} 
r_1=\tfrac{1}{4x_1^2}((2\kappa+1)x_1-\Sigma), \\ 
r_2=\tfrac{1}{4x_2^2}((2\kappa+1)x_2-\Sigma), \qquad\mbox{where}\quad \Sigma:=\sum\limits_{l=0}^t\kappa_lx_{3,l},\\ 
r_3^0=\tfrac{x_{3,0}}{8}\left(\tfrac{1}{x_1^2}+\tfrac{1}{x_2^2}\right), \\ 
r_3^l=\tfrac{a_l}{2x_{3,l}}+\tfrac{(1-a_l)x_{3,l}}{8}\left(\tfrac{1}{x_1^2}+\tfrac{1}{x_2^2}\right), \qquad l=1,\dots,t.
\end{array}\right.
\end{equation}
In particular, $r_1=r_2$ if and only if either $x_1=x_2$ or $\tfrac{1}{x_1}+\tfrac{1}{x_2}=\tfrac{2\kappa+1}{\Sigma}$, and if $x_1=x_2=1$, then $r_1=r_{3,0}$ if and only if $x_{3,0}=\tfrac{2\kappa+1-\Sigma_1}{\kappa_0+1}$, where $\Sigma_1:=\sum\limits_{l=1}^t\kappa_lx_{3,l}$.   

\begin{lemma}\label{nec}
If there exists an Einstein metric $g=(x_1,x_2,x_{3,0},\dots,x_{3,t})_{\gk}$ such that $x_1\ne x_2$, then 
$$
\tfrac{(2\kappa+1)^2}{4a_l(1-a_l)}>8, \qquad \forall l=0,1,\dots,t.
$$
\end{lemma}

\begin{proof}
If $x_1\ne x_2$ and we normalize by $\tfrac{1}{x_1}+\tfrac{1}{x_2}=\tfrac{2\kappa+1}{\Sigma}=1$, then $x_2=\tfrac{x_1}{x_1-1}$, $x_1>1$, $x_1\ne 2$ and $r_1=r_3^l$ for $l=1,\dots,t$ if and only if 
$$
(1-a_1)(x_1^2-2x_1+2)y_l^2 -2(2\kappa+1)(x_1-1)y_l +4a_1x_1^2=0, 
$$
which admits a solution $y_l>0$ if and only if 
$$
\tfrac{(2\kappa+1)^2}{4a_l(1-a_l)}\geq \tfrac{x_1^2(x_1^2-2x_1+2)}{(x_1-1)^2}=:f(x_1).
$$
Since $f(x)\geq 8$ for all $x>1$ and equality holds if and only if $x=2$, we obtain that the existence of a solution $y_l>0$ is equivalent to the condition stated in the lemma.
\end{proof}

We first consider the case when the isotropy irreducible space $H/K$ satisfies that $\kil_\kg=a\kil_\hg$ for some $a>0$ (i.e., $K$ semisimple and $a_1=\dots=a_t=a$).  According to \cite[Tables 3-7]{HHK}), this class consists of the $3$ infinite families and $3$ sporadic examples listed in Table \ref{irr}.  Note that $t\leq 4$ in all cases.  It was shown in \cite{HHK} that all the corresponding spaces $M=H\times H/\Delta K$ admit two Einstein metrics of the form $g=(1,1,y,\dots,y)_{\gk}$, which by \cite[Theorem 1.1]{HHK} it is equivalent to the condition $(2\kappa+1)^2\geq 8a(1-a+\kappa)$, with the only exception of $M=\Spe(2m)\times\Spe(2m)/\Spe(m)^2$.   
 
All diagonal Einstein metrics in this case are necessarily of the form $g=(1,1,y_1,\dots,y_t)$ up to scaling.  Indeed, it is straightforward to check that the only spaces in Table \ref{irr} satisfying the necessary condition given in Lemma \ref{nec} are $\SU(m^2)/\SU(m)^2$, $m\geq 5$, for which it is proved in the following example that there are no diagonal Einstein metrics with $x_1\ne x_2$.   

\begin{example}\label{sum2}
The aligned homogeneous spaces $M=\SU(m^2)\times\SU(m^2)/\SU(m)^2$, $m\geq 3$, have $a=\tfrac{1}{m^2}$, $\kappa=\tfrac{2}{m^2}$ and $\kappa_1=\kappa_2=\unm\kappa$.  Thus the necessary condition in Lemma \ref{nec} holds if and only if $m^4-24 m^2 +48>0$, that is, $m\geq 5$.   If $g=(x_1,x_2,y_1,y_2)_{\gk}$ is Einstein with $x_1\ne x_2$ (in particular, $y_1\ne y_2$) and we consider the normalization $y_1 \kappa_1+y_2 \kappa_2=1$, then by \eqref{riHHK}, 
$$
x_1= \tfrac{ m^2x_2}{(m^2+4)x_2 - m^2},   
$$
and $x_2$ is a root of the quadratic equation $r_1=r_{32}$, whose discriminant is given by  
$$
D(y_2):=-4 m^4 y_2 \big( ( m-1)^2 (m + 1)^2 (m^2 + 4)^2 y_2^3 - m^4 (m^4 + 24) y_2 + 8 m^6 \big).  
$$
On the other hand, it is straightforward to see that $r_1=r_{31}$ if and only if 
$$
y_2 =\tfrac{ (m^4 + 3 m^2 - 4 \pm \sqrt{m^8 + 6 m^6 - 15 m^4 - 8 m^2 + 16}) m^2}{2 (m^2 - 1) (m^2 + 4)},
$$
which satisfy that $D(y_2)<0$, a contradiction.  We therefore obtain that there are no Einstein metrics of the form $g=(x_1,x_2,y_1,y_2)_{\gk}$ with $x_1\ne x_2$ on $M$. 
\end{example}
 
\begin{proposition}\label{a}
If $H/K$ is an isotropy irreducible homogeneous space such that $\kil_\kg=a\kil_\hg$ for some $a>0$, then the aligned homogeneous space $M=H\times H/\Delta K$ admits an Einstein metric of the form $g=(1,1,y_1,\dots,y_t)$ with not all $y_l$'s equal if and only if either $t=2$ and 
\begin{equation}\label{Ea}
(2\kappa+1)^2\geq \tfrac{8a(1-a+\kappa_1)(1-a+\kappa_2)}{1-a}, 
\end{equation}
or $M=\SO(8)\times\SO(8)/\SO(4)^2$ (i.e., $t=4$) and $g=(1,1,y,y,z,z)$, $y\ne z$.  
\end{proposition}

\begin{proof}
It follows from \eqref{riHHK} that a metric $g=(1,1,y_1,\dots,y_t)_{\gk}$ is Einstein if and only if 
$$
2\kappa+1-\kappa_1y_1-\dots-\kappa_ty_t = \tfrac{2a}{y_l}+(1-a)y_l, \qquad\forall l=1,\dots,t.  
$$
This implies that there is a partition $\{1,\dots,t\}=\{l_1,\dots,l_r\}\cup\{ m_1,\dots,m_{t-r}\}$ such that $y_{l_1}=\dots=y_{l_r}=:y$, $y_{m_1}=\dots=y_{m_{t-r}}=:z$.  If $y\ne z$, then $yz=\tfrac{2a}{1-a}$ and both $y$ and $z$ solve the quadratic equation
$$
(1-a+\kappa_{l_1}+\dots+\kappa_{l_r})x^2-(2\kappa+1)x+\tfrac{2a}{1-a}(1-a+\kappa_{m_1}+\dots+\kappa_{m_{t-r}})=0,
$$
which admits two different positive solutions if and only if $t=2$ and \eqref{Ea} holds, or $t\geq 3$, $\kappa_l=\tfrac{\kappa}{t}$ for any $l$ (see Table \ref{irr}, first line with $m=4$ and fourth line) and so the following holds, 
\begin{equation}\label{Ea2}
(2\kappa+1)^2\geq \tfrac{8a((1-a)t+r\kappa)((1-a)t+(t-r)\kappa)}{(1-a)t}. 
\end{equation}
It is easy to check that the only space satisfying \eqref{Ea2} is $\SO(8)/\SO(4)^2$ and only for $s=2$, concluding the proof.   
\end{proof}

\begin{example}\label{so2m}
If $M=\SO(2m)\times\SO(2m)/\SO(m)^2$, $m\geq 4$, then $n = m^2$, $d = m (m - 1)$, $a = \frac{m - 2}{2(m - 1)}$, $\kappa = 1/2$ and $\kappa_1 = \kappa_2 =\unc$, so condition \eqref{Ea} is equivalent to    
$$
4 \ge \tfrac{ (3m-1)^2 (m-2)}{2m(m-1)^2} =: f(m).
$$
Since $f(m) \to \tfrac{9}{2}$ as $m \to \infty$ and $f$ is strictly increasing, we obtain that there exists an Einstein metric on $M$ of the form $g=(1,1,y_1,y_2)_{\gk}$ with $y_1\ne y_2$ if and only if $m=4,5,6,7$.  On the other hand, for $M=\Spe(2m)\times\Spe(2m)/\Spe(m)^2$, $m\geq 1$, condition \eqref{Ea} is equivalent to 
$$
4 \ge  \tfrac{ (m+1) (6m+1)^2 }{2m(2m+1)^2}, 
$$
which never holds.  
\end{example}

\begin{example}\label{sum}
For the spaces $M=\SU(m^2)\times\SU(m^2)/\SU(m)^2$, $m\geq 3$ considered in Example \ref{sum2}, condition \eqref{Ea} reads
$$
\left(\tfrac{4}{m^2}+1\right)^2  \ge  \tfrac{ 8 }{m^2-1},  \quad\mbox{equivalent to} \quad 
m^6 - m^4 +8m^2 -16 \ge  0,
$$
which strictly holds for any $m \ge 3$.  We therefore obtain that there exist two Einstein metrics of the form $g=(1,1,y_1,y_2)_{\gk}$ with $y_1\ne y_2$ on each $M$ (the other one is actually $g=(1,1,y_2,y_1)_{\gk}$). 
\end{example}

The existence of Einstein metrics of the form $g=(1,1,y_1,y_2)_{\gk}$ with $y_1\ne y_2$ has been established using Proposition \ref{a} as in the above examples for the remaining $3$ sporadic spaces in Table \ref{irr}.  The results were added in the last column of the table, completing the proof of Theorem \ref{E-intro}, (iv).  Note that $\Spe(2m)\times\Spe(2m)/\Spe(m)^2$ are the only spaces in Table \ref{irr} which do not admit any diagonal Einstein metric.  

\begin{remark}\label{error}
The aligned spaces $H\times H/\Delta K$, where $H/K$ is one of the irreducible symmetric spaces $\SO(2m)/\SO(m)^2$, $m=4,5,6,7$, therefore admit a fourth Einstein metric of the form $g=(1,1,y_1,y_2)$, $y_1\ne y_2$, besides the three Einstein metrics  given in \cite[Theorem 7.3]{HHK} (two of the form $(1,1,y,y)$ and one which is not diagonal).  On the contrary, if $H/K$ is one of $\SO(2m)/\SO(m)^2$, $m\geq 8$ or $\Spe(2m)/\Spe(m)^2$, then no diagonal Einstein metric with $y_1\ne y_2$ exists.  This, together with \cite[Theorem 7.3]{HHK}, complete the classification of $H\times H$-invariant metrics on spaces $H\times H/\Delta K$, where $H/K$ is an irreducible symmetric space such that $\kil_\kg=a\kil_\hg$ for some $a>0$. 
\end{remark}

\begin{table}
{\tiny 
$$
\begin{array}{c|c|c|c|c|c|c|c|c}
H/K & & n & d_0 & d_1 & d_2 & a_1 & a_2 &     
\\[2mm] \hline \hline \rule{0pt}{14pt}
\SU(p+q)/S^1\times \SU(p)\times\SU(q) & 2\leq p \leq q & 2pq & 1 & p^2-1 & q^2-1 & \tfrac{p}{p+q} & \tfrac{q}{p+q} & \nexists 
\\[2mm]  \hline \rule{0pt}{14pt}
\SU(m+1)/S^1\times \SU(m) & 2\leq m & 2m & 1 & m^2-1 & - & \tfrac{m}{m+1} & - & \nexists 
\\[2mm]  \hline \rule{0pt}{14pt}
\SO(2m)/S^1\times \SU(m) & 3\leq m & m(m-1) & 1 & m^2-1 & - & \tfrac{m}{2(m-1)} & - & \nexists
\\[2mm]  \hline \rule{0pt}{14pt}
\Spe(m)/S^1\times \SU(m) & 2\leq m & m(m+1) & 1 & m^2-1 & - & \tfrac{m}{2(m+1)} & - & \exists 
\\[2mm]  \hline \rule{0pt}{14pt}
\Eg_6/S^1\times \SO(10) & - & 32 & 1 & 45 & - & \tfrac{2}{3} & - & \nexists 
\\[2mm]  \hline \rule{0pt}{14pt}
\Eg_7/S^1\times\Eg_6 & - & 54 & 1 & 78 & - & \tfrac{2}{3} & - & \nexists 
\\[2mm]  \hline \rule{0pt}{14pt}
\SO(p+q)/\SO(p)\times \SO(q) & 2\leq p<q & pq & - & \tfrac{p(p-1)}{2} & \tfrac{q(q-1)}{2} & \tfrac{p-2}{p+q-2} & \tfrac{q-2}{p+q-2} & (*)
\\[2mm]  \hline \rule{0pt}{14pt}
\Spe(p+q)/\Spe(p)\times \Spe(q) & 1\leq p<q & 4pq & - & p(2p+1) & q(2q+1) & \tfrac{p+1}{p+q+1} & \tfrac{q+1}{p+q+1} & (**)
\\[2mm]  \hline \rule{0pt}{14pt}
\Gg_2/\SO(3)\times \SO(3) & - & 8 & - & 3 & 3 & \unm & \tfrac{1}{6} & \exists 
\\[2mm]  \hline \rule{0pt}{14pt}
\Fg_4/\Spe(3)\times \SU(2) & - & 28 & - & 21 & 3 & \tfrac{4}{9} & \tfrac{2}{9} & \exists 
\\[2mm]  \hline \rule{0pt}{14pt}
\Eg_6/\SU(6)\times \SU(2) & - & 40 & - & 35 & 3 & \unm & \tfrac{1}{6} & \exists 
\\[2mm]  \hline \rule{0pt}{14pt}
\Eg_7/\SO(12)\times \SO(3) & - & 64 & - & 66 & 3 & \tfrac{5}{9} & \tfrac{1}{9} & \nexists 
\\[2mm]  \hline \rule{0pt}{14pt}
\Eg_8/\Eg_7\times \SO(3) & - & 112 & - & 133 & 3 & \tfrac{3}{5} & \tfrac{1}{15} & \nexists
\\[2mm] \hline\hline
\end{array}
$$}
\caption{Irreducible symmetric spaces $H/K$ such that $\kil_\kg\ne a\kil_\hg|_\kg$ for all $a\in\RR$  (see \cite[7.102]{Bss}).  We have in all cases that $\kappa=\unm$, $\kappa_0=\tfrac{1}{n}$ and $\kappa_l=\tfrac{(1-a_l)d_l}{n}$ for $l=1,2$ and $\kil_{\kg_l}=a_l\kil_{\hg}|_{\kg_l}$.  In the last column, the existence or not of an Einstein metric of the form $g=(1,1,x_{3,1},\dots,x_{3,t})$ on $M=H\times H/\Delta K$ is indicated (the $x_{3,l}$'s can never be all equal by \cite{HHK}).  (*): there are infinite existence subfamilies (e.g., $q=p+1$, $p\geq 4$) as well as infinite non-existence subfamilies (e.g., $q\geq p^2+2$). (**): there are infinite existence subfamilies (e.g., $q=5p$, $p\geq 2$) as well as infinite non-existence subfamilies (e.g., $q=p+1$). } \label{sym}
\end{table}

\subsection{Spaces $H\times H/\Delta K$ with $H/K$ irreducible symmetric}\label{class4} 
We assume from now on that $H/K$ is an irreducible symmetric space, so $\kappa=\unm$.  In this way, according to  \cite[Theorem 7.3]{HHK}, the corresponding aligned space $M=H\times H/\Delta K$ always admits a non-diagonal Einstein metric, and if in addition $\kil_\kg=a\kil_\hg$, then it also admits two diagonal Einstein metrics if $a<\unm$ and one non-diagonal Einstein metric if $a>\unm$.  As an application of \eqref{riHHK}, we explore here the existence of  diagonal Einstein metrics in the case when $\kil_\kg\ne a\kil_\hg$ for all $a\in\RR$.  These spaces have been listed in Table \ref{sym}, there are $6$ infinite families and $7$ sporadic examples.  

\begin{proposition}\label{HKsym}
Let $H/K$ be an $n$-dimensional irreducible symmetric space and consider $M=H\times H/\Delta K$. 
\begin{enumerate}[{\rm (i)}] 
\item If $K=S^1\times K_1$, then the metric $g=(1,1,y_0,y_1)$ on $M$ is Einstein if and only if $y_0=\tfrac{4n-(n-2)y_1}{2(n+1)}$ and $y_1$ solves the quadratic equation
$$
(3n-2a_1(n+1))y_1^2-4ny_1+4a_1(n+1)=0, 
$$
which admits a positive solution if and only if either $a_1\leq\tfrac{n}{2(n+1)}$ or $\tfrac{n}{n+1}\leq a_1$.  

\item For $K=K_1\times K_2$, the metric $g=(1,1,y_1,y_2)$ is Einstein if and only if 
$$
\left\{\begin{array}{l}
(1-a_1)y_1^2+(\kappa_1y_1+\kappa_2y_2-2)y_1+2a_1=0, \\ 
(1-a_2)y_2^2+(\kappa_1y_1+\kappa_2y_2-2)y_2+2a_2=0. 
\end{array}\right.
$$
\item If $K=S^1\times K_1\times K_2$, then $g=(1,1,y_0,y_1,y_2)$ is Einstein if and only if 
$$
y_0=\tfrac{n(2-\kappa_1y_1-\kappa_2y_2)}{n+1} = \tfrac{2a_1}{y_1}+(1-a_1)y_1 = \tfrac{2a_2}{y_2}+(1-a_2)y_2.
$$
\end{enumerate}
\end{proposition}

\begin{proof}
Using that $\kappa=\unm$, $\kappa_0=\tfrac{1}{n}$ and $\kappa=\kappa_0+\kappa_1$, in part (i), it follows from \eqref{riHHK} that $g$ is Einstein if and only if 
$$
y_0=\tfrac{2-\kappa_1y_1}{\kappa_0+1} =\tfrac{4-(1-2\kappa_0)y_1}{2(\kappa_0+1)} =\tfrac{4n-(n-2)y_1}{2(n+1)},
$$
and $r_1=r_3^1$, which is equivalent to the stated quadratic equation.  

Parts (ii) and (iii) also follow easily from \eqref{riHHK}.  
\end{proof}

\begin{example}\label{supq}
The only spaces as in Proposition \ref{HKsym}, (iii) are 
$$
M=\SU(p+q)\times\SU(p+q)/S^1\times\SU(p)\times\SU(q), \qquad 2\leq p\leq q,
$$
on which $n=2pq$, $a_1=\tfrac{p}{p+q}$, $a_2=\tfrac{q}{p+q}$, $\kappa_0=\tfrac{1}{2pq}$, $\kappa_1=\frac{p^2-1}{2 p (p+q)}$ and $\kappa_2=\frac{q^2-1}{2 q (p+q)}$.  Assume that $g=(1,1,y_0,y_1,y_2)$ is Einstein.  From the third equation in Proposition \ref{HKsym}, (iii) we obtain that either $y_1y_2=2$ or $qy_1=py_2$.  If $y_1y_2=2$, then the first two equations give that 
$$
y_0 = \tfrac{q y_1^2 + 2p}{ (p + q) y_1} = \tfrac{p y_2^2 + 2q}{ (p + q) y_2},  
$$ 
from which follows that $r_{1} = r_3^1$ if and only if 
$$
(2p + q) y_2^2 - 4(p + q)y_2 + 2(p + 2q)=0,
$$
with discriminant $16(p + q)^2 - 8(2p + q)(p + 2q) = -16 pq < 0$, so no solution in this case.  If $qy_1=py_2$, then  
$$
((p + q)^2 - 1) p y_2^2 - 4 p q (p + q) y_2 + 2 q (2 p q + 1)=0,
$$
with discriminant $- 16 (p^2+q^2-1) < 0$.  We therefore obtain that $M$ does not admit any Einstein metric of the form $g=(1,1,y_0,y_1,y_2)$.  
\end{example}

\begin{example}\label{k0k1}
Using the information provided by Table \ref{sym}, it is easy to check that among the spaces of the form $H/S^1\times K_1$, the existence condition in Proposition \ref{HKsym}, (i) only holds for 
$\Spe(m)/S^1 \times \SU(m)$, providing two Einstein metrics of the form $g=(1,1,y_0,y_1)_{\gk}$ which can be explicitly computed for each $m\geq 2$:  
$$
y_0=\tfrac{4n-(n-2)y_1}{2(n+1)}, \qquad y_1 = \tfrac{ 2(m+1)^2 \pm  \sqrt{2 m (m^2 + 3m + 1)}}{ 2 m^2 + 5m + 2}.
$$
\end{example}

Concerning the spaces of the form $H/K_1\times K_2$, it is straightforward to see that the Einstein equations for a metric $g=(1,1,y_1,y_2)_{\gk}$ given in Proposition \ref{HKsym}, (ii) are equivalent to 
$$ 
y_1 = \tfrac{(1 - a_2 + \kappa_2) y_2^2  - 2 y_2 + 2 a_2}{ \kappa_1 y_2},  
$$
and $y_2$ being a positive real root of the quartic polynomial 
\begin{align}
Q(y) =& \beta (\alpha \beta-\kappa_2 \kappa_1) y^4  + 2 ( \beta \kappa_1- 2 \alpha \beta  + \kappa_2 \kappa_1)y^3 \label{polq}
 \\
& + (\alpha(4 a_2 \beta +4) -2 a_2 \kappa_1 \kappa_2 +2 a_1 \kappa_1^2 -4 \kappa_1) y^2 
 - 4 a_2(2 \alpha-\kappa_1) y +4 \alpha a_2^2, \notag
\end{align}
where $\alpha := (1 - a_1 + \kappa_1)$ and $\beta := (1 - a_2 + \kappa_2)$.  On the other hand, since 
$$
-y_1 y_2 = (1 - a_1 + \kappa_1) y_1^2   - 2 y_1 + 2 a_1  = (1 - a_2 + \kappa_2) y_2^2   - 2 y_2 + 2 a_2,
$$
we obtain as necessary conditions that the discriminants of these two quadratic polynomials must be positive, that is,  
\begin{equation}\label{cond12}
 a_1 (1 - a_1 +\kappa_1) <  \tfrac{1}{2},  \qquad  
 a_2 (1 - a_2 +\kappa_2) <  \tfrac{1}{2}.  
\end{equation}

\begin{example}\label{f4}
Consider the aligned space $M=\Fg_4/\Spe(3)\times \SU(2)$.   According to Proposition \ref{HKsym}, (ii), a metric $g=(1,1,y_1,y_2)_{\gk}$ is Einstein if and only if 
$$
\tfrac{36}{35}  y_1^2 + \tfrac{1}{12}  y_1  y _2 - 2 y_1 + \tfrac{8}{9}=0, \qquad 
\tfrac{31}{36} y_2^2 + \tfrac{5}{12} y_1 y_2 - 2 y_2 + \tfrac{4}{9}=0.
$$
The quartic polynomial defined in \eqref{polq} is given up to scaling by 
$$
Q(y)= 403 y^4 - 1494 y^3 + 1811 y^2 - 792 y + 112,
$$
and it is easy to check that its discriminant and other invariants (see \cite[Section 4.2]{Es2}) satisfy that $\Delta>0$ and $R,S<0$.  
This implies that there exists four different real roots, giving rise to four Einstein metrics of the form $g=(1,1,y_1,y_2)_{\gk}$ with $y_1\ne y_2$, where the numerical approximation of $(y_1,y_2)$ is one of 
$$
(0.67352, 0.30511),\quad (0.86492, 1.57672),\quad (1.11732,1.41794),\quad
(1.33983, 0.40740). 
$$
\end{example}

\begin{example}\label{g2}
For the aligned space $M=\Gg_2\times\Gg_2/\SO(4)$, the Einstein equations for a metric $g=(1,1,y_1,y_2)_{\gk}$ given in Proposition \ref{HKsym}, (ii) are 
$$
11 y_1^2 + 5 y_1  y_2 - 32 y_1 + 16=0, \qquad 
\tfrac{55}{3} y_2^2 + 3 y_1 y_2 - 32 y_2 + \tfrac{16}{3}=0.   
$$
In this case, the discriminant $\Delta$ of $q$ is negative, so there exists exactly two real roots.  The corresponding Einstein metrics are numerically approximated by
$$
g_1=(1,1, 0.68110,0.20333)_{\gk}, \qquad g_2=(1,1, 2.10190,0.25337)_{\gk}.
$$
\end{example}

\begin{example}\label{sopsoq}
If $H/K=\SO(p+q)/\SO(p)\times\SO(q)$, $2\leq p<q$, then it is easy to check that the necessary conditions \eqref{cond12} are equivalent to $q\leq p^2+1$.  In particular, for $q\geq p^2+2$ there is no Einstein metric of the form $g=(1,1,y_1,y_2)_{\gk}$.  On the contrary, if $q=p+1$, then  
\begin{align*}
Q(y) =& 3p^2(2p+1) y^4 -2p ( 14 p^2 +3p-5) y^3+2 (25 p^3 -7 p^2 -13 p+3) y^2  \\ 
& -4 (10 p^3 -9 p^2 - 4 p+3) x +4(3p^3-5p^2+p+1),
\end{align*}
and using that $Q(1) = -\tfrac{p^2-4p+2}{2 (2p-1)^3}<0$ for all $p \ge 4$, we obtain that existence holds in this case.  Actually, since for $q=p+k$ we have that 
$$
Q(1) = \tfrac{-2 p^2 + (k^2 - 7k + 14)p + (k^3 - 4k^2 + 7k - 8)}{4(2p + k - 2)^3}, 
$$
for any $k\geq 1$ there exist $p(k)$ such that existence holds for all $p\geq p(k)$.  
\end{example}

\begin{example}\label{sppspq}
Consider $H/K=\Spe(p+q)/\Spe(p)\times\Spe(q)$, $1\leq p<q$.  If we set $q=p+1$, then the quartic polynomial is given by  
\begin{align*}
Q(y) =& (48 p^3 + 84 p^2 + 48 p + 9) x^4 - (224 p^3 + 528 p^2 + 400 p + 96) x^3 \\ 
& + (400 p^3 + 1232 p^2 + 1184 p + 366) x^2 - (320 p^3 + 1248 p^2 + 1504 p + 576) x \\ 
&+ 96 p^3 + 464 p^2 + 704 p + 320, 
\end{align*}
and a straightforward computation gives that two of its invariants are 
$$ 
\Delta = (p + 2)^2 (2 p + 1)^4 (2 p + 3)^2\Delta_1, \qquad 
%
%
S=\tfrac{3 (16 p^4 + 152 p^3 + 168 p^2 + 12 p - 27) (2 p + 1)^2}{1024 (p + 1)^6},
$$
where 
$$
\Delta_1:=\tfrac{ (512 p^9 + 7424 p^8 + 40704 p^7 + 92544 p^6 + 104928 p^5 + 66992 p^4 + 30128 p^3 + 13176 p^2 + 4482 p + 621) }{ 1073741824 (p + 1)^{18}}.
$$
Since $\Delta, S>0$, we obtain that there is no Einstein metric of the form $g=(1,1,y_1,y_2)_{\gk}$ if $q=p+1$.  On the other hand, for $q=5p$ we have that  
\begin{align*}
Q(y) =&
(504 p^3 + 78 p^2 + 3 p) x^4 - (3408 p^3 + 832 p^2 + 44 p) x^3 \\ 
& + (8744 p^3 + 2824 p^2 + 230 p + 2) x^2 - (10080 p^3 + 3936 p^2 + 424 p + 8) x \\ 
&+ 4400 p^3 + 1960 p^2 + 256 p + 8,
\end{align*}
with discriminant $\Delta= -p^2(5p + 1)^2 (2 p + 1)^4\Delta_2$, where 
$$ 
\Delta_2:= \tfrac{6327360 p^9 - 21285376 p^8 + 21803952 p^7 + 888784 p^6 - 7744308 p^5 - 2427512 p^4 - 269643 p^3 - 13715 p^2 - 405 p - 9}{ 128 (6 p + 1)^{18}}.
$$
Using that $\Delta<0$ for all $p \ge 2$ we obtain that existence holds in this case.  
\end{example}

The remaining $3$ sporadic spaces in Table \ref{sym} with $K=K_1\times K_2$ were analyzed in much the same way as above, obtaining the results added in the last column of the table.  This concludes the proof of Theorem \ref{E-intro}, (v).   

Finally, we analyze the existence of diagonal Einstein metrics $g=(x_1,x_2,y_0,\dots,y_t)_{\gk}$ such that $x_1\ne x_2$.  It is straightforward to check that the only spaces $M=H\times H/K$, where $H/K$ is in Table \ref{sym}, which satisfy the necessary condition in Lemma \ref{nec} are those having $H$ equal to $\SU(p+q)$, $\SU(m+1)$, $\SO(p+q)$ and $\Spe(p+q)$.  

\begin{example}\label{sum1}
It follows from \eqref{riHHK} that a metric $g=(x_1,x_2,y_0,1)_{\gk}$, $x_1\ne x_2$ on the aligned space $M=\SU(m+1)\times\SU(m+1)/S^1\times\SU(m)$ is Einstein if and only if 
$$
x_1= \tfrac{ x_2 (\kappa_0 y_0 +\kappa_1 ) }{2x_2 - (\kappa_0 y_0 +\kappa_1)}, \qquad 
x_2 = \tfrac{(\kappa_0 y_0 +\kappa_1 ) \left( (\kappa_0+1) y_0 +\kappa_1 \pm \sqrt{ ((\kappa_0+1) y_0 +\kappa_1 ) ((\kappa_0-1) y_0 +\kappa_1)}\right) }{2 y_0}, 
$$
where $\kappa_0=\tfrac{1}{2m}$, $\kappa_1=\tfrac{m-1}{2m}$ and $y_0=m+1$ or $y_0=\tfrac{m+1}{2m+1}$.  But this implies that  
$$
(\kappa_0-1) y_0 +\kappa_1 = \tfrac{y_0(1-2m) +m-1}{2m}<0,
$$ 
which is a contradiction since this term determines the sign of the expression under the square root in the formula for $x_2$ above.  Therefore $M$ can not admit such an Einstein metric.  
\end{example}

\begin{example}\label{supq2}
Consider the spaces studied in Example \ref{supq} and assume that the metric $g=(x_1,x_2,y_0,y_1,y_2)_{\gk}$ is Einstein, where $x_1\ne x_2$.  The necessary condition in Lemma \ref{nec} is given by $q>6p$.  According to \eqref{riHHK}, $r_1=r_2$ implies that 
 $$
 x_1= \tfrac{ x_2 (\kappa_0 y_0+\kappa_1 y_1 +\kappa_2 y_2) }{2 x_2 - (\kappa_0 y_0+\kappa_1 y_1 +\kappa_2 y_2)},
 $$ 
and if we consider the normalization $\kappa_0y_0  +  \kappa_1y_1+ \kappa_2y_2 =1$, then from $r_1=r_{3}^0$ we obtain that 
$$
y_0=\tfrac{ 2x_2 - 1}{2 x_2^2 - 2 x_2 + 1}.
$$
We now have that $r_1=r_{32}$ if and only if 
$$ 
2((1 - a_2)y_2^2 + a_2) x_2^2 -2 ( (1 - a_2) y_2^2 + y_2) x_2 + y_2( 1+ (1 - a_2)y_2 )=0.
$$
The positivity of the discriminant of this quadratic equation in $x_2$ implies that 
\begin{equation}\label{inter}
\tfrac{p + q - \sqrt{p^2 - 6pq + q^2}}{2p}=\tfrac{ 1 - \sqrt{8a_2^2 - 8 a_2 + 1}}{2 (1 - a_2)} \leq  y_2  \leq \tfrac{ 1 + \sqrt{8a_2^2 - 8 a_2 + 1}}{2 (1 - a_2)} = \tfrac{ p + q + \sqrt{p^2 - 6pq + q^2}}{2p}.
\end{equation}
Since 
$$
r_{3}^1=\tfrac{4 a_1 \kappa_1^2 + T^2 \left( \frac{(2x_2 - 1)^2}{x_2^2} + \frac{1}{x_2^2} \right) (1 - a_1)}{ 8 \kappa_ 1 T},
\qquad\mbox{where}\qquad 
T:=1- \kappa_2 y_2 - \tfrac{\kappa_0 (2x_2 - 1)}{2 x_2^2 - 2 x_2 + 1}, 
$$
and it is straightforward to see that 
$$
1- \kappa_2 y_2 = 1 - \tfrac{y_2 (q^2-1)}{2q(p+q)}<0
$$
for all $y_2$ in the interval \eqref{inter}, we arrive to a contradiction.  Thus these spaces do not admit diagonal Einstein metrics with $x_1\ne x_2$.  
\end{example}

The following result follows from \cite[Theorem 7.3]{HHK}, Lemma \ref{nec}, Example \ref{sum2} and the above two examples. 

\begin{proposition}\label{x1nex2}
The aligned space $M=H\times H/\Delta K$, where $H/K$ is any irreducible symmetric space other than $\SO(p+q)/\SO(p)\times\SO(q)$ or $\Spe(p+q)/\Spe(p)\times\Spe(q)$, never admits a diagonal Einstein metric of the form $g=(x_1,x_2,y_0,y_1,\dots,y_t)_{\gk}$ with $x_1\ne x_2$.  
\end{proposition}

We have obtained only partial results for the two exceptions.  For $\SO(p+q)/\SO(p)\times\SO(q)$, the necessary condition from Lemma \ref{nec} is $q>6p+2$ and we can prove that there are no solutions with $x_1\ne x_2$ if $q=kp$, $k\geq 7$.  On the other hand, in the case $\Spe(p+q)/\Spe(p)\times\Spe(q)$, the necessary condition from Lemma \ref{nec} is $q>6(p+1)$ and we can prove that  if $q=10p$, then there are no solutions with $x_1\ne x_2$.  We do not know whether there is a diagonal Einstein metric of the form $g=(x_1,x_2,y_1,y_2)_{\gk}$ with $x_1\ne x_2$ for some of these spaces.


\begin{thebibliography}{MM}

\bibitem[B]{Bss} {\sc A. Besse}, Einstein manifolds, {\it Ergeb. Math.} {\bf 10} (1987), Springer-Verlag, Berlin-Heidelberg.

\bibitem[BtD]{BrcDck} {\sc T. Br\"ocker, T. tam Dieck}, Representations of compact Lie groups, {\it GTM} {\bf 98} (1985), Springer-Verlag, New-York.

\bibitem[DZ]{DtrZll} {\sc J. D'Atri, W. Ziller}, Naturally reductive metrics and Einstein metrics on compact lie groups, {\it Mem. Amer. Math. Soc.} {\bf 215} (1979).

\bibitem[LL]{stab} {\sc E.A. Lauret, J. Lauret}, The stability of standard homogeneous Einstein manifolds, {\it Math. Z.}, in press (arXiv).   

\bibitem[LW1]{H3}  {\sc J. Lauret, C.E. Will}, Harmonic $3$-forms on compact homogeneous spaces, {\it J. Geom. Anal.} {\bf 33} (2023), 175.  

\bibitem[LW2]{BRF}  {\sc J. Lauret, C.E. Will}, Bismut Ricci flat generalized metrics on compact homogeneous spaces, {\it Transactions Amer. Math. Soc.} {\bf 376} (2023), 7495-7519.  Corrigendum: {\it Transactions Amer. Math. Soc.}, in press, DOI: https://doi.org/10.1090/tran/9179 (arXiv:2301.02335v4).

\bibitem[LW3]{HHK}  {\sc J. Lauret, C.E. Will}, Einstein metrics on homogeneous spaces $M=H\times H/\Delta K$, preprint 2024 (arXiv). 

\bibitem[LW4]{Es2}  {\sc J. Lauret, C.E. Will}, Einstein metrics on aligned homogeneous spaces with two factors, preprint 2024 (arXiv). 

\bibitem[NN]{NklNkn} {\sc Y.Y. Nikolayevsky, Y.G. Nikonorov}, On invariant Riemannian metrics on Ledger-Obata spaces, {\it Manusc. Math.} {\bf 158} (2019) 353-370.  

\bibitem[S]{Spn} {\sc M. Sepanski}, Compact Lie groups, {\it GTM} {\bf 235} (2007), Springer.

\bibitem[WZ]{WngZll2} {\sc M. Y. Wang, W. Ziller}, On normal homogeneous Einstein manifolds, {\it Ann.
Sci. \'Ecole Norm. Sup.} {\bf 18} (1985), 563-633.
\end{thebibliography}
\end{document}